\documentclass[11pt,a4paper,reqno]{amsart}
\usepackage{vmargin,color}
\usepackage[latin1]{inputenc}
\usepackage{amsmath,amsfonts,amsthm,epsfig,graphicx}
\usepackage{caption}

\def\D{\mathcal{D}}

\def\L{\mathcal{L}}

\def\RR{\mathcal{R}}

\def\N{\mathbb N}

\def\R{\mathbb R}

\def\Om{\Omega}

\def\a{\alpha}
\def\b{\beta}
\def\g{\gamma}
\def\de{\delta}
\def\e{\varepsilon}
\def\k{\kappa}
\def\l{\lambda}
\def\s{\sigma}
\def\om{\omega}
\def\vphi{\varphi}

\def\Lip{{\rm Lip}}

\def\Div{{\rm div}\,}

\def\Id{{\rm Id}\,}

\def\weak{\rightharpoonup}

\def\ov{\overline}
\def\pa{\partial}

\def\X{\mathcal{X}}

\def\ac{\mathcal{AC}}
\def\AC{\mathcal{AC}_{\e}}

\def\Err{\mathrm{E}}

\newtheorem{theorem}{Theorem}[section]
\newtheorem{remark}{Remark}[section]

\newtheorem{lemma}[theorem]{Lemma}
\newtheorem{corollary}[theorem]{Corollary}

\numberwithin{equation}{section}
\numberwithin{figure}{section}

\title{Uniform stability \\ in the Euclidean isoperimetric problem \\ for the Allen--Cahn energy}
\author{Francesco Maggi}
\address{Department of Mathematics, The University of Texas at Austin, 2515 Speedway, Stop C1200, Austin TX 78712-1202, United States of America}
\email{maggi@math.utexas.edu}
\author{Daniel Restrepo}
\email{daerestrepomo@math.utexas.edu}

\begin{document}
	
\begin{abstract} {\rm We consider the isoperimetric problem defined on the whole $\R^n$ by the Allen--Cahn energy functional. For non-degenerate double well potentials, we prove sharp quantitative stability inequalities of quadratic type which are uniform in the length scale of the phase transitions. We also derive a rigidity theorem for critical points analogous to the classical Alexandrov's theorem for constant mean curvature boundaries.}
\end{abstract}
		
\maketitle
	
\tableofcontents

\section{Introduction} \subsection{Overview}\label{section overview} In this paper we study the following family of ``Euclidean isoperimetric problems'' on $\R^n$, $n\ge 2$,
\begin{equation}
\label{main problem}
\Psi(\s,m)=\inf \Big\{ \ac_\s(u):  \int_{ \R^n} V(u)=m\,, u \in H^1(\R^n;[0,1])\Big\}\,,\qquad\s\,,m>0\,,
\end{equation}
associated to the Allen-Cahn energy functionals of a non-degenerate double-well potential $W$ (see \eqref{W basic} and \eqref{W normalization} below)
\begin{equation}
\label{AC eps}
\ac_\s(u)=\s\,\int_{ \R^n}|\nabla u|^2+\frac1\s\,\int_{ \R^n}W(u)\,,\qquad\s>0\,.
\end{equation}
We analyze in particular the relation of these problems to the classical Euclidean isoperimetric problem
\begin{equation}
\label{euclidean isoperimetric}
\Psi_{\rm iso}(m)=\inf\Big\{P(E):E\subset\R^n\,,|E|=m\Big\}=n\,\om_n^{1/n}\,m^{(n-1)/n}\,,\qquad m>0\,,
\end{equation}
in the natural regime where the phase transition length scale $\s$ and the volume constraint $m$ satisfy
\begin{equation}
	\label{natural regime}
	0<\s<\e_0\,m^{1/n}\,,
	\end{equation}
for some sufficiently small (dimensionless) constant $\e_0=\e_0(n,W)$. The volume constraint in $\Psi(\s,m)$ is prescribed by means of the potential $V(t)=(\int_0^t\sqrt{W})^{n/(n-1)}$. This specific choice is natural in light of the classical estimate obtained by combining Young's inequality with the $BV$-Sobolev inequality/Euclidean isoperimetry, and showing that, if $u\in H^1(\R^n;[0,1])$, then, for $\Phi(t)=\int_0^t\sqrt{W}$,
\begin{equation}\label{deduction mass constraint}
\ac_\s(u)\geq 2\int_{\R^n} |\nabla u|\sqrt{W(u)}=2\int_{\R^n} |\nabla \Phi(u)| > 2\,n\,\omega_n^{1/n}\,\Big(\int_{\R^n}V(u)\Big)^{(n-1)/n}\,.
\end{equation}
In particular, by our choice of $V$, $\Psi(\s,m)$ is always non-trivial\footnote{Obviously, this is not always true with others choices of $V$. For example, setting $V(t)=t$ in \eqref{main problem}, which is the most common choice in addressing diffuse interface capillarity problems in bounded containers, one has $\Psi(\s,m)=0$ by a simple scaling argument. Among the possible choices that make $\Psi(\s,m)$ non-trivial, our has of course the advantage of appearing naturally in the lower bound \eqref{deduction mass constraint}. For this reason, and in the interest of definiteness and simplicity, we have not considered more general options here.}, with
\begin{equation}
  \label{first lb}
  \Psi(\s,m)>2\,\Psi_{{\rm iso}}(m)\,,\qquad\forall\s,\,m>0\,.
\end{equation}
(The strict sign does not follow from \eqref{deduction mass constraint} alone, but also requires the existence of minimizers in \eqref{deduction mass constraint}.) By combining \eqref{first lb} with a standard construction of competitors for $\Psi(\s,m)$, one sees immediately that
\begin{equation}
\label{convergence of}
\lim_{\s \to 0^+} \Psi(\s,m)=2\,\Psi_{\rm iso}(m)\,,\qquad\forall m>0\,.
\end{equation}
The relation between the Allen--Cahn energy and the perimeter functional is of course a widely explored subject (without trying to be exhaustive, see, for example, \cite{modicamortola,modicaARMA1,sternberg88,luckhausmodica89,hutchtoneCVPDE,rogertoneCVPDE,le2011second,tonewickra2012,dalmasofonsecaleoni15,le2015second,gaspar2020second}), and so is the the relation between the ``volume constrained'' minimization of $\ac_\s$ and relative isoperimetry/capillarity theory in bounded or periodic domains (e.g. \cite{modicacontact,SternZum1,SternZum2,pacardritore,carlencarvalhoespositolebowitzmarra2006,bellettinigellietcCVPDE,LeoniMurray}). The goal of this paper is exploring in detail the proximity of $\Psi(\s,m)$ to the classical Euclidean isoperimetric problem $\Psi_{{\rm iso}}(m)$ in connection with two fundamental properties of the latter:

\medskip

\noindent (i) the validity of the sharp quantitative Euclidean isoperimetric inequality \cite{fuscomaggipratelli}: if $E\subset\R^n$ has finite perimeter $P(E)$ and positive and finite volume (Lebesgue measure) $\L^n(E)$, then
\begin{equation}
  \label{euclidean isoperimetric stability inq}
  C(n)\,\sqrt{\frac{P(E)}{n\,\om_n^{1/n}\,\L^n(E)^{(n-1)/n}}-1}\ge\inf_{x_0\in\R^n}\frac{\L^n(E\Delta B_r(x_0))}{\L^n(E)}\,,\qquad r=\Big(\frac{\L^n(E)}{\om_n}\Big)^{1/n}\,,
\end{equation}
where $\om_n$ denotes the volume of the unit ball in $\R^n$;

\medskip

\noindent (ii) Alexandrov's theorem \cite{alexandrov} (see \cite{delgadinomaggiAPDE} for a distributional version): a bounded open set whose boundary is smooth and has constant mean curvature is a ball; in other words, among bounded sets, the only volume-constrained critical points of the perimeter functional are its (global) volume-constrained minimizers.

\medskip

Concerning property (i), the natural question in relation to $\Psi(\s,m)$ is if a sharp stability estimate similar to \eqref{euclidean isoperimetric stability inq} holds {\it uniformly} with respect to the ratio $\s/m^{1/n}\in(0,\e_0)$ for $\Psi(\s,m)$. Uniformity in $\s/m^{1/n}$ seems indeed a necessary feature for a stability estimate of this kind to be physically meaningful and interesting.

\medskip

Concerning property (ii), we notice that the notion of smooth, volume-constrained critical point of $\Psi(\s,m)$ is that of a non-zero function $u\in C^2(\R^n;[0,1])$ such that the semilinear PDE
\begin{equation}
  \label{semilinear PDE}
  -2\,\s^2\,\Delta u=\s\,\lambda\,V'(u)-W'(u)\qquad\mbox{on $\R^n$}\,,
\end{equation}
holds for a Lagrange multiplier $\lambda\in\R$. The boundedness assumption in Alexandrov's theorem is crucial to avoid examples of non-spherical constant mean curvature boundaries, like cylinders and unduloids. This is directly translated, for solutions of \eqref{semilinear PDE}, into the requirement that $u(x)\to 0$ as $|x|\to\infty$, without which semilinear PDEs like \eqref{semilinear PDE} are known to possess non-radial solutions modeled on the aforementioned examples of unbounded constant mean curvature boundaries, see e.g. \cite{pacardritore}.

\medskip

Under the decay assumption $u(x)\to 0$ as $|x|\to\infty$, and without further constraints on $\s$ and $\lambda$, every solution of \eqref{semilinear PDE} will be radial symmetric thanks to the moving planes method \cite{gidas1981symmetry}. However, even in presence of symmetry, possible solutions to \eqref{semilinear PDE} will have a geometric meaning (and thus a chance of being exhausted by the family of global minimizers of $\Psi(\s,m)$) only if the parameters $\s$ and $\lambda$ are taken in the ``geometric regime'' where $\s\,\lambda$ is small. To explain why we consider such regime geometrically significant, we notice that the Lagrange multiplier $\lambda$ in \eqref{semilinear PDE} has the dimension of an inverse length, which, geometrically, is the dimensionality of curvature. For $\s$ to be the length of a phase transition  around an interface of curvature $\lambda$, it must be that
\begin{equation}
  \label{natural regime sigma ell}
  0<\s\,\lambda<\nu_0\,,
\end{equation}
for some sufficiently small (dimensionless) constant $\nu_0=\nu_0(n,W)$, Notice that since inverse length is ${\rm volume}^{-1/n}=m^{-1/n}$, \eqref{natural regime sigma ell} is compatible with \eqref{natural regime}. We conclude that a natural generalization of Alexandrov's theorem to the Allen--Cahn setting is showing the existence of constants $\e_0$ and $\nu_0$, depending on $n$ and $W$ only, such that, if $u\in C^2(\R^n;[0,1])$ vanishes at infinity and solves \eqref{semilinear PDE} for $\s$ and $\lambda$ as in \eqref{natural regime sigma ell}, then $u$ is a minimizer of $\Psi(\s,m)$ for some value $m$ such that \eqref{natural regime} holds.

\subsection{Statement of the main theorem} We start by setting the following notation and conventions:

\medskip

\noindent {\bf Assumptions on $W$:} The double-well potential $W\in C^{2,1}[0,1]$ satisfies the standard set of non-degeneracy assumptions
	\begin{equation}
	\label{W basic}
	\mbox{$W(0)=W(1)=0$}\,,\quad\mbox{$W>0$ on $(0,1)$}\,,\quad\mbox{$W''(0),W''(1)>0$}\,,
	\end{equation}
	as well as the normalization
	\begin{equation}
	\label{W normalization}
	\int_0^1\sqrt{W}=1\,.
	\end{equation}
	Correspondingly to $W$, we introduce the potential $V$ used in imposing the volume constraint in $\Psi(\s,m)$, by setting
	\begin{equation}
	\label{choice of V}
	V(t)=\Phi(t)^{n/(n-1)}\,,\qquad \Phi(t)=\int_{0}^{t} \sqrt{W}\,,\qquad t\in[0,1]\,.
	\end{equation}
	Notice that both $V$ and $\Phi$ are strictly increasing on $[0,1]$, with $V(1)=\Phi(1)=1$ and $\Phi(t)\approx t^2$ and  $V(t)\approx t^{2n/(n-1)}$ as $t\to 0^+$. All the relevant properties of $W$, $\Phi$ and $V$ are collected in section \ref{subsection W}.

\medskip
	
\noindent {\bf Classes of radial decreasing functions:} We say that $u:\R^n\to\R$ is radial if $u(x)=\zeta(|x|)$ for some $\zeta:[0,\infty)\to\R$, and that $u$ is radial decreasing if, in addition, $\zeta$ is decreasing. We denote by
\[
\RR_0\,,\qquad \RR_0^*\,,
\]
the family of radial decreasing and radial strictly decreasing functions. For the sake of simplicity, when $u$ is radial we shall simply write $u$ in place of $\zeta$, that is, we shall use indifferently $u(x)$ and $u(r)$ to denote the value of $u$ at $x$ with $|x|=r$. Similarly, we shall write $u'$, $u''$, etc. for the radial derivatives of $u$.

\medskip
	
\noindent {\bf Universal constants and rates:} We say that a real number is a {\bf universal constant} it is positive and can be defined in terms of the dimension $n$ and of the double-well potential $W$ only. Following a widely used convention, we will use the latter $C$ for a generically ``large'' universal constant, and $1/C$ for a generically ``small'' one. We will use $\e_0$, $\de_0$, $\nu_0$, $\ell_0$, etc. for small universal constants whose value will be typically ``chosen'' at the end of an argument to make products like $C\,\e_0$ ``sufficiently small''. Finally, given $k\in\N$, we will write ``$f(\e)={\rm O}(\e^k)$ as $\e\to 0^+$'' if there exists a universal constant $C$ such that $|f(\e)|\le C\,\e^k$ for every $\e\in(0,1/C)$; similar definitions are given for ``${\rm O}(t)$ as $t\to\infty$'', etc.
	
\medskip

\begin{theorem}[Main theorem]\label{theorem main}
  If $n\ge 2$ and $W\in C^{2,1}[0,1]$ satisfies \eqref{W basic} and \eqref{W normalization}, then there exists a universal constant $\e_0$ such that, setting,
  \[
  \X(\e_0)=\Big\{(\s,m):0<\s<\e_0\,m^{1/n}\Big\}\,,
  \]
  the following holds:

  \medskip

  \noindent {\bf (i):} for every $(\s,m)\in\X(\e_0)$ there exists a minimizer $u_{\s,m}$ of $\Psi(\s,m)$ such that $u_{\s,m}\in\RR_0^*\cap C^2(\R^n;(0,1))$, every other minimizer of $\Psi(\s,m)$ is obtained from $u_{\s,m}$ by translation, and the Euler--Lagrange equation
  \begin{equation}
    \label{main EL}
    -2\,\s^2\,\Delta u_{\s,m}=\s\,\Lambda(\s,m)\,V'(u_{\s,m})-W'(u_{\s,m})
  \end{equation}
  holds on $\R^n$ for some $\Lambda(\s,m)>0$.

  \medskip

  \noindent {\bf (ii):} $\Psi$ is continuous on $\X(\e_0)$ and
  \begin{eqnarray}\label{main psi strict concave and differentiable}
  &&\mbox{$\Psi(\s,\cdot)$ is strictly concave, strictly increasing,}
  \\\nonumber
  &&\hspace{1.3cm}\mbox{and continuously differentiable on $((\s/\e_0)^n,\infty)$}\,,
  \\\label{main psi strictly decreasing in m}
  &&\mbox{$\Lambda(\s,\cdot)=\frac{\pa\Psi}{\pa m}(\s,\cdot)$ is strictly decreasing and continuous on $((\s/\e_0)^n,\infty)$}\,,
  \\\label{main psi strictly increasing in eps}
  &&\mbox{$\Psi(\cdot,m)$ is strictly increasing on $(0,\e_0\,m^{1/n})$}\,.
  \end{eqnarray}
  Moreover, setting $\e=\s/m^{1/n}$, we have
  \begin{eqnarray}\label{main Psi expansion}
    \frac{\Psi(\s,m)}{m^{(n-1)/n}}&=&2\,n\,\om_n^{1/n}+2\,n\,(n-1)\,\omega_n^{2/n}\,\k_0\,\e+{\rm O}(\e^2)\,,
    \\\label{main Lambda expansion}
    m^{1/n}\,\Lambda(\s,m)&=&2\,(n-1)\,\omega_n^{1/n}+{\rm O}(\e)\,,
  \end{eqnarray}
  as $\e\to 0^+$ with $(\s,m)\in\X(\e_0)$. Here $\kappa_0$ is the universal constant defined by
  \begin{equation}\label{main tau0 tau1}
  \kappa_0=\int_\R \big(V'(\eta)\,\eta'+ W(\eta)\big)\,s\,ds\,,
  \end{equation}
  and $\eta$ is the unique solution to $\eta'=-\sqrt{W(\eta)}$ on $\R$ with $\eta(0)=1/2$.

  \medskip

  \noindent {\bf (iii)-uniform stability:} for every $(\s,m)\in\X(\e_0)$ and $u\in H^1(\R^n;[0,1])$ with $\int_{\R^n}V(u)=m$ we have, for a universal constant $C$,
  \begin{equation}\label{global quantitative estimate}
  C\,\sqrt{\frac{\ac_\s(u)}{\Psi(\s,m)}-1}\ge \inf_{x_0\in\R^n}\,\frac1m
  \int_{\R^n} \big|\Phi(u) -\Phi(T_{x_0}u_{\s,m})\big|^{n/(n-1)}
  \end{equation}
  where $T_{x_0}u_{\s,m}(x)=u_{\s,m}(x-x_0)$, $x\in\R^n$;

  \medskip

  \noindent {\bf (iv)-rigidity of critical points:} there exists a universal constant $\nu_0$ such that, if $\s>0$, $u\in C^2(\R^n;[0,1])$, $u(x)\to 0^+$ as $|x|\to\infty$, and $u$ is a solution of
  \begin{equation}\label{critical point}
    -2\,\s^2\,\Delta u=\s\,\lambda\,V'(u)-W'(u)\qquad\mbox{on $\R^n$}\,,
  \end{equation}
  for a parameter $\lambda$ such that
  \begin{equation}
    \label{alexandrov constraints}
    0<\s\,\lambda<\nu_0\,,
  \end{equation}
  then there exist $x_0\in\R^n$ and $m>0$ such that
  \[
  \s<\e_0\,m^{1/n}\,,\qquad \lambda=\Lambda(\s,m)\,,\qquad u=T_{x_0}u_{\s,m}\,.
  \]
  In particular, $u$ is a minimizer of $\Psi(\s,m)$.
\end{theorem}

\subsection{Relation of Theorem \ref{theorem main}--(iii) to Euclidean isoperimetric stability} We start with some remarks connecting the $(\s,m)$-uniform stability estimate \eqref{global quantitative estimate} to the sharp quantitative Euclidean isoperimetric inequality \eqref{euclidean isoperimetric stability inq}. To this end, it will be convenient to introduce the unit volume problem
\[
\psi(\e)=\Psi(\e,1)=\inf\Big\{\AC(u):\int_{\R^n}V(u)=1\,,u \in H^1(\R^n;[0,1])\Big\}\,,\qquad\e>0\,,
\]
and correspondingly set
\[
\l(\e)=\Lambda(\e,1)=\frac{\pa\Psi}{\pa m}(\e,1)\,,\qquad u_\e=u_{\e,1}\,,\qquad\e>0\,.
\]
Notice that all the information about $\Psi(\s,m)$, $u_{\s,m}$, and $\Lambda(\s,m)$, is contained in $\psi(\e)$, $u_\e$ and $\l(\e)$, thanks to the identities
\[
\frac{\Psi(\s,m)}{m^{(n-1)/n}}=\psi\Big(\frac\s{m^{1/n}}\Big)\,,\qquad m^{1/n}\,\Lambda(\s,m)=\l\Big(\frac{\s}{m^{1/n}}\Big)\,,\qquad u_{\s,m}(x)=u_{\s/m^{1/n}}\Big(\frac{x}{m^{1/n}}\Big)\,,
\]
which are easily proved by a scaling argument (see \eqref{scaling potential} and \eqref{scaling AC}).

\medskip

With this terminology at hand, we start by noticing that the right-hand side of \eqref{global quantitative estimate} is bounded from above by $C(n)$ thanks to the volume constraint $\int_{\R^n}V(u)=m$. Therefore, in proving \eqref{global quantitative estimate} with, say, $(\s,m)=(\e,1)$, one can directly assume that $u$ is a ``low energy competitor for $\psi(\e)$'', in the sense that, for a suitably small universal constant $\ell_0$,
\begin{equation}
  \label{low energy for psieps}
  \AC(u)\le\psi(\e)+\ell_0\,.
\end{equation}
Now, if $u$ is such a low energy competitor $u$, then $f=\Phi(u)$ is $(\ell_0+C\,\e)$-close to be an equality case for $BV$-Sobolev inequality
\begin{equation}
  \label{sobolev BV}
  |Df|(\R^n)\ge n\,\om_n^{1/n}\,,\qquad\mbox{if $\int_{\R^n}|f|^{n/(n-1)}=1$}\,,
\end{equation}
(where $|Df|$ denotes the total variation measure of $f\in BV(\R^n)$, and $|Df|=|\nabla f|\,dx$ if $f\in W^{1,1}(\R^n)$; see \cite{AFP}). Indeed, by an elementary comparison argument, we have
\begin{equation}
\label{basic energy upper bound}
\psi(\e)\le 2\,n\,\om_n^{1/n}+ C\,\e\,,\qquad\forall \e<\e_0\,,
\end{equation}
while \eqref{deduction mass constraint} gives
\begin{equation}
  \label{mm identity}
  \AC(u)-2\,n\,\om_n^{1/n}=
  \int_{\R^n}\Big(\sqrt\e\,|\nabla u|-\sqrt{\frac{W(u)}\e}\Big)^2
  +2\,\Big\{\int_{\R^n}|\nabla[\Phi(u)]|-n\,\om_n^{1/n}\Big\}\,,
\end{equation}
so that the combination of \eqref{low energy for psieps}, \eqref{basic energy upper bound} and \eqref{mm identity} gives
\[
\int_{\R^n}|\nabla[\Phi(u)]|-n\,\om_n^{1/n}\le C\,\big(\ell_0+\e\big)\,,
\]
while, clearly, $\int_{\R^n}f^{n/(n-1)}=\int_{\R^n}V(u)=1$.

\medskip

It is well-known that \eqref{sobolev BV} boils down to the Euclidean isoperimetric inequality if $f=1_E$ is the characteristic function of $E\subset\R^n$, and that equality holds in \eqref{sobolev BV} if and only if $f=a\,1_{B_r(x_0)}$ for some $r,a\ge0$. A sharp quantitative version of \eqref{sobolev BV} was proved in \cite{fuscomaggipratelli} on sets, and then in \cite[Theorem 1.1]{fuscomaggipratelliBV} on functions, and takes the following form: if $n\ge 2$, $f\in BV(\R^n)$, $f\ge 0$, and $\int_{\R^n}f^{n/(n-1)}=1$, then there exists $x_0\in\R^n$ and $r>0$ such that
\begin{equation}
\label{quadratic BV estimate}
C(n)\,\sqrt{|Df|(\R^n)-n\,\om_n^{1/n}}\ge \inf_{x_0\in\R^n\,,r>0}\int_{\R^n}\big|f-a(r)\,1_{B_r}(x_0)\big|^{n/(n-1)}\,,
\end{equation}
where $a(r)$ is defined by $\om_n\,r^n\,a(r)^{n/(n-1)}=1$. The uniform stability estimate \eqref{global quantitative estimate} is thus modeled after \eqref{quadratic BV estimate}, where of course one is working with a different ``deficit'', namely, $\AC(u)-\psi(\e)$ rather than $|Df|(\R^n)-n\,\om_n^{1/n}$ for $f=\Phi(u)$, and with a different ``asymmetry'', namely, the $n/(n-1)$-power of the distance of $\Phi(u)$ from $\Phi$ composed with $u_\e$ rather than with the multiple of the characteristic function of a ball.

\medskip

The key result behind \eqref{global quantitative estimate} is the following {\it Fuglede-type estimate} for $\psi(\e)$ (Theorem \ref{theorem fuglede estimate}): there exist universal constants $\de_0$ and $\e_0$ such that if $\e<\e_0$, $u\in H^1(\R^n;[0,1])$ is a radial (but not necessarily radial decreasing) function, $\int_{\R^n}V(u)=1$ and
\begin{equation}
  \label{fuglede smallness}
  \int_{\R^n}|u-u_\e|^2\le C\,\e\,,\qquad \|u-u_\e\|_{L^\infty(\R^n)}\le\de_0\,,
\end{equation}
then
\begin{equation}
  \label{fuglede conclusion}
  C\,\big(\AC(u)-\psi(\e)\big)\ge\int_{\R^n}\,\e\,|\nabla (u-u_\e)|^2+\frac{(u-u_\e)^2}\e\,.
\end{equation}
Note carefully the restriction here to {\it radial} functions. The right-hand side of \eqref{fuglede conclusion} is the natural $\e$-dependent Hilbert norm associated to $\AC$. By the usual trick based on Young's inequality, \eqref{fuglede conclusion} implies
\begin{equation}
  \label{fuglede conclusion 2}
  C\,\big(\AC(u)-\psi(\e)\big)\ge\int_{\R^n}\,|\nabla[(u-u_\e)^2]|\,,\qquad \mbox{$\forall\,u$ radial, $\int_{\R^n}V(u)=1$}\,,
\end{equation}
and, then, thanks to the $H^1$-Sobolev inequality,
\begin{equation}
  \label{fuglede conclusion 3}
  C\,\big(\AC(u)-\psi(\e)\big)\ge\Big(\int_{\R^n}|u-u_\e|^{2n/(n-1)}\Big)^{(n-1)/n}\,,\qquad \mbox{$\forall\,u$ radial, $\int_{\R^n}V(u)=1$}\,.
\end{equation}
The $\e$-independent stability estimate \eqref{fuglede conclusion 3} (and, {\it a fortiori}, the stronger estimate \eqref{fuglede conclusion 2}) cannot hold on general $u\in H^1(\R^n;[0,1])$ with $\int_{\R^n}V(u)=1$: indeed, if this were the case, one could take in \eqref{fuglede conclusion 3} $u=v_\e$ to be a family of smoothings of $1_E$  for any set $E\subset\R^n$, and then let $\e\to 0^+$, to find a version of \eqref{euclidean isoperimetric stability inq} with linear rather than quadratic rate. However, such linear estimate is well-known to be false, since the rate in \eqref{euclidean isoperimetric stability inq} is saturated, for example, by a family of ellipsoids converging to a ball.

\medskip

We conclude that, on radial functions, one can get estimates, like \eqref{fuglede conclusion}, \eqref{fuglede conclusion 2} and \eqref{fuglede conclusion 3}, that are stronger than what is available for generic functions. We notice in this regard that the validity of stronger stability estimates in presence of symmetries is well-known. For example, in the case of the $BV$-Sobolev inequality, it was proved in \cite[Theorem 3.1]{fuscomaggipratelliBV} that if $f\in BV(\R^n)$ is radial decreasing, $f\ge0$, and $\int_{\R^n}f^{n/(n-1)}=1$, then \eqref{quadratic BV estimate} can be improved to
  \begin{equation}
    \label{quadratic BV estimate radial}
    C(n)\,\Big(|Df|(\R^n)-n\,\om_n^{1/n}\Big)\ge \int_{\R^n}\big|f-a(r)\,1_{B_r}\big|^{n/(n-1)}\,,
  \end{equation}
i.e., the quadratic rate in \eqref{quadratic BV estimate} is refined into a linear rate.

\medskip

We finally notice that \eqref{global quantitative estimate} implies the sharp quantitative form of the Euclidean isoperimetric inequality \eqref{euclidean isoperimetric stability inq} by a standard approximation argument. However, since our proof of \eqref{global quantitative estimate} exploits \eqref{euclidean isoperimetric stability inq}, we are not really providing a new proof of \eqref{euclidean isoperimetric stability inq}. We approach the proof of \eqref{global quantitative estimate} as follows. Adopting the general selection principle strategy of Cicalese and Leonardi \cite{CicaleseLeonardi} we start by deducing \eqref{global quantitative estimate} on radial functions from the Fuglede-type inequality \eqref{fuglede conclusion}. Then we adapt to our setting the quantitative symmetrization method from the proof of \eqref{euclidean isoperimetric stability inq} originally devised in \cite{fuscomaggipratelli}, and thus reduce the proof of \eqref{global quantitative estimate} from the general case to the radial decreasing case. (It is in this reduction step, see in particular Theorem \ref{theorem n symmetric to radial}, that we exploit \eqref{euclidean isoperimetric stability inq}.) In principle, one could have tried to approach \eqref{global quantitative estimate} by working on general functions in both the selection principle and in the Fuglede-type estimate steps. This approach does not seem convenient, however, since it would not save the work needed to implement the selection principle and the Fuglede-type estimates on radial functions, while, at the same time, it would still require the repetition of all the work done in \cite{CicaleseLeonardi} to prove \eqref{euclidean isoperimetric stability inq}.  In other words, an advantage of the approach followed here is that it separates neatly the two stability mechanisms at work in \eqref{global quantitative estimate}, the one related to the relation with the Euclidean isoperimetric problem, and the one specific to optimal transition profile problem (which is entirely captured by working with radial functions).

\subsection{Remarks on the Alexandrov-type result}\label{subsect remarks on alex} We now make some comments on the proof of Theorem \ref{theorem main}-(iv), and explain why this result is closely related to the stability problem addressed in Theorem \ref{theorem main}-(iii).

\medskip

We start noticing that any $u\in C^2(\R^n;[0,1])$, with $u(x)\to0$ as $|x|\to\infty$, and solving \eqref{critical point} for some $\s>0$ and $\lambda\in\R$, will necessarily be a radial function by the moving planes method of \cite{gidas1981symmetry}; see Theorem \ref{theorem ggn plus sp}-(i) below.

\medskip

However, as explained in the overview, there is no clear reason to expect these solutions to have a geometric meaning unless $\s$ and $\lambda$ are in a meaningful geometric relation, which, interpreting $\lambda$ as a curvature and $\s$ as a phase transition length, must take the form of $0<\s\,\lambda<\nu_0$ for some sufficiently small $\nu_0$, see \eqref{natural regime sigma ell}. In Theorem \ref{theorem ggn plus sp}-(ii) we apply to \eqref{critical point} a classical result of Peletier and Serrin \cite{peletier1983uniqueness} about the uniqueness of radial solutions of semilinear PDEs on $\R^n$. Interestingly, the condition $0<\s\,\lambda<\nu_0$, which was introduced because its natural geometric interpretation, plays a crucial role in checking the validity of one of the assumptions of the Peletier--Serrin's uniqueness theorem\footnote{In particular, it is not obvious to us if, outside of the ``geometrically natural'' regime defined by \eqref{natural regime sigma ell}, we should expect uniqueness of radial solutions of \eqref{critical point} with decay at infinity.}.

\medskip

Once symmetry and uniqueness have been addressed by means of classical results like \cite{gidas1981symmetry} and \cite{peletier1983uniqueness}, proving Theorem \ref{theorem main}-(iv) essentially amounts to answering the following question: what is the range of values of $\lambda$ in \eqref{critical point} corresponding to the minimizers $u_{\s,m}$ of $\Psi(\s,m)$ (with $0<\s<\e_0\,m^{1/n}$)? Can we show that every $\lambda$ satisfying $0<\s\,\lambda<\nu_0$ for a sufficiently small universal $\nu_0$ falls in that range?

\medskip

Looking back at \eqref{main EL} we are thus trying to identify the range of $m\mapsto\Lambda(\s,m)=(\pa\Psi/\pa m)(\s,m)$ for $m>(\s/\e_0)^n$, and to show that it contains an interval of the form $(0,\nu_0/\s)$. Such range is indeed proved to be an interval in Theorem \ref{theorem main}-(ii), where we show that $\Lambda(\s,\cdot)$ is decreasing and continuous. The fact that this interval contains a sub-interval of the form $(0,\nu_0/\s)$ is also something that is established in Theorem \ref{theorem main}-(ii), specifically when we analyze the asymptotic behavior of $\Lambda(\s,m)$ as $\s/m^{1/n}\to 0$, see \eqref{main Lambda expansion}. Here we want to stress, however, the role of the {\it continuity} of $\Lambda(\s,\cdot)$, which is of course crucial in showing that $\{\Lambda(\s,m)\}_{m>(\s/\e_0)^n}$ covers the {\it interval} of values between the end-points $\Lambda(\s,+\infty)=0$ and $\Lambda(\s,(\s/\e_0)^n)$. In turn, the Fuglede-type stability estimate \eqref{fuglede conclusion} plays a crucial role in our proof of this continuity property: see step three in the proof of Corollary \ref{corollary optimal energy and lagrange multiplier}.

\medskip

The importance of the Fuglede-type estimate \eqref{fuglede conclusion} in answering both questions of uniform stability and of Alexandrov-type rigidity is the main reason why both problems have been addressed in a same paper.

\subsection{Organization of the paper and proof of Theorem \ref{theorem main}} The existence of minimizers of $\psi(\e)$ (for $\e<\e_0$) and the fact that such minimizers must be radial decreasing (although not necessarily unique up to translations) is established in section \ref{sec existence of minizers}, see Theorem \ref{theorem existence solutions}, through a careful concentration-compactness argument, which exploits both the quantitative stability for the $BV$-Sobolev inequality (in ruling out vanishing) and the specific properties of the Allen-Cahn energy (in ruling out dichotomy). After deducing the validity of the Euler--Lagrange equation (which, because of the range constraint $0\le u\le 1$, holds initially only as a system of variational inequalities), the radial decreasing rearrangement of a minimizer is proved to be {\it strictly} decreasing, so that the Brothers--Ziemer theorem \cite{brothersziemer} can be used to infer that generic minimizers belong to $\RR_0^*$. This existence argument is then adapted to a more general family of perturbations of $\psi(\e)$, which later plays a crucial role in obtaining the main stability estimates \eqref{global quantitative estimate} on radial decreasing functions, see Theorem \ref{theorem compactness for selection principle}. Here the notion of ``critical sequence'' for $\psi(\e_j)$, $\e_j\in(0,\e_0)$, which mixes the notion of ``low-energy sequence'' to that of ``Palais--Smale sequence'', is introduced.

\medskip

In section \ref{section resolution of minimizers} we prove a resolution result for minimizers of $\psi(\e)$ (and, more generally, for the above mentioned notion of critical sequence). In particular, in Theorem \ref{theorem asymptotics of minimizers}, we show, quantitatively in $\e$, that minimizers $u_\e$ of $\psi(\e)$ in $\RR_0$ are close to an {\it Ansatz} which is well-known in the literature (see, e.g., \cite{niethammer1995existence,LeoniMurray}), and is given by
\[
u_\e(x)\approx\eta\Big(\frac{|x|-R_0}\e-\tau_0\Big)\,,\qquad R_0=\frac1{\om_n^{1/n}}\,,\qquad \tau_0=\int_\R \eta'\,V'(\eta)\,s\,ds,
\]
where $\eta$ is the unique solution of $\eta'=-\sqrt{W(\eta)}$ on $\R$ with $\eta(0)=1/2$. Exponential decay rates against this {\it Ansatz} are then obtained in that same theorem. Our analysis is comparably simpler to that of \cite{LeoniMurray}, because our solutions are monotonic decreasing, and, in particular, cannot exhibit the oscillatory behavior at infinity  also described, for positive solution of general semilinear PDEs like \eqref{critical point}, in \cite{Ni}.

\medskip

Section \ref{section strict stability radial profiles} is devoted to the proof of the Fuglede-type estimate \eqref{fuglede conclusion}. This is crucially based on the resolution theorem and on a careful contradiction argument based on the concentration-compactness principle. The Fuglede-type estimate is then shown to imply the uniqueness of radial minimizers (in particular, there is a unique minimizer $u_\e$ of $\psi(\e)$ in $\RR_0$, and every other minimizer of $\psi(\e)$ is obtained from $u_\e$ by translation), the continuity of $\lambda(\e)$ on $\e<\e_0$, and the expansions as $\e\to 0^+$ for $\psi(\e)$ and $\l(\e)$ (which, by scaling, imply \eqref{main Psi expansion} and \eqref{main Lambda expansion}).

\medskip
	
In section \ref{section proof uniform stab} we prove the uniform stability inequality \eqref{global quantitative estimate}. As explained in the remarks above, we first prove \eqref{global quantitative estimate} on radial decreasing functions by means of the selection principle method of Cicalese and Leonardi \cite{CicaleseLeonardi} (this is where Theorem \ref{theorem compactness for selection principle} and the above mentioned notion of critical sequence are used), and then reduce the proof of \eqref{global quantitative estimate} from the general case to the radial decreasing case by adapting to our setting the quantitative symmetrization method introduced in \cite{fuscomaggipratelli} for proving \eqref{euclidean isoperimetric stability inq}.

\medskip

In section \ref{section Alexandrov} we finally prove the Alexandrov-type result along the lines already illustrated in section \ref{subsect remarks on alex}.

\medskip

Finally, in appendix \ref{appendix auxiliary} we collect, for ease of reference, some basic facts and results which are frequently used throughout the paper. Readers are recommended to quickly familiarize themselves with the basic estimates for the potentials $W$, $\Phi$ and $V$ contained therein before entering into the technical aspects of our proofs.

\bigskip
	
\noindent {\bf Acknowledgement:} This work was supported by NSF-DMS RTG 1840314, NSF-DMS FRG 1854344, and NSF-DMS 2000034. Filippo Cagnetti and Matteo Focardi are thanked for collaborating to a preliminary stage of this project. We thank Xavier Cabr\'e and Giovanni Leoni for some insightful comments on a preliminary draft of this work.
	
\section{Existence and radial decreasing symmetry of minimizers}\label{sec existence of minizers} We begin by proving the following existence and symmetry result for minimizers of $\psi(\e)$.

\begin{theorem}\label{theorem existence solutions}
If $n\ge 2$ and $W\in C^{2,1}[0,1]$ satisfies \eqref{W basic} and \eqref{W normalization}, then there exists a universal constant $\e_0$ such that $\psi$ is continuous on $(0,\e_0)$ and, for every $\e<\e_0$, there exist minimizers of $\psi(\e)$. Moreover, if $u_\e$ is a minimizer of $\psi(\e)$ with $\e<\e_0$, then, up to a translation, $u_\e\in \RR_0^*\cap C^{2,\a}_{{\rm loc}}(\R^n)$ for every $\a\in(0,1)$, $0<u_\e<1$ on $\R^n$, and for some $\l\in\R$, $u_\e$ solves
\begin{equation}
	\label{EL classic W V}
	-2\,\e^2\,\Delta u_\e=\e\,\l\,V'(u_\e)-W'(u_\e)\qquad\mbox{on $\R^n$}\,,
\end{equation}
where $\l$ satisfies
\begin{equation}
\label{lambda formula}
\l=\frac{(n-1)}n\,\psi(\e)+\,\frac1n\,\Big\{\frac1\e\int_{\R^n}W(u_\e)-\e\,\int_{\R^n}|\nabla u_\e|^2\Big\}\,.
\end{equation}
Finally, $\l$ obeys the bound
\begin{equation}
\label{lambda estimate eps}
\Big|\l-2\,(n-1)\,\om_n^{1/n}\Big|\le C\,\sqrt\e\,,\qquad\forall \e<\e_0\,,
\end{equation}
so that, in particular, $0<1/C\le \l\le C$ for a universal constant $C$.
\end{theorem}

\begin{proof}
{\it Step one}: We show the existence of universal constants $\ell_0$, $M_0$, and $C$ such that if $\e<\e_0$ and $u\in H^1(\R^n;[0,1])$ satisfies
\begin{equation}
  \label{monotonicity low energy}
\AC(u)\le 2\,n\,\om_n^{1/n}+\ell\,,\qquad\int_{\R^n}V(u)=1\,,
\end{equation}
for some $\ell<\ell_0$, then, up to a translation,
\begin{equation}
\label{conclusion of step one}
\int_{B_{M_0}}V(u)\geq  1-C\,\sqrt{\ell}\,.
\end{equation}
Moreover, in the special case when $u\in\RR_0$, the factor $\sqrt{\ell}$ in \eqref{conclusion of step one} can replaced by $\ell$.

\medskip

Indeed, by applying \eqref{quadratic BV estimate}, to $f=\Phi(u)$, we deduce that, up to a translation of $u$, we have
\begin{equation}\label{non vanishing bound}
		\int_{\R^n}|\Phi(u)-(\om_n^{1/n}\,r)^{1-n}\,1_{B_{r}}|^{n/(n-1)}
\leq C(n)\,\Big(\frac{\mathcal{AC}_{\e}(u)}{2}-n\,\omega_n^{1/n}\Big)^{1/2}\le C\,\sqrt{\ell}\,,
\end{equation}
for suitable $r>0$ (and with $\ell$ in place of $\sqrt{\ell}$ if $u\in\RR_0$). Clearly, \eqref{non vanishing bound} implies
\begin{equation}
  \label{leggere}
  \int_{B_{r}^c}V(u)\le C\,\sqrt{\ell}\,.
\end{equation}
Let us now define $M_0$ by setting
\[
\Phi(1/4)\, [\om_n^{1/n}\,M_0]^{n-1}=1\,.
\]
Clearly, if $r\leq M_0$, then \eqref{leggere} gives
$$
\int_{B_{M_0}^c}V(u)\leq C\,\sqrt{\ell}\,,
$$
and \eqref{conclusion of step one} follows. Assuming by contradiction that  $r>M_0$, by definition of $M_0$ we find
\[
[\om_n^{1/n}\,r]^{1-n}<[\om_n^{1/n}\,M_0]^{1-n}=\Phi(1/4)<\Phi(1/2)\,,
\]
so that
		\begin{eqnarray*}
			\int_{\{u\geq 1/2\}\cap B_{r}} \big|\Phi(1/2)-[\om_n^{1/n}\,r]^{1-n}\big|^{n/(n-1)}\leq
			\int_{\{u\ge 1/2\}}\big|\Phi(u)-[\om_n^{1/n}\,r]^{1-n}\,1_{B_{r}}\big|^{n/(n-1)}\,.
		\end{eqnarray*}
In particular, \eqref{non vanishing bound} and the fact that $\Phi(1/2)-\Phi(1/4)$ is a universal constant imply
		\begin{equation}
		\label{due}
		\big|\{u\geq 1/2\}\cap B_{r}\big|\le C\,\sqrt{\ell_0}\,.
		\end{equation}
At the same time \eqref{W controlla V fin quasi ad uno} gives
		\begin{equation}
		\label{uno}
		\int_{\{u<1/2\}}V(u) \leq C \int_{\{u< 1/2\}} W(u) \leq C\,\e\, \mathcal{AC}_{\e}(u)\le C\,\e\,.
		\end{equation}
By using, in the order, \eqref{uno}, the fact that $V$ is increasing with $V(1)=1$, \eqref{due} and \eqref{leggere}, we conclude
		\begin{eqnarray*}
			1&=&\int_{\R^n}  V(u) \le \int_{\{u\ge 1/2\}} V(u)+
			C\,\e
		\le |\{u\geq 1/2\}\cap B_{r}|+\int_{ B_{r}^c} V(u)+C\,\e
			\\
			&\leq& C\,\big(\sqrt{\ell_0}+\e_0\big)
		\end{eqnarray*}	
which is a contradiction provided we take $\ell_0$ and $\e_0$ small enough.

\medskip

\noindent {\it Step two}: We show the existence of a universal constant $\ell_0$ such that, if $\e<\e_0$ and $\{u_j\}_j$ is a sequence in $H^1(\R^n;[0,1])$ with
\begin{equation}
  \label{end 1}
  \AC(u_j)\le\psi(\e)+\ell_0\,,\qquad \int_{\R^n}V(u_j)=1\,,\qquad\forall j\,,
\end{equation}
then there exists $u\in H^1(\R^n;[0,1])$ such that, up to extracting subsequences and up to translations, $\Phi(u_j)\to\Phi(u)$ in $L^{n/(n-1)}(\R^n)$ and, in particular, $\int_{\R^n}V(u)=1$.

\medskip

We first notice that, by the elementary upper bound \eqref{basic energy upper bound} and by \eqref{end 1}, we have $\AC(u_j)\le C$ for every $j$. Next, we apply the concentration-compactness principle (see appendix \ref{subsection ccc}) to $\{V(u_j)\,dx\}_j$. By \eqref{conclusion of step one} in step one, we find that
\begin{equation}
\label{conclusion of step one from step three}
\int_{B_{M_0}}V(u_j)\geq  1-C\,\sqrt{\ell_0}\,,\qquad\forall j\,.
\end{equation}
This rules out the vanishing case. We consider the case that the dichotomy case occurs. To that end, it will be convenient to notice the validity of the Lipschitz estimate
\begin{equation}\label{lipschitz in epsilon}
  |\AC(u)-\ac_{\nu\,\e}(u)|\le C\,|1-\nu|\,\AC(u)\,,\qquad\forall \nu\ge \frac1{C}\,,\,\,u\in H^1(\R^n;[0,1])\,,
\end{equation}
which is deduced immediately from
\begin{eqnarray*}
  \ac_{\nu\,\e}(u)-\AC(u)=(\nu-1)\,\e\,\int_{\R^n}|\nabla u|^2+\Big(\frac1\nu-1\Big)\frac1\e\int_{\R^n}W(u)\,.
\end{eqnarray*}
By \eqref{conclusion of step one from step three}, if we are in the dichotomy case, then there exists
\begin{equation}
  \label{alfa dicho}
  \a\in\Big(1-C\,\sqrt{\ell_0},1\Big)\,,
\end{equation}
such that for every $\tau\in(0,\a/2)$ we can find $S(\tau)>0$ and $S_j(\tau)\to\infty$ as $j\to\infty$ such that
\begin{eqnarray}\label{dicho 2}
\Big|\a-\int_{B_{S(\tau)}}V(u_j)\Big|<\tau\,,\qquad\Big|(1-\a)-\int_{B_{S_j(\tau)}^c}V(u_j)\Big|<\tau\qquad\forall j\,.
\end{eqnarray}
We now pick a cut-off function\footnote{Notice that $\vphi$ depends on both $j$ and $\tau$. We will not stress this dependency in the notation.} $\vphi$ between $B_{S(\tau)}$ and $B_{S_j(\tau)}$, so that $\vphi\in C^\infty_c(B_{S_j(\tau)})$ with $0\le\vphi\le 1$ and $|\nabla\vphi|\le (S_j(\tau)-S(\tau))^{-1}\le 2\,S_j(\tau)^{-1}$ on $\R^n$, and with $\vphi=1$ on $B_{S(\tau)}$. We notice that  \eqref{dicho 2} and the monotonicity of $V$ entail
\begin{eqnarray}\label{dicho 3}
\Big|\a-\int_{\R^n}V(\vphi\,u_j)\Big|<2\,\tau\,,\qquad\Big|(1-\a)-\int_{\R^n}V\big((1-\vphi)\,u_j\big)\Big|<2\,\tau\qquad\forall j\,.
\end{eqnarray}
We compute that
\begin{eqnarray*}
  \AC(u_j)&=&\AC(\vphi\,u_j)+\AC((1-\vphi)\,u_j)+a_j+b_j\,,
  \\
  a_j&=&2\,\e\,\int_{B_{S_j(\tau)}\setminus B_{S(\tau)}}\vphi\,(1-\vphi)\,|\nabla u_j|^2-u_j^2\,|\nabla\vphi|^2-(1-2\,\vphi)\,u_j\,\nabla u_j\cdot\nabla\vphi\,,
  \\
  b_j&=&\frac1\e\,\int_{B_{S_j(\tau)}\setminus B_{S(\tau)}}W(u_j)-W(\vphi\,u_j)-W\big((1-\vphi)\,u_j\big)\,,
\end{eqnarray*}
where we have taken into account that $\vphi\,(1-\vphi)$ and $\nabla\vphi$ are supported in $B_{S_j(\tau)}\setminus B_{S(\tau)}$, as well as that $W(0)=0$. Let us now set, for $\s\in(0,1)$,
\[
\Gamma_j^+(\tau,\sigma)=\big(B_{S_j(\tau)}\setminus B_{S(\tau)}\big)\cap\{u_j>\s\}\,,\qquad
\Gamma_j^-(\tau,\sigma)=\big(B_{S_j(\tau)}\setminus B_{S(\tau)}\big)\cap\{u_j<\s\}\,.
\]
By \eqref{dicho 2}, we have
\[
V(\s)\,\L^n\big(\Gamma_j^+(\tau,\sigma)\big)\le\int_{B_{S_j(\tau)}\setminus B_{S(\tau)}}V(u_j)\le C\,\tau\,,\qquad\forall j\,.
\]
Taking into account \eqref{V near the wells}, if $\sigma<\de_0$, then we have
\[
\L^n\big(\Gamma_j^+(\tau,\sigma)\big)\le C\,\frac{\tau}{V(\sigma)}\le C\,\frac{\tau}{\sigma^{2\,n/(n-1)}}\,,\qquad\forall j\,.
\]
Provided $\tau\le\tau_*$ for a suitable small universal constant $\tau_*$ we can thus guarantee that
\begin{equation}
  \label{def of tau star}
  \s(\tau):=\tau^{1/[1+(2n/(n-1))]}=\tau^{(n-1)/(3\,n-1)}<\de_0\,,
\end{equation}
and, therefore, that, setting for brevity $\s=\s(\tau)$ as in \eqref{def of tau star},
\[
\L^n\big(\Gamma_j^+(\tau,\sigma)\big)\le C\,\tau^{(n-1)/(3\,n-1)}=C\,\sigma\,,\qquad\forall j\,.
\]
At the same time, by applying \eqref{W second order taylor} with $b=u_j$ and $a=0$ to get
\begin{equation}\label{W second order taylor applied}
\Big|W(u_j)-W''(0)\frac{u_j^2}2\Big|\le C\,u_j^3\le C\,\s\,u_j^2\,,\qquad\mbox{on $\Gamma_j^-(\tau,\s)$}\,,
\end{equation}
and identical inequalities with $\vphi\,u_j$ and $(1-\vphi)\,u_j$ in place of $u_j$, thus finding
\begin{eqnarray*}
  b_j&\ge&\frac{W''(0)}{2\,\e}\,
  \int_{\Gamma_j^-(\tau,\s)}u_j^2-(\vphi\,u_j)^2-\big((1-\vphi)\,u_j\big)^2
  -\frac{C\,\s}\e\,\int_{\Gamma_j^-(\tau,\s)}\,u_j^2-\frac{C}\e\,\L^n\big(\Gamma_j^+(\tau,\sigma)\big)
  \\
  &\ge&\frac{W''(0)}\e\,\int_{\Gamma_j^-(\tau,\s)}\vphi\,(1-\vphi)\,u_j^2
  -\frac{C\,\s}\e\,\int_{\Gamma_j^-(\tau,\s)}\,u_j^2-C\,\frac{\sigma}\e
  \\
  &\ge&-\frac{C\,\s}\e\,\int_{\R^n}\,W(u_j)-C\,\frac{\sigma}\e
  \ge-C\,\frac{\sigma}\e\,.
\end{eqnarray*}
where, in the last line, we have used $W''(0)\ge0$, $\e^{-1}\,\int_{\R^n}\,W(u_j)\le\AC(u_j)\le C$, and the fact that \eqref{W near the wells} and $u_j\le\s\le\de_0$ on $\Gamma_j^-(\tau,\s)$ give us
\begin{equation}
  \label{ancora}
  u_j^2\le C\,W(u_j)\qquad\mbox{on $\Gamma_j^-(\tau,\s)$}\,.
\end{equation}
Similarly, if we discard the first term in the expression for $a_j$ (which is, indeed, non-negative), we find
\begin{eqnarray*}
  a_j&\ge&-2\,\e\,\int_{B_{S_j(\tau)}\setminus B_{S(\tau)}}u_j^2\,|\nabla\vphi|^2+u_j\,|\nabla u_j|\,|\nabla\vphi|
  \\
  &\ge&-C\,\e\,\|\nabla\vphi\|_{C^0(\R^n)}\,\int_{B_{S_j(\tau)}\setminus B_{S(\tau)}}\e\,|\nabla u_j|^2+\frac{u_j^2}\e
  \ge- \frac{C}{S_j(\tau)}\,,
\end{eqnarray*}
where we have used $\|\nabla\vphi\|_{C^0(\R^n)}\le 2\,S_j(\tau)^{-1}$, as well as have noticed that
\begin{eqnarray*}
  \e\,\int_{\R^n}|\nabla u_j|^2&\le&C\,\AC(u_j)\le C\,,
  \\
  \int_{\R^n} u_j^2&\le&\L^n(\{u_j\ge\de_0\})+C\,\int_{\{u_j\le \de_0\}} W(u_j)
  \\
  &\le&C\,\int_{\{u_j\ge \de_0\}}V(u_j)+C\,\e\,\AC(u_j)\le C\,,
\end{eqnarray*}
thanks to $V(t)\ge 1/C$ for $t\in(\de_0,1)$ and to $W(t)\ge t^2/C$ on for $t\in(0,\de_0)$, see \eqref{W near the wells} and \eqref{V limitata dal basso via da zero}. Combining the lower bounds for $a_j$ and $b_j$, we have thus proved
\begin{equation}
\label{dicho 1}
\AC(u_j)\ge \AC(\vphi\,u_j)+\AC((1-\vphi)\,u_j)-C\,\Big(\frac{\s}{\e}+\frac{1}{S_j(\tau)}\Big)\,.
\end{equation}
If we set
\[
m_j=\int_{\R^n}V(\vphi\,u_j)\,,\qquad n_j=\int_{\R^n}V\big((1-\vphi)\,u_j\big)\,,
\]
and define
\begin{equation}
  \label{def of vj wj}
  v_j(x)=(\vphi\,u_j)(m_j^{1/n}\,x)\,,\qquad w_j(x)=\big((1-\vphi)\,u_j\big)(n_j^{1/n}\,x)\,,\qquad x\in\R^n\,,
\end{equation}
then by \eqref{scaling potential} and \eqref{scaling AC} we find
\begin{eqnarray}
  \label{fdr}
  \int_{\R^n}V(v_j)=1\,,\qquad \mathcal{AC}_{\e/m_j^{1/n}}(v_j)=m_j^{(1-n)/n}\,\AC(\vphi\,u_j)\,,
\end{eqnarray}
with analogous identities for $w_j$. By \eqref{dicho 3} and \eqref{lipschitz in epsilon}, and keeping in mind \eqref{alfa dicho}, we find
\begin{eqnarray}\nonumber
  \AC(\vphi\,u_j)&=&m_j^{(n-1)/n}\,\ac_{\e/m_j^{1/n}}(v_j)
  \\\nonumber
  &\ge&(\a-C\,\tau)^{(n-1)/n}\,\Big(1-C\,\big|m_j^{-1/n}-1\big|\Big)\,\AC(v_j)
  \\\label{dicho lb on vphiuj}
  &\ge&(\a-C\,\tau)^{(n-1)/n}\,\big(1-C\,|\a-1|-C\,\tau\big)\,\psi(\e)\,.
\end{eqnarray}
Similarly, taking $\tau$ small enough with respect to $1-\a$, since $\int_{\R^n}V(w_j)=1$ we have that
\begin{eqnarray}\label{dicho lb on 1 minus vphi uj}
  \AC((1-\vphi)\,u_j)=n_j^{(n-1)/n}\,\ac_{\e/n_j^{1/n}}(w_j)\ge((1-\a)-C\,\tau)^{(n-1)/n}\,2\,n\,\om_n^{1/n}\,.
\end{eqnarray}
By combining \eqref{dicho lb on vphiuj} and \eqref{dicho lb on 1 minus vphi uj} with \eqref{dicho 1} we get
\begin{eqnarray*}
\frac{\AC(u_j)}{\psi(\e)}&\ge& (\a-C\,\tau)^{(n-1)/n}\,\big(1-C\,|\a-1|-C\,\tau\big)
\\
&&+\frac{c(n)}{\psi(\e)}\,((1-\a)-C\,\tau)^{(n-1)/n}-\frac{C}{\psi(\e)}\,\Big(\frac{\s}\e+\frac1{S_j(\tau)}\Big)\,.
\end{eqnarray*}
Considering that $\psi(\e)\le C$ for $\e<\e_0$, we let first $j\to\infty$ and then $\tau\to 0^+$ (recall that $\s\to 0^+$ as $\tau\to 0^+$) to find
\begin{eqnarray}\nonumber
1&\ge&\big(1-C\,|\a-1|\big)\,\a^{(n-1)/n}+c(n)\,(1-\a)^{(n-1)/n}
\\\label{care}
&\ge& 1-C\,|\a-1|+c(n)\,(1-\a)^{(n-1)/n}\,.
\end{eqnarray}
Since $1>\a>1-C\,\sqrt{\ell_0}$, by taking $\ell_0$ small enough we can make $\a$ arbitrarily close to $1$ in terms of $n$ and $W$, thus obtaining a contradiction with \eqref{care}. This proves that $\{V(u_j)\,dx\}_j$ is in the compactness case of the concentration--compactness principle. Since \eqref{end 1} implies that $\{\Phi(u_j)\}_j$ has bounded total variation on $\R^n$ and since $V(u_j)=\Phi(u_j)^{n/(n-1)}$ does not concentrate mass at infinity, the compactness statement now follows by standard considerations.

\medskip

\noindent {\it Step three}: Let $\{u_j\}_j$ be a minimizing sequence of $\psi(\e)$, for some $\e<\e_0$. By \eqref{basic energy upper bound} we can assume that for every $j$
\[
\AC(u_j)\le  \psi(\e)+C\,\e\le 2\,n\,\om_n^{1/n}+ C\,\e\,.
\]
We can then apply the compactness statement of step two to deduce the existence of minimizers of $\psi(\e)$. To prove the continuity of $\psi$ on $(0,\e_0)$, let $\e_j\to\e_*\in(0,\e_0)$ as $j\to\infty$, and, for each $\e_j$, let $u_j$ be a minimizer of $\psi(\e_j)$. By \eqref{basic energy upper bound} we can apply step two to $\{u_j\}_j$ and deduce the existence, up to translations and up to extracting subsequences, of $u_*\in H^1(\R^n;[0,1])$ such that $\Phi(u_j)\to \Phi(u_*)$ in $L^{n/(n-1)}(\R^n)$ as $j\to\infty$. If $v\in H^1(\R^n;[0,1])$ with $\int_{\R^n}V(v)=1$, then
\[
\ac_{\e_j}(u_j)\le\ac_{\e_j}(v)
\]
so that, letting $j\to\infty$ and using lower semicontinuity,
\[
\ac_{\e_*}(u_*)\le\liminf_{j\to\infty}\ac_{\e_j}(u_j)\le\lim_{j\to\infty}\ac_{\e_j}(v)=\ac_{\e_*}(v)\,.
\]
Since $\int_{\R^n}V(u_*)=1$, we conclude that $u_*$ is a minimizer of $\psi(\e_*)$; and by plugging $v=u_*$ in the previous chain of inequalities, we find that $\psi(\e_j)\to\psi(\e_*)$ as $j\to\infty$.

\medskip

\noindent {\it Step four}: We now notice that, by the P\'olya--Szeg\"o inequality \cite{brothersziemer}, once there is a minimizer of $\psi(\e)$, there is also a minimizer of $\psi(\e)$ which belongs to $\RR_0$, or, in brief, a {\it radial decrasing minimizer} (more precisely: a radial decreasing minimizer with maximum at $0$).  In this step we prove that every radial decreasing minimizer $u_\e$ of $\psi(\e)$ satisfies $0<u_\e<1$ on $\R^n$ and $u_\e\in C^{2,\a}_{{\rm loc}}(\R^n)$, and that in correspondence of $u_\e$ one can find $\l\in\R$ such that
\begin{equation}
  \label{EL proof z}
  -2\,\e^2\,\Delta u_\e=\e\,\l V'(u_\e)-W'(u_\e)\qquad\mbox{on $\R^n$}\,.
\end{equation}
To begin with, since $u_\e$ is radial decreasing and has finite Dirichlet energy, $u_\e$ is continuous on $\R^n$. In particular, there exist $0\le a<b\le+\infty$ such that
\[
\{u_\e>0\}=B_b\,,\qquad \{u_\e<1\}=\R^n\setminus \ov{B_a}=\Big\{x:|x|>a\Big\}\,.
\]
A standard first variation argument shows the existence of $\l\in\R$ such that
\begin{equation}\label{proof of EL1}
  -2\,\e^2\,\Delta u_\e=\e\,\l\,V'(u_\e)-W'(u_\e)\qquad\mbox{in $\D'(\Om)$, $\Om=B_b\setminus \ov{B_a}$}\,.
\end{equation}
Since \eqref{proof of EL1} implies that $\Delta u_\e$ is bounded in $\Om$, by the Calderon--Zygmund theorem we find that $u_\e\in\Lip_{{\rm loc}}(\Om)$. As a consequence, \eqref{proof of EL1} gives that $-2\,\e^2\,\Delta u_\e=f(u_\e)$ for some $f\in C^1(0,1)$, and thus, by Schauder's theory, $u_\e\in C^{2,\alpha}_{\rm loc}(\Om)$ for every $\alpha \in (0,1)$. We complete this step by showing that $\Om=\R^n$.

\medskip

\noindent {\it Proof that $\Om=\R^n$}: Considering functions of the form $u+t\,\vphi$ with $t\ge 0$ and either $\vphi\in C^\infty_c(\R^n\setminus\ov{B_a})$, $\vphi\ge0$, or $\vphi\in C^\infty_c(B_b)$, $\vphi\le 0$, and then adjusting the volume constraint by a suitable variation localized in $B_b\setminus\ov{B_a}$, we also obtain the validity, in distributional sense, of the inequalities
\begin{eqnarray}
  \label{proof EL out}
  &&-2\,\e^2\,\Delta u_\e\ge\e\,\l\,V'(u_\e)-W'(u_\e)\qquad\mbox{in $\D'(\R^n\setminus \ov{B_a})$}\,.
  \\
  \label{proof EL in}
  &&-2\,\e^2\,\Delta u_\e\le\e\,\l\,V'(u_\e)-W'(u_\e)\qquad\mbox{in $\D'(B_b)$\,.}
\end{eqnarray}
We now stress that, in the rest of the argument, the only property of
\[
f(t)=\e\,\l\,V'(t)-W'(t)\,,\qquad t\in[0,1]\,,
\]
that will be used is the validity of the bound
\begin{equation}
  \label{f bound}
  |f(t)|\le C\,(1+|\l|)\,t\,(1-t)\,,\qquad\forall t\in[0,1]\,.
\end{equation}
This remark will be useful to avoid repetitions when we come to step two of the proof of Theorem \ref{theorem compactness for selection principle}. Notice that \eqref{f bound} indeed holds true thanks to \eqref{W near the wells} and \eqref{V near the wells}, and that in \eqref{f bound} we cannot absorb $|\l|$ into $C$ since we do not know yet that $|\l|$ admits a universal bound (this will actually be proved in step five below).

\medskip

By \eqref{f bound}, \eqref{proof EL out} implies
\begin{equation}
\label{proof EL out radial}
-2\,\e^2\,\Big\{u_\e''+(n-1)\,\frac{u_\e'}r\Big\}\ge -C\,(1+|\l|)\,u_\e\qquad\mbox{in $\D'(a,\infty)$}\,.
\end{equation}
Assuming by contradiction that $b<\infty$, let $r\in(a,b)$, $s$ be such that $(r-s,r+s)\subset(a,b)$, and $\zeta_s$ be the Lipschitz function with $\zeta_s=0$ on $(0,r-s)$, $\zeta_s=1$ on $(r+s,\infty)$, and $\zeta_s'=1/2s$ on $(r-s,r+s)$. Testing \eqref{proof EL out radial} with $-u_\e'\,\zeta_s\ge0$ (which is compactly supported in $(a,\infty)$) we find that
\[
\e^2\,\int_a^\infty\,[(u_\e')^2]'\,\zeta_s+2\,(n-1)\frac{(u_\e')^2}{t}\,\zeta_s\ge C\,(1+|\l|)\, \int_a^\infty\,u_\e\,u_\e'\,\zeta_s
\]
so that, after integration by parts, we obtain
\[
2\,(n-1)\,\e^2\,\int_a^\infty\frac{(u_\e')^2}{t}\,\zeta_s+\frac{C\,(1+|\l|)}{2\,s}\int_{r-s}^{r+s}\frac{u_\e^2}2\ge
\frac{\e^2}{2\,s}\int_{r-s}^{r+s}\,(u_\e')^2\,.
\]
Letting $s\to 0^+$ we obtain
\[
2\,(n-1)\,\e^2\,\int_r^b\frac{(u_\e')^2}{t}+
C\,(1+|\l|)\frac{u_\e(r)^2}2
\ge
\e^2\,u_\e'(r)^2\,.
\]
Finally letting $r\to b^-$ we conclude that $u_\e'(b^-)=0$. This fact, combined with $u_\e(b)=0$ and the uniqueness theorem for the second order ODE \eqref{proof of EL1}, implies that $u_\e=0$ on $(a,b)$, which is in contradiction with the continuity of $u_\e$ if $a>0$, and with $\int_{\R^n}V(u_\e)=1$ if $a=0$. This proves that $b=+\infty$ (and thus that $u_\e>0$ on $\R^n$).

\medskip

The proof of $a=0$ (that is, of $u_\e<1$ on $\R^n$) is analogous. After the change of variables $v=1-u_\e$, we have $v\ge0$, $v'\ge0$, $v=0$ on $(0,a)$, and, thanks to \eqref{proof EL in},
\begin{equation}
\label{proof EL in radial}
-2\,\e^2\,\Big\{v''+(n-1)\,\frac{v'}r\Big\}\ge -C\,(1+|\l|)\,v\qquad\mbox{in $\D'(0,\infty)$}\,.
\end{equation}
Notice that \eqref{proof EL in radial} is identical to \eqref{proof EL out radial}, and that an even reflection by $r=a$ maps the boundary conditions of $v$ into those of $u_\e$: the same argument used for proving $u_\e'(b^-)=0$ will thus show that $v'(a^+)=0$. For the sake of clarity we give some details. We pick $r>a$, introduce a Lipschitz function $\bar{\zeta}_s$ with $\bar\zeta_s=1$ on $(0,r-s)$, $\bar\zeta_s=0$ on $(r+s,\infty)$, and $\bar\zeta_s'=-1/2s$ on $(r-s,r+s)$, and test \eqref{proof EL in radial} with $v'\,\bar\zeta_s\ge0$, to get
\[
-\e^2\,\int_0^\infty\,[(v')^2]'\,\bar\zeta_s+2\,(n-1)\frac{(v')^2}{t}\,\bar\zeta_s\ge -C\,(1+|\l|)\, \int_0^\infty\,v\,v'\,\bar\zeta_s\,.
\]
Integration by parts now gives
\[
-\frac{\e^2}{2s}\,\int_{r-s}^{r+s}(v')^2\,-2\,(n-1)\,\e^2\int_a^{r+s}\frac{(v')^2}{t}\,\bar\zeta_s\ge - \frac{C\,(1+|\l|)}{2s}\, \int_{r-s}^{r+s}\,\frac{v^2}2\,,
\]
so that in the limit $s\to 0^+$, and then $r\to a^+$, we find $v'(a^+)=0$, that is to say, $u_\e'(a^+)=0$. If $a>0$ and thus $u_\e(a)=1$, this, combined with \eqref{proof of EL1}, implies $u_\e=1$ on $\R^n$, a contradiction.

\medskip

\noindent {\it Step five}: Given a radial decreasing minimizer $u_\e$ of $\psi(\e)$, we prove that the corresponding $\l\in\R$ such that \eqref{EL proof z} holds satisfies
\begin{equation}
\label{mel}
n\,\l=(n-1)\,\AC(u_\e)+\,\frac1\e\int_{\R^n}W(u_\e)-\e\,\int_{\R^n}|\nabla u_\e|^2\,,
\end{equation}
as well as
\begin{eqnarray}\label{mel rescaled}
\Big|\l-2\,(n-1)\,\om_n^{1/n}\Big|\le C\,\sqrt{\e}\,.
\end{eqnarray}
In particular, up to decrease the value of $\e_0$, we always have $1/C\le \l\le C$ for a universal constant $C$. To prove \eqref{mel}, following \cite{luckhaus1989gibbs}, we test the distributional form of \eqref{EL proof z} with $\vphi=X\cdot\nabla u_\e$, for some $X\in C^\infty_c(\R^n;\R^n)$, and get
		\begin{eqnarray}\nonumber
		2\,\e\,\int_{\R^n}\nabla u_\e\cdot\nabla X[\nabla u_\e]&=&-\int_{\R^n}\Big\{2\,\e\,\nabla^2u_\e[\nabla u_\e]
		+\Big(\frac{W'(u_\e)}\e-\l\,V'(u_\e)\Big)\nabla u_\e\Big\}\cdot X
		\\\label{LM estimate}
		&=&\int_{\R^n}\Big\{\e\,|\nabla u_\e|^2
		+\frac{W(u_\e)}\e-\l\,V(u_\e)\Big\}\,\Div\,X\,.
		\end{eqnarray}
		We now pick $\eta \in C_c^\infty(B_2)$ with $0\le\eta \le 1$ on $B_2$ and $\eta=1$ in $B_1$. We set $\eta_R(x)=\eta(x/R)$ and test \eqref{LM estimate} with $X(x)=\eta_R(x)\,x$. We notice that $\Div X=n\,\eta_R+(x/R)\cdot(\nabla\eta)_R$, and that, by dominated convergence,
		\begin{eqnarray*}
			&&\lim_{R\to\infty}\int_{\R^n}\Big\{\e\,|\nabla u_\e|^2
			+\frac{W(u_\e)}\e-\l\,V(u_\e)\Big\}\,n\,\eta_R=n\,\big(\AC(u_\e)-\l\big)\,,
			\\
			&&\lim_{R\to\infty}\int_{\R^n}\Big\{\e\,|\nabla u_\e|^2
			+\frac{W(u_\e)}\e-\l\,V(u_\e)\Big\}\,\frac{x}R\,\cdot(\nabla\eta)_R=0\,,
			\\
			&&\lim_{R\to\infty}\int_{\R^n}\nabla u_\e\cdot\Big(\eta_R\,\Id+\frac{x}R\otimes(\nabla \eta)_R\Big)[\nabla u_\e]=\int_{\R^n}|\nabla u_\e|^2\,.
		\end{eqnarray*}
		In particular, \eqref{LM estimate} implies
\[
n\,\l=n\,\AC(u_\e)-2\,\e\,\int_{\R^n}|\nabla u_\e|^2\,,
\]
which can be easily rearranged into \eqref{mel}. At the same time, by \eqref{basic energy upper bound} we find
		\begin{eqnarray*}
			&&\int_{\R^n}\Big|\e\,|\nabla u_\e|^2-\frac{W(u_\e)}\e\Big|
			\\
			&\le&\Big(\int_{\R^n}\Big|\sqrt\e\,|\nabla u_\e|-\sqrt{\frac{W(u_\e)}\e}\Big|^2\Big)^{1/2}
			\Big(\int_{\R^n}\Big|\sqrt\e\,|\nabla u_\e|+\sqrt{\frac{W(u_\e)}\e}\Big|^2\Big)^{1/2}
			\\
			&=&\Big(\AC(u_\e)-\int_{\R^n}|\nabla\Phi(u_\e)|\Big)^{1/2}
\Big(\int_{\R^n}\Big|\sqrt\e\,|\nabla u_\e|+\sqrt{\frac{W(u_\e)}\e}\Big|^2\Big)^{1/2}
			\\
			&\le& C\,\sqrt{\e}\,\sqrt{\AC(u_\e)}\le C\,\sqrt\e\,,
		\end{eqnarray*}
which can be combined with \eqref{mel} and with \eqref{basic energy upper bound} to deduce \eqref{mel rescaled}.

\medskip

\noindent {\it Step six}: We are left to prove that every minimizer of $\psi(\e)$ is radial decreasing. Indeed, let $u$ be a generic, possibly non-radial, minimizer of $\psi(\e)$, and let $v\in\RR_0$ denote its radial decreasing rearrangement. By standard properties of rearrangements, $\int_{\R^n}V(u)=\int_{\R^n}V(v)=1$, while by the P\'olya-Szeg\"o inequality $\AC(u)\ge\AC(v)$, so that $v$ is a minimizer of $\psi(\e)$ and equality holds in the P\'olya-Szeg\"o inequality for $u$, that is
\begin{equation}
  \label{ps equality}
  \int_{\R^n}|\nabla u|^2=\int_{\R^n}|\nabla v|^2\,.
\end{equation}
By step four and five, $v$ solves the ODE
\begin{equation}
  \label{orf}
2\,\e^2\,\Big\{v''+(n-1)\,\frac{v'}r\Big\}=W'(v)-\l\,\e\,V'(v)\,,\qquad\mbox{on $(0,\infty)$}\,,
\end{equation}
with $0<1/C\le\l\le C$. Multiplying in \eqref{orf} by $v'$ and integrating over $(0,r)$ for some $r>0$ we obtain
\begin{equation}
  \label{ps equality 2}
  \e^2\,v'(r)^2+2\,(n-1)\,\int_0^r\frac{(v')^2}t=W(v(r))-\l\,\e\,V(v(r))+\l\,\e\,V(v(0))\,,\qquad\forall r>0\,,
\end{equation}
where we have used $v'(0)=0$, $v(1)=1$, and $W(1)=0$. If $r$ is such that $v(r)\le\de_0$, then by \eqref{W near the wells}, \eqref{V near the wells} and \eqref{ps equality 2} we find
\[
\e^2\,v'(r)^2\ge W(v)-C\,\e V(v)\ge\frac{v(r)^2}{C}-C\,\e\,\frac{v(r)^{2\,n/(n-1)}}{C}\ge \frac{v(r)^2}{C}\,,
\]
which gives, in particular, $v'(r)<0$; if $r$ is such that $v(r)\in(\de_0,1-\de_0)$, then, by the same method and thanks to $\inf_{(\de_0,1-\de_0)}W\ge 1/C$, we find that
\[
\e^2\,v'(r)^2\ge W(v)-C\,\e V(v)\ge\frac1{C}-C\,\e\ge\frac1{C}\,,
\]
so that, once again, $v'(r)<0$; finally, if the interval $\{v\ge 1-\de_0\}$ is non-empty, then it has the form $(0,a]$ for some $a>0$; multiplying \eqref{orf} by $r^{n-1}$, integrating over $(0,r)$, and taking into account that $W'<0$ on $(1-\de_0,1)$, $V'>0$ on $(0,1)$ and $\l>0$, we find
\begin{eqnarray*}
2\,\e^2\,r^{n-1}\,v'(r)=\int_0^r\,[W'(v)-\l\,\e\,V'(v)]\,r^{n-1}\,dr<0\,,
\end{eqnarray*}
that is, once again $v'(r)<0$. We have thus proved that $v'<0$ on $(0,\infty)$. This information, combined with \eqref{ps equality}, allows us to exploit the Brothers--Ziemer theorem \cite{brothersziemer} to conclude that $u$ is a translation of $v$. This shows that every minimizer of $\psi(\e)$ is in $\RR_0^*$, and concludes the proof of the theorem.
\end{proof}

The compactness argument used in the proof of Theorem \ref{theorem existence solutions} is relevant also in the implementation of the selection principle used in the proof of the stability estimate \eqref{global quantitative estimate} in the radial decreasing case. Specifically, an adaptation of that argument is needed in showing the existence of minimizers in the variational problems used in the selection principle strategy. In the interest of clarity, it thus seems convenient to discuss this adaptation in this same section. We thus turn to the proof of Theorem \ref{theorem compactness for selection principle} below. In the statement of this theorem we use for the first time the quantity
\begin{equation}
  \label{def of d Phi}
  d_\Phi(u,v)=\int_{\R^n}|\Phi(u)-\Phi(v)|^{n/(n-1)}\,,
\end{equation}
which is finite whenever $u,v\in H^1(\R^n;[0,1])$ (indeed, $u\in H^1(\R^n;[0,1])$ and $W(t)\le C\,t^2$ for $t\in[0,1]$ imply $\AC(u)<\infty$, thus $|D(\Phi(u))|(\R^n)<\infty$, and hence $\Phi(u)\in L^{n/(n-1)}(\R^n)$ by the $BV$-Sobolev inequality).  

\begin{theorem}\label{theorem compactness for selection principle}
  If $n\ge 2$ and $W\in C^{2,1}[0,1]$ satisfies \eqref{W basic} and \eqref{W normalization}, then there exist universal constant $\e_0$, $a_0$, $\ell_0$ and $C$ with the following properties.

  \medskip

  \noindent {\bf (i):} If $a\in(0,a_0)$, $\e<\e_0$, $u_\e$ is a minimizer of $\psi(\e)$, and $v_\e\in H^1(\R^n;[0,1])$ is such that
  \begin{equation}
    \label{spr hp}
      \int_{\R^n}V(v_\e)=1\,,\qquad \AC(v_\e)\le\psi(\e)+a\,\ell_0\,,\qquad d_\Phi(v_\e,u_\e)\le \ell_0\,,
  \end{equation}
  then the variational problem
  \[
  \g(\e,a,v_\e)=\inf\Big\{\AC(w)+a\,d_\Phi(w,v_\e):w\in H^1(\R^n;[0,1])\,,\int_{\R^n}V(w)=1\Big\}\,,
  \]
  admits minimizers.

  \medskip

  \noindent {\bf (ii):} If, in addition, $v_\e\in\RR_0$, then $\g(\e,a,v_\e)$ admits a minimizer $w_\e\in\RR_0$. Every such minimizer satisfies $w_\e\in \RR_0^*\cap C^{2,1/(n-1)}_{{\rm loc}}(\R^n)$, $0<w_\e<1$ on $\R^n$, and solves
\begin{equation}
	\label{EL classic W V Z}
	-2\,\e^2\,\Delta w_\e=\e\,w_\e\,(1-w_\e)\,\Err_\e-W'(w_\e)\qquad\mbox{on $\R^n$}\,,
\end{equation}
where $\Err_\e$ is a continuous radial function on $\R^n$ with
\begin{equation}
  \label{err eps}
  \sup_{\R^n}|\Err_\e|\le C\,.
\end{equation}
\end{theorem}

\begin{proof}
  {\it Step one}: Set $\g=\g(\e,a,v_\e)$ for the sake of brevity, and let $\{u_j\}_j$ be a minimizing sequence for $\g$.
  Since $a>0$, we can assume that
  \begin{equation}
    \label{compare through}
      \AC(u_j)+a\,d_\Phi(u_j,v_\e)\le\g+a\,\ell_0\,,\qquad\forall j\,.
  \end{equation}
  In particular, comparing $u_j$ by means of \eqref{compare through} with $v_\e$ and $u_\e$ respectively, we obtain the two basic bounds
  \begin{eqnarray}
    \label{spr 1}
    &&\AC(u_j)+a\,d_\Phi(u_j,v_\e)\le\AC(v_\e)+a\,\ell_0\le\psi(\e)+2\,\ell_0\,,
    \\
    \label{spr 2}
    &&\AC(u_j)+a\,d_\Phi(u_j,v_\e)\le\psi(\e)+a\,d_\Phi(u_\e,v_\e)+a\,\ell_0\,.
  \end{eqnarray}
  Subtracting $\psi(\e)$ from \eqref{spr 2}, noticing that $\AC(u_j)\ge\psi(\e)$, and using \eqref{spr hp}, we also find
  \begin{equation}
    \label{spr 3}
    d_\Phi(u_j,v_\e)\le d_\Phi(u_\e,v_\e)+\ell_0\le 2\,\ell_0\,,
  \end{equation}
  and hence, using again \eqref{spr hp},
  \begin{equation}
    \label{spr 4}
    d_\Phi(u_j,u_\e)\le C\,\ell_0\,.
  \end{equation}
  Finally, by \eqref{spr hp}, \eqref{spr 1}, and $\psi(\e)\le 2\,n\,\om_n^{1/n}+C\,\e$, we can apply step one of the proof of Theorem \ref{theorem existence solutions} to $u_j$, $u_\e$ and $v_\e$, to find
  \begin{equation}
    \label{veps concentration}
    \min\Big\{\int_{B_{M_0}}V(u_j)\,,\int_{B_{M_0}}V(u_\e)\,,\int_{B_{M_0}}V(v_\e)\Big\}\ge 1-C\,\sqrt{\ell_0+\e_0}\,,\qquad\forall j\,,
  \end{equation}
  where $M_0$ is a universal constant. Since \eqref{veps concentration} rules out the possibility of the vanishing case for $\{V(u_j)\,dx\}_j$, we can directly assume that the dichotomy case occurs, and in particular that there exists
  \begin{equation}
  \label{alfa dicho spr}
  \a\in\Big(1-C\,\sqrt{\ell_0+\e_0},1\Big)\,,
  \end{equation}
  such that for every $\tau\in(0,\min\{\a/2,\tau_*\})$ (here $\tau_*$ is as in \eqref{def of tau star})  we can find $S(\tau)>0$, $S_j(\tau)\to\infty$ and a cut-off function $\vphi$ between $B_{S(\tau)}$ and $B_{S_j(\tau)}$ such that $|\nabla \vphi|\le 2\,S_j(\tau)^{-1}$ on $\R^n$,
  and
\begin{eqnarray}\label{dicho 3 wj}
&&\hspace{1.2cm}\a-C\,\tau\le\int_{B_{S(\tau)}}V(u_j),\int_{\R^n}V(\vphi\,u_j)\le\a+C\,\tau\,,
\\\nonumber
&&(1-\a)-C\,\tau\le\int_{B_{S_j(\tau)}^c}V(u_j),\int_{\R^n}V((1-\vphi)\,u_j)\le (1-\a)+C\,\tau\,.
\end{eqnarray}
We can now {\it verbatim} repeat the argument used in step two of the proof of Theorem \ref{theorem existence solutions} to deduce \eqref{dicho 1}, and find that, if $\s=\tau^{(n-1)/(3\,n-1)}$ as in \eqref{def of tau star}, then
\begin{equation}
\label{dicho 1 spr}
\AC(u_j)\ge \AC(\vphi\,u_j)+\AC((1-\vphi)\,u_j)-C\,\Big(\frac{\s}\e+\frac{1}{S_j(\tau)}\Big)\,;
\end{equation}
in the same vein, by exactly the same argument used to deduce \eqref{dicho lb on 1 minus vphi uj}, we also have
\begin{eqnarray}\label{dicho lb on 1 minus vphi uj spr}
  \AC((1-\vphi)\,u_j)\ge c(n)\,((1-\a)-C\,\tau)^{(n-1)/n}\,.
\end{eqnarray}
We now need to show that the $\AC(\vphi\,u_j)$ term is larger than $\g$ up to ${\rm O}(1-\a)$ and ${\rm O}(\tau)$ errors, but, for reasons that will become clearer in a moment, we cannot do this by just taking a rescaling of $\vphi\,u_j$ as done in Theorem \ref{theorem existence solutions}. We will rather need to introduce the ``localized'' family of rescalings which we now pass to describe.

\medskip

We let $\zeta\in C^\infty_c(B_{2\,M_0};[0,1])\cap\RR_0$ with $\zeta=1$ on $B_{M_0}$ and $|\zeta'|\le 2/M_0$. In particular,
\begin{equation}
  \label{zeta bound}
  |x|\,|\zeta'|\le 2\qquad\mbox{on $\R^n$}\,.
\end{equation}
Next, we set $f_t(x)=x+t\,\zeta(|x|)\,x$ and $\hat x=x/|x|$ for $x\in\R^n$ and $t>0$. By \eqref{zeta bound}, if $|t|\le t_0=t_0(n)<1$, then $f_t:\R^n\to\R^n$ is a diffeomorphism with
\begin{eqnarray*}
f_t(x)&=&x\,,\qquad\hspace{0.9cm}\mbox{on $B_{2\,M_0}^c$}\,,
\\
f_t(x)&=&(1+t)\,x\qquad\mbox{on $B_{M_0}$}\,,
\\
  \nabla f_t(x)&=&(1+t\,\zeta)\,\Id+t\,|x|\,\zeta'\,\hat x\otimes \hat x\,,
  \\
  Jf_t(x)&=&(1+t\,\zeta)^{n-1}\,\big(1+t\,(\zeta+|x|\,\zeta')\big)=1+\big(n\,\zeta+|x|\,\zeta'\big)\,t+{\rm O}(t^2)\,.
\end{eqnarray*}
We set $v_j(t)=(\vphi\,u_j)\circ f_t$, so that $v_j(0)=\vphi\,u_j$, and consider the functions
\[
b_j(t)=\int_{\R^n}V\big(v_j(t)\big)=\int_{\R^n}V(\vphi\,u_j)\,Jf_t\,,\qquad |t|\le t_0\,.
\]
Clearly we have
\begin{eqnarray}\label{bj 0}
  b_j(0)&=&\int_{\R^n}V(\vphi\,u_j)\in[\a-C\,\tau,\a+C\,\tau]\,,
  \\\label{bj second t}
  |b_j''(t)|&=&\int_{\R^n}V(\vphi\,u_j) \Big|\frac{d^2(Jf_t)}{dt^2}\Big|\le C\,,\qquad\forall |t|\le t_0\,;
\end{eqnarray}
more crucially, if we choose $\e_0$ and $\ell_0$ small enough, then by \eqref{veps concentration} and \eqref{zeta bound} we find
\begin{eqnarray*}
  b_j'(0)=\int_{\R^n}V(\vphi\,u_j)\,\big(n\,\zeta+|x|\,\zeta'\big)\ge n\,\int_{B_{M_0}}V(u_j)-(n+2)\,\int_{B_{2\,M_0}\setminus B_{M_0}}V(u_j)
  \ge\frac{n}2\,.
\end{eqnarray*}
As a consequence, by \eqref{bj second t}, we can find a universal constant $t_1$ such that
\begin{equation}
  \label{bj prime t bounded below}
  b_j'(t)\ge\frac{n}3\qquad\forall |t|\le t_1\,.
\end{equation}
 In particular, $b_j$ is strictly increasing on $[-t_1,t_1]$, with
\begin{eqnarray*}
&&b_j(t_1)\ge b_j(0)+\frac{n}3\,t_1\ge \a-C\,\tau+\frac{n}3\,t_1>1-C\,(\ell_0+\e_0+\tau)+\frac{n}3\,t_1>1\,,
\\
&&b_j(-t_1)\le b_j(0)-\frac{n}3\,t_1\le \a + C\,\tau-\frac{n}3\,t_1\le 1+C\,(\ell_0+\e_0+\tau)-\frac{n}3\,t_1<1-\frac{n}4\,t_1\,,
\end{eqnarray*}
so that, for every $j$, there exists $t_j\in(-t_1,t_1)$ such that $b_j(t_j)=1$: in other words,
\begin{equation}
  \label{tj gives vol 1}
  \int_{\R^n}V(v_j(t_j))=1\,.
\end{equation}
We now compare the energy of $v_j(t_j)=(\vphi\,u_j)\circ f_{t_j}$ to that of $\vphi\,u_j$. To this end, we first notice that, by comparing $b_j(0)=\int_{\R^n}V(\vphi\,u_j)=\a+{\rm O}(\tau)$ to $b_j(1)=1$, thanks to \eqref{bj prime t bounded below} we conclude that
\begin{equation}
  \label{tj relates to alfa}
  |t_j|\le C\,\big((1-\a)+\tau\big)\,,\qquad\forall j\,.
\end{equation}
Denoting by $\|A\|$ the operator norm of a linear map $A$, we have
\[
\|\nabla f_t(x)-\Id\|\le C\,|t|\,,\qquad |Jf_t(x)-1|\le C\,|t|\,,\qquad\forall x\in\R^n\,,
\]
so that
\begin{eqnarray*}
  \AC(v_j(t))&=&\int_{\R^n}\Big\{\e\,\Big|\Big(\nabla f_t\circ f_t^{-1}\Big)[\nabla(\vphi\,u_j)]\Big|^2+\frac{W(\vphi\,u_j)}\e\Big\}\,Jf_t
  \\
  &\le&\int_{\R^n}\Big\{\e\,(1+C\,|t|)^2|\nabla(\vphi\,u_j)|^2+\frac{W(\vphi\,u_j)}\e\Big\}\,(1+C\,|t|)
  \\
  &\le&(1+C\,|t|)\,\AC(\vphi\,u_j)\,.
\end{eqnarray*}
Therefore if we combine \eqref{dicho 1 spr}, \eqref{dicho lb on 1 minus vphi uj spr}, and \eqref{tj relates to alfa} with this last estimate, and take into account that $\AC(u_j),\AC(\vphi\,u_j)\le C$, then we obtain
\begin{eqnarray}\nonumber
\AC(u_j)+a\,d_\Phi(u_j,v_\e)&\ge& \AC(v_j(t_j))+a\,d_\Phi(v_j(t_j),v_\e)
\\\label{key step 2}
&&
+a\,\Big(d_\Phi(u_j,v_\e)-d_\Phi(v_j(t_j),v_\e)\Big)
\\\nonumber
&&+c(n)\,((1-\a)-C\,\tau)^{(n-1)/n}-C\,\Big((1-\a)+\tau+\frac{1}{S_j(\tau)}+\frac{\s}\e\Big)\,.
\end{eqnarray}
We notice that for every $u,v\in H^1(\R^n;[0,1])$, thanks to the triangular inequality in $L^{n/(n-1)}$ and to $|b^{1/n'}-a^{1/n'}|\ge c(n)\,b^{-1/n}(b-a)$ for $0<a<b$, we have
\begin{equation}
  \label{careful with that}
  c(n)\,\frac{|d_\Phi(u,v_\e)-d_\Phi(v,v_\e)|}{\max\{d_\Phi(u,v_\e),d_\Phi(v,v_\e)\}^{1/n}}\le d_\Phi(u,v)^{(n-1)/n}\,.
\end{equation}
We apply \eqref{careful with that} with $u=u_j$ and $v=u_j\,\vphi$ to find
\begin{eqnarray*}
\big|d_\Phi(u_j,v_\e)-d_\Phi(\vphi\,u_j,v_\e)\big|&\le& C\, \int_{\R^n}|\Phi(u_j)-\Phi(\vphi\,u_j)|^{n/(n-1)}
\\
&\le&\int_{\R^n\setminus B_{S(\tau)}}V(u_j)\le C\,\big((1-\a)+\tau\big)
\end{eqnarray*}
where we have used \eqref{dicho 3 wj}. Similarly, noticing that
\begin{eqnarray*}
\frac{d}{ds}\,\Phi(v_j(s))&=&\sqrt{W(v_j(s))}\,\big[\nabla(\vphi\,u_j)\circ f_s\big]\cdot\frac{d}{ds}\,f_s
\\
&=&\sqrt{W(v_j(s))}\,\big[\nabla(\vphi\,u_j)\circ f_s\big]\cdot\big(\zeta(|x|)\,x\big)\,,
\end{eqnarray*}
with $\zeta(x)\,|x|\le 2\,M_0$ for every $x\in\R^n$ by \eqref{zeta bound}, we find\footnote{This is the key step where using $f_t(x)$ rather than $(1+t)\,x$ (as done when proving Theorem \ref{theorem existence solutions}) makes a substantial difference. Indeed, by using a global rescaling to fix the volume constraint of $\vphi\,u_j$, we end up having to control, in the analogous estimate to \eqref{key step}, the first moment of the energy density of $\vphi\,u_j$, i.e.
\[
\int_{\R^n}|x|\,\Big(\e\,|\nabla(\vphi\,u_j)|^2+\frac{W(\vphi\,u_j)}\e\Big)\,,
\]
rather than the trivially bounded quantity $M_0\,\AC(u_j)$.}
\begin{eqnarray}\nonumber
&&\big|d_\Phi(v_j(t_j),v_\e)-d_\Phi(\vphi\,u_j,v_\e)\big|
\\\nonumber
&\le&C\,\int_{\R^n}|\Phi(v_j(t_j))-\Phi(\vphi\,u_j)|^{n/(n-1)}
\le C\,\int_{\R^n}|\Phi(v_j(t_j))-\Phi(\vphi\,u_j)|
\\\nonumber
&\le&C\,\Big|\int_0^{t_j}\,ds\,\int_{\R^n}\sqrt{W(v_j(s))}\,\big[\nabla(\vphi\,u_j)\circ f_s\big]\cdot\big(\zeta(|x|)\,x\big)\Big|
\\\nonumber
&\le&C\,\Big|\int_0^{t_j}\,ds\,\int_{\R^n}\sqrt{W(\vphi\,u_j)}\,\nabla(\vphi\,u_j)\cdot\big(\zeta(|f_s^{-1}|)\,f_s^{-1}\big)\,Jf_s\Big|
\\\label{key step}
&\le&C\,M_0\,|t_j|\,\int_{\R^n}\sqrt{W(\vphi\,u_j)}\,|\nabla(\vphi\,u_j)|\le C\,|t_j|\,\AC(\vphi\,u_j)\,.
\end{eqnarray}
We finally combine \eqref{tj relates to alfa}, \eqref{key step 2}, \eqref{key step},and the fact that $v_j(t_j)$ is a competitor for $\g$ to conclude that
\begin{equation*}
\AC(u_j)+a\,d_\Phi(u_j,v_\e)\ge\g+c(n)\,((1-\a)-C\,\tau)^{(n-1)/n}-C\,\Big((1-\a)+\tau+\frac\s\e+\frac{1}{S_j(\tau)}\Big)\,.
\end{equation*}
Letting $j\to\infty$ and then $\tau\to 0^+$ (so that $\s\to 0^+$ thanks to \eqref{def of tau star}), we finally conclude
\begin{eqnarray*}
0\ge c(n)\,(1-\a)^{(n-1)/n}-C\,(1-\a)\,,
\end{eqnarray*}
which gives a contradiction with \eqref{alfa dicho spr} if $\e_0$ and $\ell_0$ are small enough. Having excluded vanishing and dichotomy, by a standard argument we deduce the existence of a minimizer of $\g$.

\medskip

\noindent {\it Step two}: We now assume that $v_\e\in\RR_0$. Since $\Phi$ is an increasing function on $[0,1]$, if $u^*$ denotes the radial decreasing rearrangement of $u:\R^n\to[0,\infty)$, then $\Phi(u^*)=\Phi(u)^*$. In particular, by a standard property of rearrangements,
\[
d_\Phi(u,v)=\int_{\R^n}|\Phi(u)-\Phi(v)|^{n/(n-1)}\ge \int_{\R^n}|\Phi(u)^*-\Phi(v)^*|^{n/(n-1)}=d_\Phi(u^*,v^*)\,.
\]
This fact, combined with the P\'olya--Szeg\"o inequality and the fact that $v_\e^*=v_\e$, implies that the radial decreasing rearrangement of a minimizer of $\g$ is also a minimizer of $\g$ (in brief, a radial decreasing minimizer).

\medskip

We now show that every radial decreasing minimizer $w_\e$ of $\g$ satisfies $0<w_\e<1$ on $\R^n$, that $w_\e\in C^{2,1/(n-1)}_{{\rm loc}}(\R^n)$, and that \eqref{EL classic W V Z} holds for a radial continuous function $\Err_\e$ bounded by a universal constant. Arguing as in step four of the proof of Theorem \ref{theorem existence solutions}, with $0\le a<b\le+\infty$ and $\Om=B_b\setminus\ov{B}_a=\{0<w_\e<1\}$, we see that $w_\e$ solves
\begin{equation}
  \label{EL proof z Z Omega}
  -2\,\e^2\,\Delta w_\e=\e\,\l V'(w_\e)-W'(w_\e)-a\,\e\,Z_\e(x,w_\e)\qquad\mbox{in $\D'(\Om)$}\,,
\end{equation}
where, for $x\in\R^n$ and $t\in[0,1]$, we have set
\[
Z_\e(x,t)=\frac{n}{n-1}\,\big|\Phi(t)-\Phi(v_\e)\big|^{(n/(n-1))-2}\,\big(\Phi(t)-\Phi(v_\e)\big)\,\sqrt{W(t)}\,.
\]
By \eqref{EL proof z Z Omega}, $\Delta w_\e$ is bounded in $\Om$, and thus, by the Calderon--Zygmund theorem, $w_\e\in\Lip_{{\rm loc}}(\Om)$. This implies that $Z_\e(x,t)\in C^{0,1/(n-1)}_{{\rm loc}}(\Om)$, and thus, by Schauder's theory, that $w_\e\in C^{2,1/(n-1)}_{{\rm loc}}(\Om)$. We now want to prove that $\Om=\R^n$. By the same variational arguments used in deriving \eqref{proof EL out} and \eqref{proof EL in}, we have that
\begin{eqnarray}
  \label{proof EL out ww}
  &&-2\,\e^2\,\Delta w_\e\ge f(x,t)\qquad\mbox{in $\D'(\R^n\setminus \ov{B_a})$}\,,
  \\
  \label{proof EL in ww}
  &&-2\,\e^2\,\Delta w_\e\le f(x,t)\qquad\mbox{in $\D'(B_b)$\,,}
\end{eqnarray}
where $f(x,t)$ satisfies
\begin{equation}
  \label{f bound x t}
  |f(x,t)|\le C\,\,t\,(1-t)\,,\qquad\forall (x,t)\in\R^n\times[0,1]\,,
\end{equation}
thanks to \eqref{W near the wells} and \eqref{V near the wells} (which, in particular, give $|Z_\e(x,t)|\le C\,t\,(1-t)$ for every $(x,t)\in\R^n\times[0,1]$). By repeating the same argument used in step four of the proof of Theorem \ref{theorem existence solutions}, we thus see that $\Om=\R^n$. Finally, it is easily seen that \eqref{EL proof z Z Omega} with $\Om=\R^n$ and $w_\e\in C^2(\R^n)$, takes the form
\begin{equation}
  \label{EL proof z Z Omega with Err}
  -2\,\e^2\,\Delta w_\e=\e\,w_\e\,(1-w_\e)\,\Err_\e-W'(w_\e)\qquad\mbox{on $\R^n$}\,,
\end{equation}
for a radial function $\Err_\e$ bounded by a universal constant on $\R^n$, as claimed.
\end{proof}

 \section{Resolution of almost-minimizing sequences}\label{section resolution of minimizers} In the main result of this section, Theorem \ref{theorem asymptotics of minimizers} below, we provide a sharp description, up to first order as $\e\to 0^+$, of the minimizers of $\psi(\e)$.  This resolution result is proved not only for minimizers of $\psi(\e)$, but also for a general notion of ``critical sequence for $\psi(\e_j)$ as $\e_j\to 0^+$'' modeled after the selection principle minimizers of Theorem \ref{theorem compactness for selection principle}.

 \medskip

 In the following statement, $\eta$ is the solution of $\eta'=-\sqrt{W(\eta)}$ on $\R$ with $\eta(0)=1/2$,
 \[
 \tau_0=\int_\R \eta'\,V'(\eta)\,s\,ds,\qquad \tau_1=\int_\R\,W(\eta)\,s\,ds\,,
 \]
 and $R_0=\om_n^{-1/n}$. Relevant properties of $\eta$ are collected in section \ref{subsection eta}.

\begin{theorem}\label{theorem asymptotics of minimizers}
If $n\ge 2$ and $W\in C^{2,1}[0,1]$ satisfies \eqref{W basic} and \eqref{W normalization}, then there exist universal constants $\e_0$, $\de_0$, and $\ell_0$ with the following properties:

\medskip

\noindent  {\bf Ansatz:} for every $\e<\e_0$ there exists a unique $\tau_\e\in\R$ such that if we set
  \begin{equation}
    \label{def of zeta eps}
      z_\e(x)=\eta\Big(\frac{|x|-R_0}{\e}-\tau_\e\Big)\,,
  \end{equation}
  then
  \begin{equation}
    \label{def of R and taueps}
    \int_{\R^n}V(z_\e)=1\,.
  \end{equation}
  Moreover, we have $|\tau_\e-\tau_0|\le C\,\e$ and, in the limit as $\e\to 0^+$,
  \begin{equation}
  \label{ansatz energy}
  \AC(z_\e)= 2\,n\,\omega_n^{1/n}+2\,n\,(n-1)\,\omega_n^{2/n}\,\big(\tau_0+\tau_1\,\big)\,\e +{\rm O}(\e^2)\,.
  \end{equation}

\medskip
  \noindent {\bf Resolution of critical sequences:} if $\e_j\to 0^+$ as $j\to\infty$, $\{v_j\}_j$ is a sequence in $C^2(\R^n;[0,1])\cap\RR_0$ such that
\begin{eqnarray}
  \label{critical sq volume 1}
 &&\int_{\R^n}V(v_j)=1\,,
 \\
 \label{critical sq L energy}
 &&\mathcal{AC}_{\e_j}(v_j)\le 2\,n\,\om_n^{1/n}+\ell_0\,,
\end{eqnarray}
and $\{\Err_j\}_j$ is a sequence of radial continuous functions on $\R^n$ with
 \begin{eqnarray}
 \label{critical sq ode}
 &&-2\,\e_j^2\,\Delta v_j=\e_j\,v_j\,(1-v_j)\,\Err_j-W'(v_j)\qquad\mbox{on $\R^n$}\,,
 \\
  \label{critical sq bounds}
 &&\sup_j\|\Err_j\|_{C^0(\R^n)}\le C\,,
 \end{eqnarray}
 then, for $j$ large enough, we have
  \begin{equation}
    \label{critical sq resolution of vj}
      v_j(x)=z_{\e_j}(x)+f_j\Big(\frac{|x|-R_0}{\e_j}\Big)\,,\qquad x\in\R^n\,,
  \end{equation}
  where $f_j\in C^2(-R_0/\e_j,\infty)$, and
  \begin{equation}
    \label{critical sequence bounds on fj}
     |f_j(s)|\le C\,\e_j\,e^{-|s|/C}\,,\qquad\forall s\ge-R_0/\e_j\,.
  \end{equation}
  Moreover, for $j$ large enough, there exist positive constants $b_j$ and $c_j$ such that
  \begin{equation}
    \label{def of bj and cj}
     v_j(R_0+c_j)=\de_0\,,\qquad v_j(R_0-b_j)=1-\de_0\,,
  \end{equation}
  and $b_j$ and $c_j$ satisfy
  \begin{equation}\label{critical sq ring basic}
  \frac{\e_j}{C}\le b_j\,,c_j\le C\,\e_j\,.
  \end{equation}
  Finally, one has
  \begin{equation}\label{critical sq vj prime annulus}
  \hspace{1.7cm}\,\frac{C}{\e_j}\ge-v_j'(r)\ge\frac1{C\,\e_j}\,,
    \hspace{3.5cm}\qquad\forall r\in[R_0-b_j,R_0+c_j]\,,
  \end{equation}
  \begin{equation}
    \label{critical sq vj estimates on 0 R}
  \hspace{1.5cm}\,\left\{
    \begin{split}
      &v_j(r)\le C\,e^{-(r-R_0)/C\,\e_j}\,,
      \\
      &|v_j^{(k)}(r)|\leq \frac{C}{\e_j^k}\,e^{-(r-R_0)/C\,\e_j}\,,
    \end{split}\right .
    \hspace{1.4cm}\qquad\forall r\in[R_0+c_j,\infty)\,,\,k=1,2\,,
  \end{equation}
  \begin{equation}
    \label{critical sq vj estimates on R infinity}
    \hspace{1.5cm}\,\left\{
    \begin{split}
      &1-v_j(r)\le C\,e^{-(R_0-r)/C\,\e_j}\,,
      \\
      &|v_j'(r)|\leq C\,\min\Big\{\frac{r}{\e_j^2},\frac1{\e_j}\Big\}\,e^{-(R_0-r)/C\,\e_j}\,,
      \\
      &|v_j''(r)|\leq \frac{C}{\e_j^2}\,e^{-(R_0-r)/C\,\e_j}\,,
    \end{split}\right .
    \qquad\hspace{1.1cm}\forall r\in(0,R_0-b_j)\,.
  \end{equation}
\end{theorem}

\begin{proof} The first two steps of the proof take care of the {\it Ansatz}-part of the statement, while starting from step three we address the resolution result. We premise the remark that, if we set $z_\tau(x)=\eta([(|x|-R_0)/\e]-\tau)$, then $f(\tau)=\int_{\R^n}V(z_\tau)$ is strictly increasing in $\tau$ with $f(-\infty)=0$ and $f(+\infty)=+\infty$. For this reason, $\tau_\e$ is indeed uniquely defined by \eqref{def of R and taueps}.

\medskip

\noindent  {\it Step one}: We prove that if $\{w_\e\}_{\e>0}$ is defined by
  \[
  w_\e(x)=\eta\Big(\frac{|x|-R_0}{\e}-t_\e\Big)+f_\e\Big(\frac{|x|-R_0}{\e}\Big)\,,\qquad x\in\R^n\,,\e>0\,,
  \]
  for some $t_\e\in\R$ and some functions $f_\e\in C^2(-R_0/\e,\infty)$ such that
  \begin{eqnarray}
    \label{eq volume const pertur}
  &&\int_{\R^n} V(w_\e)=1\,,
  \\
   \label{thanks to the decay}
   &&|f_\e(s)|\le C\,\e\,e^{-|s|/C}\,,\qquad\forall s\ge-R_0/\e\,,
  \end{eqnarray}
  then
  \begin{equation}\label{lemma tau}
  |t_\e-\tau_0|\le C\,\e\,,\qquad\forall \e<\e_0\,,
  \end{equation}
  Of course, in the particular case when $f_\e\equiv 0$, we have $w_\e=z_\e$ and $t_\e=\tau_\e$ thanks to \eqref{def of zeta eps} and \eqref{def of R and taueps}.

  \medskip

  Indeed, setting $z_0(x)=\eta([(|x|-R_0)/\e]-\tau_0)$ for $x\in\R^n$, and recalling \eqref{def of R and taueps} and \eqref{eq volume const pertur}, we consider the quantity
  \begin{equation}\label{def of ell eps}
  \k_\e=\int_{\R^n}V(1_{B_{R_0}})-V(z_0)=\int_{\R^n}V(w_\e)-V(z_0)\,.
  \end{equation}
  We look at the first expression for $\k_\e$, passing first to the radial coordinate $r=|x|$ and then changing variables into $s=(r-R_0)/\e$. By taking into account the fact that $\tau_0$ satisfies
  \[
  \int_{\R}\Big(1_{(-\infty,0)}(s)-V(\eta(s-\tau_0))\Big)\,ds=0\,,
  \]
  see \eqref{definition of tau0}, we find
  \begin{eqnarray*}
      \frac{\k_\e}{n\,\om_n}&=&\e\,\int_{-R_0/\e}^\infty \Big( \,1_{(-\infty,0)}(s)-V(\eta(s-\tau_0))\Big)\,(R_0+\e\,s)^{n-1}\,ds
      \\
      &=&\e\,R_0^{n-1}\,\int_{-R_0/\e}^\infty \Big( \,1_{(-\infty,0)}(s)-V(\eta(s-\tau_0))\Big)\,ds
      \\
      &&+\e\,\int_{-R_0/\e}^\infty \Big( \,1_{(-\infty,0)}(s)-V(\eta(s-\tau_0))\Big)\,[(R_0+\e\,s)^{n-1}-R_0^{n-1}]\,ds
      \\
      &=&-\e\,R_0^{n-1}\,\int_{-\infty}^{-R_0/\e} \Big( \,1_{(-\infty,0)}(s)-V(\eta(s-\tau_0))\Big)\,ds
      \\
      &&+\e\,\sum_{k=0}^{n-2}a_k\,\int_{-R_0/\e}^\infty \Big( \,1_{(-\infty,0)}(s)-V(\eta(s-\tau_0))\Big)\,R_0^k\,(s\,\e)^{n-1-k}\,ds,
  \end{eqnarray*}
  with $a_k = \binom{n-1}{k}$. Since $\tau_0=\tau_0(W)$, by the decay properties \eqref{eta decay} of $\eta$, we have
  \begin{equation}
    \label{exp V eta}
      |\,1_{(-\infty,0)}(s)-V(\eta(s-\tau_0))|\le C\,e^{-|s|/C}\qquad\forall s\in\R\,,
  \end{equation}
  so that
  \[
  \Big|\int_{-\infty}^{-R_0/\e} \Big( \,1_{(-\infty,0)}(s)-V(\eta(s-\tau_0))\Big)\,ds\Big|\le C\,\int_{-\infty}^{-R_0/\e}\,e^{-|s|/C}\,ds\le C\,e^{-R_0/C\,\e}\,,
  \]
  and, recalling that $\om_n\,R_0^n=1$,
  \[
  |\k_\e|\le C\,\e\,e^{-R_0/C\,\e}+C\,\e^2\,\sum_{j=1}^{n-1}\,\int_{-R_0/\e}^\infty \Big| \,1_{(-\infty,0)}(s)-V(\eta(s-\tau_0))\Big|\,|s|^j\,ds\le C\,\e^2\,,
  \]
  where in the last inequality we have used \eqref{exp V eta} again. Taking into account the second formula for $\k_\e$ in \eqref{def of ell eps}, we have thus proved
  \begin{equation}
    \label{ell eps is order eps}
      C\,\e^2\ge\Big|\int_{\R^n}V(w_\e)-V(z_0)\Big|\,.
  \end{equation}
  With the same change of variables used before we have
  \[
  C\,\e\ge \Big|\int_{-R_0/\e}^\infty  \Big\{V\big(\eta(s-t_\e)+ f_\e(s)\big)-V(\eta(s-\tau_0))\Big\}\,(R_0+\e\,s)^{n-1}\,ds\Big|\,,
  \]
  while the decay properties of $f_\e$ assumed in \eqref{thanks to the decay} give
\begin{eqnarray*}
&&\Big|\int_{-R_0/\e}^\infty \Big\{ V\big(\eta(s-t_\e)+f_\e(s)\big)-V(\eta(s-t_\e))\Big\}\,(R_0+\e\,s)^{n-1}\,ds\Big|
\\
&&\leq  \int_{-R_0/\e}^\infty f_\e(s)\,(R_0+\e\,s)^{n-1}\,ds\,\int_0^1  V'\big(\eta(s-t_\e)+r\, f_\e(s)\big)\,dr\,\leq C\, \e\,;
\end{eqnarray*}
by combining the last two inequalities we thus find
\begin{eqnarray}\nonumber
C\,\e&\ge& \Big|\int_{-R_0/\e}^\infty  \Big\{V(\eta(s-t_\e))-V(\eta(s-\tau_0))\Big\}\,(R_0+\e\,s)^{n-1}\,ds\Big|
\\\label{fink}
&=& \int_{-R_0/\e}^\infty  |V(\eta(s-t_\e))-V(\eta(s-\tau_0))|\,(R_0+\e\,s)^{n-1}\,ds \,,
\end{eqnarray}
where in the last step we have used that $\tau \to V(\eta(.-\tau))$ is strictly increasing in $\tau$. Since \eqref{fink} implies $t_\e\to \tau_0$ as $\e\to 0^+$, we can choose $\e_0=\e_0(n,W)$ so that $|t_\e-\tau_0|\le 1$ and $R_0+\e\,(\tau_0-1)\ge R_0/2$. Since $V\circ\eta$ is strictly decreasing on $\R$, we have $|(V\circ\eta)'|\ge 1/C$ on $[-2,2]$, and noticing that if $|s-\tau_0|\le 1$, then $|s-t_\e|<2$, we finally conclude
  \[
  C\,\e\ge\int_{\tau_0-1}^{\tau_0+1} \frac{|(s-t_\e)-(s-\tau_0)|}{C}\,(R_0+\e\,s)^{n-1}\,ds\ge \frac{|\tau_0-t_\e|}C\,,
  \]	
  thus proving \eqref{lemma tau}.

  \medskip

  \noindent {\it Step two}: We compute $\AC(z_\e)$. Passing to the radial coordinate $r=|x|$, setting first $r=R_0+\e\,s$ and then $t=s-\tau_\e$, recalling that $\eta'=-\sqrt{W(\eta)}$, and exploiting the decay property \eqref{eta decay} of $\eta$ at $-\infty$, we find that, as $\e\to 0^+$,
  \begin{eqnarray}\nonumber
  \AC(z_\e)&=&n\,\om_n\,\int_{-R_0/\e}^\infty \Big(\eta'(s-\tau_\e)^2+ W(\eta(s-\tau_\e))\Big)\,(R_0+\e\,s)^{n-1}\,ds
    \\\nonumber
    &=&2\,n\,\om_n\,\int_{-\tau_\e-R_0/\e}^\infty W(\eta(t)) \,\,(R_0+\e (t+\tau_\e))^{n-1}\,dt
    \\\nonumber
    &=&2\,n\,\om_n\,\int_{-\infty}^\infty W(\eta(t))\,\,(R_0+\e (t+\tau_\e))^{n-1}\,dt+{\rm O}(e^{-C/\e})
    \\\label{energy expansion 1}
    &=&2\,n\,\om_n\,\int_{-\infty}^\infty W(\eta(t))\,\,(R_0+\e (t+\tau_0))^{n-1}\,dt+{\rm O}(\e^2)\,,
  \end{eqnarray}
  where in the last step we have used $\tau_\e=\tau_0+{\rm O}(\e)$. Recalling that, by \eqref{W normalization},
  \[
  \int_\R W(\eta)=-\int_\R \sqrt{W(\eta)}\,\eta'=-\int_{\R}\Phi'(\eta)\,\eta'=\Phi(\eta(-\infty))-\Phi(\eta(+\infty))=\Phi(1)=1\,,
  \]
  as well as that $\om_n\,R_0^n=1$, we find
    \begin{eqnarray}\nonumber
  \AC(z_\e)=2\,n\,\om_n^{1/n}+2\,n\,(n-1)\,\om_n^{2/n}\,(\tau_0+\tau_1)\,\e+{\rm O}(\e^2)\,,
  \end{eqnarray}
  as $\e\to 0^+$,  that is \eqref{ansatz energy}. This proves the first part of the statement of the theorem.

   \medskip

\noindent {\it Step three}: In preparation to the proof of the second part of the statement, we show that if $\e<\e_0$ and $u\in H^1(\R^n;[0,1])$ satisfies
\begin{equation}
  \label{monotonicity low energy bis}
\AC(u)\le 2\,n\,\om_n^{1/n}+\ell_0\,,\qquad\int_{\R^n}V(u)=1\,,
\end{equation}
then
\begin{equation}
  \label{stima BV utile}
  \int_{\R^n}|\Phi(u)-1_{B_{R_0}}|^{n/(n-1)}\le C\,\big((\sqrt{\ell_0})^{(n-1)/n}+\e\big)\,.
\end{equation}
Moreover, if $u\in\RR_0$, then $\sqrt{\ell_0}$ can be replaced by $\ell_0$ in \eqref{stima BV utile}.

\medskip

Indeed, by \eqref{monotonicity low energy bis}, as seen in step one of the proof of Theorem \ref{theorem existence solutions}, we have
\begin{equation}\label{non vanishing bound bis}
		\int_{\R^n}|\Phi(u)-(\om_n^{1/n}\,r(u))^{1-n}\,1_{B_{r(u)}}|^{n/(n-1)}\le C\,\sqrt{\ell_0}\,,
\end{equation}
for some $r(u)\in(0,M_0]$, where $M_0$ is a universal constant. Setting $f(r)=(\om_n^{1/n}\,r)^{1-n}$, and noticing that $f(R_0)=1$, it is enough to prove that
\begin{equation}
  \label{ru and R0}
  |r(u)-R_0|\le C\,\big((\sqrt{\ell_0})^{(n-1)/n}+\e\big)\,,\qquad \big|f(r(u))-1\big|\le C\,\big((\sqrt{\ell_0})^{(n-1)/n}+\e\big)\,.
\end{equation}
Since $\Lip(f,[R_0/2,2R_0])\le C$ and $f(R_0)=1$, it is enough to prove the first estimate in \eqref{ru and R0}. To this end, we start noticing that if $r(u)<R_0$, then $f(r(u))>f(R_0)=1\ge\Phi(u)$, and \eqref{non vanishing bound bis} gives
\begin{eqnarray}\nonumber
  C\,\sqrt{\ell_0}&\ge&\int_{B_{r(u)}}|\Phi(u)-f(r(u))|^{n/(n-1)}\ge\om_n\,r(u)^n\,\big(f(r(u))-1\big)^{n/(n-1)}
  \\\nonumber
  &=&\om_n\,r(u)^n\,\big(f(r(u))-f(R_0)\big)^{n/(n-1)}=c(n)\,\Big(1-\big(r(u)/R_0)^{n-1}\Big)^{n/(n-1)}
  \\\nonumber
  &\ge& c(n)\,(R_0-r(u))^{n/(n-1)}\,,
\end{eqnarray}
as desired. If, instead $r(u)> R_0$, then by $\int_{\R^n}W(u)\le \e\,\AC(u)\le C$, $f(r(u))\in(0,1)$ and \eqref{Phi quadratic from below} (that is, $\Phi(b)-\Phi(a)\ge (b-a)^2/C$ if $0\le a\le b\le 1$), we deduce that
\begin{eqnarray*}
  C\,\e&\ge&\int_{B_{R_0}}W(u)\ge \int_{B_{R_0}}W\Big(\Phi^{-1}\big(f(r(u))\big)\Big)-C\,\int_{B_{R_0}}\big|u-\Phi^{-1}(f(r(u)))\big|
  \\
  &\ge&\int_{B_{R_0}}W\Big(\Phi^{-1}\big(f(r(u))\big)\Big)-C\,\int_{B_{R_0}}\big|\Phi(u)-f(r(u))\big|^2
  \\
  &\ge&\int_{B_{R_0}}W\Big(\Phi^{-1}\big(f(r(u))\big)\Big)-C\,\int_{B_{R_0}}\big|\Phi(u)-f(r(u))\big|^{n/(n-1)}\,,
\end{eqnarray*}
where in the last inequality we have just used $2\ge n/(n-1)$ and the fact that $\Phi(u)$ and $f(r(u))$ lie in $[0,1]$. Hence, by $B_{R_0}\subset B_{r(u)}$, \eqref{non vanishing bound bis} and $\om_n\,R_0^n=1$,
\[
W\big(\Phi^{-1}(f(r(u)))\big)\le C\,\big((\sqrt{\ell_0})^{(n-1)/n}+\e\big)\,.
\]
Now, $R_0<r(u)\le M_0$ implies $1>f(r(u))\ge f(M_0)\ge \de_0$ (provided we further decrease the value of $\de_0$). In particular, by $W(t)\ge (1-t)^2/C$ on $(\de_0,1)$ (which can be assumed as done with \eqref{W controlla V fin quasi ad uno}), we have
\[
C\,\big((\sqrt{\ell_0})^{(n-1)/n}+\e\big)\ge \Big(1-\Phi^{-1}\big(f(r(u))\big)\Big)^2\,.
\]
By \eqref{Phi near the wells}, we have
\[
1-\Phi^{-1}(s)\ge\frac{\sqrt{1-s}}C\,,\qquad\forall s\in(0,1)\,,
\]
thus concluding
\begin{eqnarray*}
C\,\big((\sqrt{\ell_0})^{(n-1)/n}+\e\big)&\ge& 1-f(r(u))=c(n)\,\big(R_0^{1-n}-r(u)^{1-n}\big)
  \\
  &\ge& \frac{c(n)}{r(u)^{n-1}}\,\Big(\Big(\frac{r(u)}{R_0}\Big)^{n-1}-1\Big)\ge \frac{c(n)}{M_0^{n-1}}\,(r(u)-R_0)\,.
\end{eqnarray*}
This completes the proof of \eqref{ru and R0}, and thus of \eqref{stima BV utile}.

\medskip

  \noindent {\it Step four}: We now consider $\{\e_j,v_j,\Err_j\}_j$ as in the statement, and begin the proof of the resolution result. We introduce the radius $R_j(t)$ by setting $v_j(R_j(t))=t$ for every $t$ in the range of $v_j$. In this step we prove that both $\de_0$ (defined in section \ref{subsection W}) and $1-\de_0$ belong to the range of each $v_j$, that
  \begin{eqnarray}
    \label{order of Rde0}
    &&
    3\,R_0\ge R_j(\de_0)\ge R_j(1-\de_0)\ge\frac{R_0}3\,,
    \\\label{order of Rde0 layer}
    &&\frac{\e_j}{C}\le R_j(\de_0)-R_j(1-\de_0)\le C\,\e_j\,,
  \end{eqnarray}
  and that
  \begin{equation}
    \label{critical sequence derivative in layer}
      -\frac{C}{\e_j}\le v_j'\le -\frac{1}{C\e_j}\qquad\mbox{on $(R_j(1-\de_0),R_j(\de_0))$}\,.
  \end{equation}
  In particular, the constants $b_j$ and $c_j$ introduced in \eqref{def of bj and cj} are well-defined, they satisfy
  \begin{equation}
    \label{def of bj and cj proof}
      c_j=R_j(\delta_0)-R_0\,,\qquad b_j= R_0-R_j(1-\delta_0)\,,
  \end{equation}
  and property \eqref{critical sq vj prime annulus} in the statement boils down to \eqref{critical sequence derivative in layer}.

  \medskip

  By step three, for $j$ large enough and considering that $v_j\in\RR_0$, we have
  \begin{equation}
    \label{proof rj 1}
      \int_{\R^n}|1_{B_{R_0}}-\Phi(v_j)|^{n/(n-1)}\leq C\, \big(\ell_0^{(n-1)/n}+\e_0\big)\,.
  \end{equation}
  By \eqref{proof rj 1}, if $\ell_0$ and $\e_0$ are small enough, then both $\de_0$ and $1-\de_0$ must belong to the range of each $v_j$. Now, if $R_j(\de_0)\le R_0$, then
  \[
  \int_{B_{R_0}\setminus B_{R_j(\de_0)}}|1_{B_{R_0}}-\Phi(v_j)|^{n/(n-1)}\ge \om_n\,(R_0^n-R_j(\de_0)^n)\,(1-\Phi(\de_0))^{n/(n-1)}\ge \frac{R_0^n-R_j(\de_0)^n}C\,,
  \]
  and $R_j(\de_0)\ge R_0/2$ follows by \eqref{proof rj 1} for $\ell_0$ and $\e_0$ small enough; if, instead, $R_j(\de_0)\ge R_0$, then
  \[
  \int_{B_{R_j(\de_0)}\setminus B_{R_0}}|1_{B_{R_0}}-\Phi(v_j)|^{n/(n-1)}\ge \om_n\,(R_j(\de_0)^n-R_0^n)\,\Phi(\de_0)^{n/(n-1)}\ge \frac{R_j(\de_0)^n-R_0^n}C\,,
  \]
  and $R_j(\de_0)\le 2\,R_0$ follows, again, for $\ell_0$ and $\e_0$ small enough; we have thus proved $R_0/2\le R_j(\de_0)\le 2\,R_0$. Since \eqref{critical sq L energy} implies $\mathcal{AC}_{\e_j}(v_j)\le C$ we also have
  \[
  C\,\e_j\ge\int_{\R^n}W(v_j)\ge\frac{R_j(\de_0)^n-R_j(1-\de_0)^n}C\ge\frac{R_j(\de_0)-R_j(1-\de_0)}C
  \]
  where in the last inequality we have used $R_j(\de_0)\ge R/2$. Thus, we have so far proved \eqref{order of Rde0} and the upper bound in \eqref{order of Rde0 layer}. Before proving the lower bound in \eqref{order of Rde0 layer}, we prove \eqref{critical sequence derivative in layer}. To this end, we multiply \eqref{critical sq ode} by $v_j'$, and then integrate over an arbitrary interval $(0,r)$, to get
  \begin{eqnarray}
  \label{auxiliary identity ode}
  \e_j^2\,\Big( (v_j')^2+2\,(n-1)\int_0^r \frac{v_j'(t)^2}t\,dt\Big)
  =W(v_j) -W(v_j(0)) -\e_j\,\int_0^rv_j\,(1-v_j)\,\Err_j \,v_j'\,.
  \end{eqnarray}
  By \eqref{critical sq bounds}, the right-hand side of \eqref{auxiliary identity ode} is bounded in terms of $n$ and $W$, so that \eqref{auxiliary identity ode} implies $\e_j^2\,(v_j')^2\le C$ on $(0,\infty)$; the lower bound in \eqref{critical sequence derivative in layer} then follows by $v_j'\le 0$. To obtain the upper bound in \eqref{critical sequence derivative in layer}, we multiply again \eqref{critical sq ode} by $v_j'$, but this time we integrate over $(r,\infty)$ for $r\in(R_j(1-\de_0),R_j(\de_0))$, thus obtaining
   \begin{eqnarray}
  \label{auxiliary identity ode 2}
  \e_j^2\,\Big(-v_j'(r)^2+2\,(n-1)\int_r^\infty \frac{v_j'(t)^2}t\,dt\Big)
  =-W(v_j(r))-\e_j\,\int_r^\infty v_j\,(1-v_j)\,\Err_j \,v_j'\,.
  \end{eqnarray}
  By $W(v_j(r))\ge \inf_{[\de_0,1-\de_0]}W\ge 1/C$, \eqref{critical sq bounds}, and the non-negativity of the integral on the left-hand side of \eqref{auxiliary identity ode 2}, we deduce that
  \[
  2\,\e_j^2\,v_j'(r)^2\ge W(v_j(r))-C\,\e_j\ge \frac1{C}\,,\qquad\forall r\in(R_j(1-\de_0),R_j(\de_0))\,,
  \]
  which, again by $v_j'\le 0$, implies the upper bound in \eqref{critical sequence derivative in layer}. To finally prove the lower bound in \eqref{order of Rde0 layer}, we notice that thanks to the lower bound in \eqref{critical sequence derivative in layer} we have
  \[
  \frac{C}{\e_j}\,(R_j(\de_0)-R_j(1-\de_0))\ge \int_{R_j(1-\de_0)}^{R_j(\de_0)} (-v_j')=1-2\,\de_0\,.
  \]
  We have completed the proofs of \eqref{order of Rde0}, \eqref{order of Rde0 layer} and \eqref{critical sequence derivative in layer}.

  \medskip

  \noindent {\it Step five}: We obtain sharp estimates for $v_j$ as $r\to\infty$: precisely, we prove that for every $r\ge R_j(\de_0)$ one has
  \begin{eqnarray}
  \label{veps decay at infinity}
  v_j(r) &\leq& C\,e^{-{(r- R_j(\de_0))/C\,\e_j}}\,,
  \\    \label{veps higherord at infinity}
  |v_j^{(k)}(r)|&\leq& \frac{C}{\e_j^k}\,\,e^{-(r- R_j(\de_0))/C\,\e_j}\,,\qquad k=1,2\,.
  \end{eqnarray}
  By \eqref{W near the wells}, \eqref{critical sq ode} and \eqref{critical sq bounds} we have that
  \begin{equation}
    \label{dec at inf 1}
      2\,\e_j^2\,\Big\{v_j''+(n-1)\,\frac{v_j'}r\Big\}\ge \frac{v_j}{C}-C\,\e_j\,v_j\ge \frac{v_j}C\,,\qquad\mbox{on $(R_j(\de_0),\infty)$}\,.
  \end{equation}
  Multiplying by $v_j'$ and integrating over $(r,\infty)\subset (R_j(\de_0),\infty)$ we obtain
  \begin{equation}
    \label{dec at inf 2}
  -\e_j^2\,v_j'(r)^2\ge-\frac{v_j(r)^2}C\,,\qquad\forall r\ge R_j(\de_0)\,.
  \end{equation}
  Plugging \eqref{dec at inf 2} into \eqref{dec at inf 1} we reabsorb the $v_j'/r$ term and obtain
  \begin{equation}
    \label{dec at inf 3}
      \e_j^2\,v_j''\ge \frac{v_j}{C_*}\,,\qquad\mbox{on $(R_j(\de_0),\infty)$}\,,
  \end{equation}
  for a universal constant $C_*$. We now notice that
  \[
  v_*(r)=\delta_0 \,e^{-(r- R_j(\de_0))/\sqrt{C_*}\,{\e_j}}\,,
  \]
  satisfies $\e_j^2\,v_*''=v_*/C_*$ and
  \[
  v_*(R_j(\de_0))=\de_0=v_j(R_j(\de_0))\,.
  \]
  Therefore $v_j\le v_*$ on $(R_j(\de_0),\infty)$, and \eqref{veps decay at infinity} follows. The case $k=1$ of \eqref{veps higherord at infinity} follows from \eqref{veps decay at infinity} and \eqref{dec at inf 2}. The case $k=2$ of \eqref{veps higherord at infinity} follows by noticing that \eqref{critical sq ode}, \eqref{critical sq bounds}, \eqref{dec at inf 2} and \eqref{W near the wells} imply
  \[
  \e_j^2\,|v_j''|\le C\,\Big(\e_j^2\,|v_j'|+|W'(v_j)|+\e_j\,v_j\Big)\le C\,\Big(\e_j^2\,|v_j'|+v_j\Big)\le C\,v_j\,,
  \]
  and then by using \eqref{veps decay at infinity} again.

  \medskip

  \noindent {\it Step six}: We obtain sharp estimates for $v_j(r)$ when $r\to 0^+$: precisely, we prove that for every $r\le R_j(1-\de_0)$ one has
  \begin{eqnarray}
      \label{veps decay at zero}
  1-v_j(r) &\leq& C\,\,e^{-{(R_j(1-\de_0)-r)/C\,\e_j}}\,,
  \\    \label{veps first order at zero}
  |v_j'(r)|&\leq& C\,\min\Big\{\frac{r}{\e_j^2},\frac1{\e_j}\Big\}\,\,e^{-(R_j(1-\de_0)-r)/C\,\e_j}\,,
  \\    \label{veps second order at zero}
  |v_j''(r)|&\leq& \frac{C}{\e_j^2}\,\,e^{-(R_j(1-\de_0)-r)/C\,\e_j}\,.
  \end{eqnarray}
  To this end, it is convenient to recast \eqref{critical sq ode} in terms of $w_j=1-v_j$, so that
   \begin{equation}
   \label{eq ODE at 0}
      2\,\e_j^2\,\Big\{ w_j''+(n-1)\,\frac{w_j'}{r}\Big\}=-W'(1-w_j)+\e_j\,\,w_j\, (1-w_j)\,\Err_j\,.
   \end{equation}
  By \eqref{W near the wells} and \eqref{critical sq bounds}, if $r\le R_j(1-\de_0)$, then
  \begin{equation}
    \label{error estimate z}
      -W'(1-w_j)+\e_j\,\,w_j\, (1-w_j)\,\Err_j\le C\,(1-w_j)
  \end{equation}
  so that \eqref{eq ODE at 0} implies in particular
  \begin{equation}
      \label{eq ineq at 0}
   2\,\e_j^2\,\Big\{ w_j''+(n-1)\,\frac{w_j'}{r}\Big\}\leq C\,w_j\,,\qquad\mbox{on $(0,R_j(1-\de_0))$}\,.
   \end{equation}
   Multiplying by $w_j'\ge0$ and integrating on $(0,r)\subset(0,R_j(1-\delta_0))$ we deduce
   \[
   \e_j^2\Big\{ w_j'(r)^2 +\int_0^r \frac{(w_j')^2}{t}\Big\}\leq C\, (w_j(r)^2- w_j(0)^2)\leq C\,w_j(r)^2\,,
   \]
   that is,
   \begin{equation}
     \label{eq ineq at 0 1}
         \e_j\, w_j'\leq C\, w_j\,,\qquad\mbox{on $(0,R_j(1-\de_0))$}\,.
   \end{equation}
   Combining \eqref{eq ineq at 0 1} with \eqref{eq ODE at 0}, \eqref{W near the wells} and \eqref{critical sq bounds}, we find that
   \begin{eqnarray*}
   2\,\e_j^2\, w_j''+C\,\e_j\, w_j&\geq&  2\,\e_j^2\,\Big\{ w_j''+\frac{n-1}r\, w_j'\Big\}
   \\
   &=&-W'(1-w_j)+\e_j\,\,w_j\, (1-w_j)\,\Err_j
   \ge\frac{w_j}C\,-C\,\e_j\,w_j\,,
   \end{eqnarray*}
   on $[R_0/4,R_j(1-\de_0))$, so that, for $j$ large enough and for a constant $C_*$ depending on $n$ and $W$ only, we have
   \begin{equation}
   \label{eq comparison ODE}
   \e_j^2\, w_j''\geq \frac{w_j}{C_*}\,,\qquad\mbox{on $[R_0/4,R_j(1-\de_0))$}\,.
   \end{equation}
   Correspondingly to $C_*$, we introduce the barrier
   \[
   w_*(r)= \delta_0\,\Big\{e^{((R_0/4)-r)/\sqrt{C_*\,\e_j^2}}+ e^{(r-R_j(1-\delta_0))/\sqrt{C_*\,\e_j^2}}\Big\}\,,\qquad r>0\,.
   \]
   By the monotonicity of $w_j$ and by $R_j(1-\de_0)\ge R_0/3$ (recall \eqref{order of Rde0}),
   \begin{eqnarray*}
     &&w_*(R_0/4)\geq \delta_0=w_j(R_j(1-\delta_0))\geq w_j(R_0/4)\,,
     \\
     &&w_*(R_j(1-\delta_0))\geq \delta_0 = w_j(R_j(1-\delta_0))\,,
     \\
     &&\e_j^2\, w_*''=\frac{w_*}{C_*}\qquad\mbox{on $[0,\infty)$}\,.
   \end{eqnarray*}
   We thus find $w_j\le w_*$ on $[R_0/4,R_j(1-\de_0))$, that is, for every $R_0/4\le r\le R_j(1-\de_0)$,
   \begin{equation}
   \label{eq init bound origin}
   1-v_j(r)\leq \delta_0\,\Big\{e^{((R_0/4)-r)/\sqrt{C_*\,\e_j^2}}+ e^{(r-R_j(1-\delta_0))/\sqrt{C_*\,\e_j^2}}\Big\}\,.
   \end{equation}
   By testing \eqref{eq init bound origin} with
   \[
   r_* = \frac{R_0/4+R_0/3}2
   \]
   and exploiting the monotonicity of $v_j$, we find that for $r \in (0, r_*]$
   \begin{equation}
   \label{eq exponential dec at 0}
   1-v_j(r)\leq \delta_0\, e^{-1/C\,\e_j}\qquad\forall r\in(0,r_*]\,,
   \end{equation}
   (thus obtaining the crucial information that, for $j$ large enough and for every $k\in\N$, $\|1-v_j\|_{C^0[0,r_*]}={\rm o}(\e_j^k)$ as $j\to\infty$). At the same time, for $r_*\le r\le R_j(1-\de_0)$, the second exponential in \eqref{eq init bound origin} is bounded from below in terms of a universal constant, while the first exponential is bounded from above by $e^{-1/C\,\e_j}$, so that \eqref{eq init bound origin} and \eqref{eq exponential dec at 0} can be combined into
   \[
   1-v_j(r)\leq C\,e^{-(R_j(1-\delta_0)-r)/C\,\e_j}\,,\qquad\forall r\in(0,R_j(1-\de_0)]\,,
   \]
   that is \eqref{veps decay at zero}. By combining \eqref{veps decay at zero} and \eqref{eq ineq at 0 1} we also find
   \begin{equation}
   \label{eq second bound derivative origin}
   -v_j'(r) \leq \frac{C}{\e_j}\,e^{-(R_j(1-\delta_0)-r)/C\,\e_j}\,,\qquad\forall r\in(0,R_j(1-\de_0)]\,,
   \end{equation}
   which is half of the estimate for $|v_j'|$ in \eqref{veps first order at zero}. Multiplying \eqref{eq ineq at 0} by $r^{n-1}$ we find
   \[
   2\,\e_j^2\,\big(r^{n-1}\, w_j'\big)'\leq C\,r^{n-1}\,w_j\,,\qquad\forall r\in(0,R_j(1-\de_0)]\,,
   \]
   which we integrate over $(0,r)\subset(0,R_j(1-\de_0))$ to conclude that
   \begin{equation*}
   \e_j^2\,r^{n-1}\, (-v_j'(r))\leq C\,\int_0^r w_j(t)\,t^{n-1}dt\leq C\, (1-v_j(r))\,r^n\,,\qquad\forall r\in(0,R_j(1-\de_0)]\,;
   \end{equation*}
   in particular, by combining this last inequality with \eqref{veps decay at zero} we find
   \[
   -v_j'(r) \leq C\,\frac{r}{\e_j^2}\,e^{-(R_j(1-\delta_0)-r)/C\,\e_j}\,,\qquad\forall r\in(0,R_j(1-\de_0)]\,,
   \]
   that is the missing half of \eqref{veps first order at zero}. Finally, by  \eqref{eq ODE at 0} with \eqref{error estimate z} we find
   \[
   \e_j^2\,|v_j''|\le C\,\Big\{(1-v_j)+\frac{|v_j'|}r\Big\}\qquad\mbox{on $(0,R_j(1-\de_0))$}\,,
   \]
   and then \eqref{veps second order at zero} follows from \eqref{veps decay at zero} and \eqref{veps first order at zero}.

   \medskip

   \noindent {\it Step seven}: We now improve the first set of inequalities in \eqref{order of Rde0}, and show that
   \begin{equation}
     \label{tutto esp vicino}
     R_0-C\,\e_j\le R_j(1-\de_0)<R_j(\de_0)\le R_0+C\,\e_j\,.
   \end{equation}
   Let us set
   \[
   \a_j=\int_{B_{R_j(1-\de_0)}}V(v_j)\,,\qquad\b_j=\int_{B_{R_j(\de_0)}\setminus B_{R_j(1-\de_0)}}V(v_j)\,,
   \qquad\g_j=\int_{B_{R_j(\de_0)^c}}V(v_j)\,.
   \]
   By \eqref{V near the wells}, \eqref{veps decay at zero} and \eqref{order of Rde0} we have
   \begin{eqnarray*}
     \big|\a_j-\om_n\,R_j(1-\de_0)^n\big|&=&\int_{B_{R_j(1-\de_0)}}1-V(v_j)\le C\,\int_{B_{R_j(1-\de_0)}}(1-v_j)^2
     \\
     &\le&C\,\int_{B_{R_j(1-\de_0)}}e^{-{(R_j(1-\de_0)-|x|)/C\,\e_j}}\,dx
     \\
     &=&C\,\int_0^{R_j(1-\de_0)}\,e^{-{(R_j(1-\de_0)-r)/C\,\e_j}}\,r^{n-1}\,dr
     \\
     &=&C\,\e_j\,\int_{-R_j(1-\de_0)/\e_j}^0\,e^{s/C}\,(R_j(1-\de_0)+\e_j\,s)^{n-1}\,ds\le C\,\e_j\,.
   \end{eqnarray*}
   Similarly, by \eqref{V near the wells}, \eqref{order of Rde0} and \eqref{veps decay at infinity} we find
   \begin{eqnarray*}
     |\g_j|&=&\int_{B_{R_j(\de_0)}^c}V(v_j)\le C\,\int_{B_{R_j(\de_0)^c}}v_j^{2\,n/(n-1)}
     \le C\,\int_{R_j(\de_0)}^\infty\,e^{-{(r-R_j(\de_0))/C\,\e_j}}\,r^{n-1}\,dr
     \\
     &=&C\,\e_j\,\int_0^\infty\,e^{-s/C}\,(R_j(\de_0)+\e_j\,s)^{n-1}\,ds\le C\,\e_j\,.
   \end{eqnarray*}
Finally, thanks to \eqref{order of Rde0},
   \begin{eqnarray*}
     |\beta_j|&=&\int_{B_{R_j(\de_0)}\setminus B_{R_j(1-\de_0)}}V(v_j)\le C\,(R_j(\de_0)-R_j(1-\de_0))\le C\,\e_j\,.
   \end{eqnarray*}
   Combining the estimates for $\a_j$, $\b_j$ and $\g_j$ with the fact that
   \[
   \om_n\,R_0^n=1=\int_{\R^n}V(v_j)=\a_j+\b_j+\g_j\,,
   \]
   we conclude that
   \[
   C\,\e_j\ge \om_n\,|R_0^n-R_j(1-\de_0)^n|\le \frac{|R_0-R_j(1-\de_0)|}C\,,
   \]
   so that \eqref{tutto esp vicino} follows by \eqref{order of Rde0}.

   \medskip

   \noindent {\it Step eight}: We conclude the proof of the theorem: \eqref{critical sequence derivative in layer}, \eqref{def of bj and cj proof} and \eqref{tutto esp vicino} imply \eqref{def of bj and cj} and \eqref{critical sq vj prime annulus}, as well as
   \begin{equation}
     \label{critical sq ring basic weaker}
     |b_j|\,,|c_j|\le C\,\e_j\,,
   \end{equation}
   which is a weaker form of \eqref{critical sq ring basic}; \eqref{veps decay at infinity} and \eqref{veps higherord at infinity} imply \eqref{critical sq vj estimates on 0 R}, while \eqref{veps decay at zero}, \eqref{veps first order at zero}, and \eqref{veps second order at zero} imply \eqref{critical sq vj estimates on R infinity}. We are thus left to prove the full form of \eqref{critical sq ring basic} (which includes a positive lower bound in the from $\e_j/C$ for both $b_j$ and $c_j$), as well as \eqref{critical sq resolution of vj}: that is, we want to show that if $v_j$ satisfies \eqref{critical sq volume 1}, \eqref{critical sq L energy}, \eqref{critical sq ode} and \eqref{critical sq bounds}, then, for every $x\in\R^n$ and $j$ large enough, we have
   \begin{equation}
    \label{critical sq resolution of vj proof}
      v_j(x)=z_{\e_j}(x)+f_j\Big(\frac{|x|-R_0}{\e_j}\Big)=\eta\Big(\frac{|x|-R_0}{\e_j}-\tau_j\Big)+f_j\Big(\frac{|x|-R_0}{\e_j}\Big)\,,
  \end{equation}
  with functions $f_j\in C^2(I_j)$ such that
  \begin{equation}
    \label{www for fj}
    |f_j(s)|\le C\,\e_j\,e^{-|s|/C}\,,\qquad\forall s\in I_j=(-R_0/\e_j,\infty)\,,
  \end{equation}
  and with $\tau_j=\tau_{\e_j}$ for $\tau_\e$ defined by \eqref{def of zeta eps} and \eqref{def of R and taueps}. In fact, \eqref{critical sq resolution of vj proof} and \eqref{www for fj} imply the full form of \eqref{critical sq ring basic}: for example, combined with \eqref{critical sq vj prime annulus} and \eqref{lemma tau}, they give
  \begin{eqnarray*}
  C\,\frac{b_j}{\e_j}&\ge&\int_{R_0-b_j}^{R_0}(-v_j')=v_j(R_0-b_j)-v_j(R_0)=(1-\de_0)-\eta(-\tau_j)-f_j(0)
  \\
  &\ge& 1-\de_0-\eta(-\tau_0)-C\,\e_j
  \end{eqnarray*}
  where the latter quantity is positive provided $j$ is large enough and we further decrease the value of $\de_0$ to have $\de_0<1-\eta(-\tau_0)$.

  \medskip

  We can thus focus on \eqref{critical sq resolution of vj proof} and \eqref{www for fj}, which we recast by looking at the functions
  \[
  \eta_j(s)=v_j(R_0+\e_j\,s)\,,\qquad s\in I_j\,,
  \]
  in terms of which $f_j(s)=\eta_j(s)-\eta(s-\tau_j)$. Thus, our goal becomes proving that
   \begin{equation}
    \label{www for etaj}
    |\eta_j(s)-\eta(s-\tau_j)|\le C\,\e_j\,e^{-|s|/C}\,,\qquad\forall s\in I_j\,.
  \end{equation}
  We start noticing that, by \eqref{critical sq vj prime annulus}, \eqref{critical sq vj estimates on 0 R} and \eqref{critical sq vj estimates on R infinity}, we have
  \begin{eqnarray}
    \label{etaj 1}
    &&C\ge-\eta_j'(s)\ge\frac1{C}\,,\hspace{2.7cm}\qquad\forall s\in(-b_j/\e_j,c_j/\e_j)\,,
    \\
    \label{etaj 2}
    &&\eta_j^{(k)}(s)\le C\,e^{-s/C}\,,\hspace{2.7cm}\qquad\forall s\in(c_j/\e_j,\infty)\,,k=0,1,2,
   \end{eqnarray}
   \begin{equation}
   \label{etaj 3}
    \left\{\begin{split}
      &(1-\eta_j(s))+|\eta''_j(s)|\le C\,e^{s/C}\,,
      \\
      &|\eta_j'|\le C\,\min\Big\{\frac{R_0+\e_j\,s}{\e_j}\,,1\Big\}\,e^{s/C}\,,
    \end{split}\right .
    \qquad\forall s\in(-R_0/\e_j,-b_j/\e_j)\,,     
   \end{equation}
  (while the analogous estimates for $\eta$ are found in \eqref{eta decay} and \eqref{eta decay first and second derivative}). In order to estimate $f_j(s)=\eta_j(s)-\eta(s-\tau_j)$, we introduce
  \[
  g_j(s)=\eta_j(s)-\eta(s-t_j)\,,
  \]
  for $t_j$ defined by the identity
  \begin{equation}
    \label{def of tj}
    \eta\big(-(b_j/\e_j)-t_j\big)=1-\delta_0\,.
  \end{equation}
  (Notice that the definition is well-posed by $\eta'<0$ and $\eta(\R)=(0,1)$.) We claim that the proof of \eqref{www for fj} can be reduced to that of
  \begin{equation}\label{eq bound first remainder}
  |g_j(s)|\le C\,\e_j\,e^{-|s|/C}\,,\qquad\forall s\in I_j\,.
  \end{equation}
  Indeed, by \eqref{critical sq volume 1}, if \eqref{eq bound first remainder} holds, then we are in the position to apply step one, and deduce from \eqref{lemma tau} that $|t_j-\tau_0|\le C\,\e_j$. Having also (by the same argument) $|\tau_j-\tau_0|\le C\,\e_j$, we deduce that
  \[
  |\tau_j-t_j|\le C\,\e_j\,,
  \]
  that we exploit in combination with \eqref{etaj 2} and \eqref{etaj 3} to deduce
  \begin{eqnarray*}
    |f_j(s)-g_j(s)|&=&|\eta(s-t_j)-\eta(s-\tau_j)|\leq C\,\int_0^1 \big|\eta'\big(s-\tau_j-t\,(t_j-\tau_j)\big)\big|\,dt
    \\
    &\leq& C\,\e_j\, e^{-|s|/C}\,,\qquad\forall s\in I_j\,.
  \end{eqnarray*}
  We are thus left to prove \eqref{eq bound first remainder}. To this end, we preliminarily notice that, since $\eta_j(-b_j/\e_j)=v_j(R_0-b_j)=1-\de_0$, the definition of $t_j$ is such that
  \begin{equation}
    \label{gj zero value}
    g_j\big(-b_j/\e_j\big)=0\,.
  \end{equation}
  Moreover, by the decay properties \eqref{eta decay} of $\eta$ and by $|b_j|\le C\,\e_j$, \eqref{def of tj} implies
  \begin{equation}
    \label{tj is bounded}
    |t_j|\le C\,.
  \end{equation}
  We now divide the proof of \eqref{eq bound first remainder} in three separate arguments:

  \medskip

  \noindent {\it We prove \eqref{eq bound first remainder} for $|s|\ge C\,\log(1/\e_j)$}: This is trivial from the decay properties of $\eta$ and $\eta_j$. Indeed, by \eqref{eta decay}, \eqref{tj is bounded}, \eqref{etaj 2} and \eqref{etaj 3} we find that
  \begin{equation}
    \label{exterior gj pre}
      |g_j(s)|\le K_1\, e^{-|s|/K_1}\,,\qquad\forall  s\in I_j\,.
  \end{equation}
  for a universal constant $K_1$. In particular, we trivially have
  \begin{equation}
    \label{exterior gj}
      |g_j(s)|\le K_1\,\e_j\,e^{-|s|/2\,K_1}\,,\qquad\forall s\in I_j\,,|s|\ge 2\,K_1\log\Big(\frac1{\e_j}\Big)\,.
  \end{equation}
  We will later increase the value of $K_1$ in \eqref{exterior gj pre} so that \eqref{has expo decay} below holds too.

  \medskip

  \noindent {\it We prove \eqref{eq bound first remainder} on arbitrary compact subsets of $I_j$}: More precisely, we show that for every $K>0$ we can find $C_K=C_K(n,W)$ (that is, a constant that depends on $n$, $W$ and $K$ only) such that
  \begin{equation}
    \label{eq bound near the origin}
    |g_j(s)|\le C_K\,\e_j\,,\qquad\forall s\in I_j\,, |s|\le K\,.
  \end{equation}
  To this end, setting $\Err_j^*(s)=\Err_j(R_0+\e_j\,s)$, we deduce from \eqref{critical sq ode} that $\eta_j$ satisfies the ODE
  \begin{equation}\label{eq main modified ODE}
  2\,\eta_j''+ 2\,\e_j\,\frac{n-1}{R_0+\e_j\,s}\,\eta_j'=W'(\eta_j)- \e_j\,\eta_j\,(1-\eta_j)\,\Err_j^*\qquad\mbox{on $I_j$}\,.
  \end{equation}
  Multiplying \eqref{eq main modified ODE} by $-\eta_j'$ and integrating over $(s,\infty)$ we find
  \begin{equation}\label{eq first order ODE}
  \eta_j'(s)^2 - 2\,\e_j\,(n-1)\int_s^\infty\frac{\eta_j'(t)^2}{R_0+\e_j\,t}\,dt=W(\eta_j(s))
  +\e_j\,\int_s^\infty \, \eta_j\,(1-\eta_j)\,\eta_j'\,\Err_j^*\,.
  \end{equation}
  Since  $\eta'(s-t_j)^2=W(\eta(s-t_j))$ for every $s\in\R$, we find that
  \begin{eqnarray}\label{eq ODE difference}
   \eta_j'(s)^2-\eta'(s-t_j)^2&=& W(\eta_j(s))-W(\eta(s-t_j))+\e_j\,L_j(s)\,,
   \\
   \nonumber
   \mbox{where}&&L_j(s)=\int_s^\infty \Big(2\,(n-1)\,\frac{\eta_j'(t)^2}{R_0+\e_j\, t}+\eta_j\,(1-\eta_j)\,\eta_j'\,\Err_j^*\Big)\, dt\,.
  \end{eqnarray}
  Setting
  \[
  \ell_j(s)=\frac{W(\eta_j(s))-W(\eta(s-t_j))}{\eta_j(s)-\eta(s-t_j)}\,,\qquad d_j(s)=\eta_j'(s)+\eta'(s-t_j)\,,\qquad \Gamma_j(s)=\frac{\ell_j(s)}{d_j(s)}\,,
  \]
  and noticing that $d_j<0$ on $I_j$, \eqref{eq ODE difference} takes the form
  \begin{eqnarray}\label{eq ODE Gronwall}
   g_j'(s)-\Gamma_j(s)\,g_j(s)=\frac{\e_j\,L_j(s)}{d_j(s)}\,,\qquad\forall s\in I_j\,.
  \end{eqnarray}
  Multiplying \eqref{eq ODE Gronwall} by $\exp(-\int_0^s\Gamma_j)$, integrating over an interval $(-b_j/\e_j,s)$, and taking into account \eqref{gj zero value}, we find
  \begin{eqnarray}\label{key}
  g_j(s)\,e^{-\int_0^s\Gamma_j}=
   \e_j\int_{- b_j/\e_j}^s\,\frac{e^{-\int_0^t\Gamma_j}}{d_j(t)}\,L_j(t)\,dt\,,\qquad\forall s\in I_j\,.
  \end{eqnarray}
  We now notice that by
  \eqref{critical sq bounds}, \eqref{etaj 2} and \eqref{etaj 3},
  \begin{equation}
    \label{Lj bounds}
      |L_j(s)|\le C\,\min\{1,e^{-s/C}\}\,,\qquad\forall s\in I_j\,.
  \end{equation}
  Moreover, by $\Lip\,W\le C$ we have $|\ell_j|\le C$ on $I_j$, while $\eta_j'\le 0$ and \eqref{tj is bounded} give
  \begin{equation}
    \label{key 2}
      d_j(s)\le \eta'(s-t_j)\le-\frac1{C_K}\,,\qquad\forall |s|\le K\,,
  \end{equation}
  and, in particular, $|\Gamma_j(s)|\le C_K$ for $|s|\le K$. Now, assuming without loss of generality that $K$ is large enough to entail $K\ge |b_j|/\e_j$ (as we can do since $|b_j|\le C\,\e_j$ for a universal constant $C$), then \eqref{key}, \eqref{Lj bounds}, \eqref{key 2} and $|\Gamma_j|\le C_K$ on $[-K,K]$ combined give \eqref{eq bound near the origin}.

  \medskip

  \noindent {\it Finally, we prove \eqref{eq bound first remainder} in the remaining case}: Having in mind \eqref{exterior gj} and \eqref{eq bound near the origin} we are left to prove the existence of a sufficiently large universal constant $K_2$ such that \eqref{eq bound first remainder} holds (provided $j$ is large enough) for every $s\in I_j$ with $K_2\le |s|\le 2\,K_1\,\log(1/\e_j)$. To this end, we start by subtracting $2\,\eta''=W(\eta)$ from \eqref{eq main modified ODE}, and obtain
  \begin{equation}\label{eq second order ODE}
  2g_j'' -m_j\,g_j= \e_j\,\Big\{\eta_j\,(1-\eta_j)\,\Err_j^*\,- 2\,(n-1)\,\frac{\eta_j'}{R_0+\e_j\,s}\Big\}\,,\qquad\forall s\in I_j\,,
  \end{equation}
  where
  \[
  m_j(s)= \frac{W'(\eta_j(s))-W'(\eta(s-t_j))}{\eta_j(s)-\eta(s-t_j)}\,,\qquad s\in I_j\,.
  \]
  The coefficient $m_j$ is uniformly positive: indeed, the decay properties of $\eta$ and $\eta_j$ at infinity, combined with $|t_j|\le C$, imply the existence of a universal constant $K_2$ such that if $|s|\ge K_2$, then $\eta_j(s)$ and $\eta(s-t_j)$ are both at distance at most $\de_0$ from $\{0,1\}$: and since $W''\ge 1/C$ on $(0,\de_0)\cup(1-\de_0,1)$ by \eqref{W near the wells}, we conclude that, up to further increase the value of $K_2$,
  \begin{equation}
   \label{mj positive}
   m_j(s)\ge\frac1{K_2}\qquad\forall s\in I_j\,,|s|\ge K_2\,.
  \end{equation}
  At the same time, the right-hand side of \eqref{eq second order ODE} has exponential decay: indeed, by \eqref{critical sq bounds}, \eqref{etaj 1}, \eqref{etaj 2} and \eqref{etaj 3}, if $|s|\le\log(1/\e_j)$, $s\in I_j$, then we get
  \begin{equation}
    \label{has expo decay}
    \Big|\eta_j\,(1-\eta_j)\,\Err_j^*\,- 2\,(n-1)\,\frac{\eta_j'}{R_0+\e_j\,s}\Big|\le K_1\,\e_j\,e^{-|s|/K_1}\,,
  \end{equation}
  up to further increase the value of the universal constant $K_1$ introduced in \eqref{exterior gj}. Let us thus consider
  \[
  g_*(s)=C_1\,\e_j\,e^{-|s|/\sqrt{2\,C_2}}\,,\qquad s\in\R\,,
  \]
  for $C_1$ and $C_2$ universal constants to be determined. By combining \eqref{eq second order ODE} with \eqref{mj positive} and \eqref{has expo decay} we find that, if $s\in I_j$ with $K_2\le |s|\le 2\,K_1\,\log(1/\e_j)$, then
  \begin{eqnarray}\nonumber
    2\,(g_j-g_*)''-m_j\,(g_j-g_*)&\ge&m_j\,g_*-2\,g_*'' -K_1\,\e_j\,e^{-|s|/K_1}
    \\\nonumber
    &\ge&\Big(\frac1{K_2}-\frac{1}{C_2}\Big)\,g_*-K_1\,\e_j\,e^{-|s|/K_1}
    \\\nonumber
    &=&\e_j\,\Big\{\frac{C_1}{K_1}\,\Big(\frac1{K_2}-\frac{1}{C_2}\Big)\,e^{[(1/K_1)-(1/\sqrt{2\,C_2})]\,|s|}-1\Big\}\,K_1\,e^{-|s|/K_1}\,,
  \end{eqnarray}
  where the latter quantity is non-negative for every $|s|\ge K_2$ provided
  \begin{equation}
    \label{C1 C2 choice}
      C_1\ge 3\,K_1\,K_2\,e^{-K_0/2\,K_1}\,,\qquad C_2\ge\max\{2\,K_2,2\,K_1^2\}\,.
  \end{equation}
  At the same time, by \eqref{exterior gj},
  \[
  \big|g_j\big(\pm\,2\,K_1\log(1/\e_j)\big)\big|\le K_1\,\e_j^2\,,
  \]
  while $C_2\ge 2\,K_1^2$ gives
  \[
  g_*\big(\pm2\,K_1\,\log(1/\e_j)\big)=C_1\,\e_j\,e^{-2\,K_1\,\log(1/\e_j)/\sqrt{2\,C_2}}\ge C_1\,\e_j^2\,.
  \]
  Upon further requiring $C_1\ge K_1$ we thus have
  \begin{equation}
    \label{at the logs}
    g_*(s)\ge|g_j(s)|\qquad\mbox{at $s=\pm\,2\,K_1\,\log(1/\e_j)$}\,.
  \end{equation}
  Similarly, by \eqref{eq bound near the origin},
  \[
  |g_j(\pm K_2)|\le C_{K_2}\,\e_j\,,
  \]
  while $C_\ge 2\,K_2$ gives
  \[
  g_*(\pm K_2)=C_1\,\e_j\,e^{-K_2/\sqrt{2\,C_2}}\ge C_1\,\e_j\,e^{-\sqrt{K_2}/2}\,.
  \]
  Upon requiring that $C_1\ge C_{K_2}\,e^{\sqrt{K_2}/2}$, we find that
  \begin{equation}
    \label{at the K2s}
    g_*(s)\ge|g_j(s)|\qquad\mbox{at $s=\pm\,K_2$}\,.
  \end{equation}
  In summary, we have proved that if $K_1$ satisfies \eqref{exterior gj pre} and \eqref{has expo decay}, $K_2$ satisfies \eqref{mj positive}, and $C_1$ and $C_2$ are taken large enough in terms of $K_1$ and $K_2$, then \eqref{at the logs} and \eqref{at the K2s} holds. In particular, $h_j=g_j-g_*$ is non-positive on the boundary of the intervals $[-2\,K_1\,\log(1/\e_j),-K_2]$ and $[K_2,2\,K_1\,\log(1/\e_j)]$, with $h_j''-m_j\,h\ge0$, $m_j\ge0$, on those intervals thanks to \eqref{C1 C2 choice} and \eqref{mj positive}; correspondingly, by the maximum principle, $h_j\le 0$ there, that is,
  \[
  g_j(s)\le C_1\,\e_j\,e^{-|s|/\sqrt{2\,C_2}}\,,\qquad\forall s\in I_j\,,K_2\le |s|\le 2\,K_1\,\log(1/\e_j)\,.
  \]
  To get the matching lower bound we notice that, again by \eqref{has expo decay},
  \[
  (-g_*-g_j)''-m_j\,(-g_*-g_j)\ge m_j\,g_*-g_*'' -K_1\,\e_j\,e^{-|s|/K_1}
  \]
  so that, by the same considerations made before, the maximum principle can be applied to $k_j=-g_*-g_j$ on $[-2\,K_1\,\log(1/\e_j),-K_2]\cup[K_2,2\,K_1\,\log(1/\e_j)]$, to deduce $g_j\ge-g_*$. This completes the proof of \eqref{eq bound first remainder}.
\end{proof}

\section{Strict stability among radial functions}\label{section strict stability radial profiles} In this section we are going to exploit the resolution result in Theorem \ref{theorem asymptotics of minimizers} to deduce a stability estimate for $\psi(\e)$ on radial (not necessarily decreasing) functions. More precisely, we shall prove the following statement.

\begin{theorem}[Flugede type estimate]\label{theorem fuglede estimate} If $n\ge 2$ and $W\in C^{2,1}[0,1]$ satisfies \eqref{W basic} and \eqref{W normalization}, then there exist universal constants $\de_0$ and $\e_0$ with the following property: if $\e<\e_0$, $u_\e\in\RR_0$ is a minimizer of $\psi(\e)$, and $u\in H^1(\R^n;[0,1])$ is radial and such that
\begin{eqnarray}
  \label{fu hp 1}
&&\int_{\R^n} V(u)=1\,,
\\
\label{fu hp 2}
&&\int_{\R^n}(u-u_\e)^2\le C\,\e\,,
\\
\label{fu hp 3}
&&\|u-u_\e\|_{L^\infty(\R^n)}\le\de_0\,,
\end{eqnarray}
then, setting $h=u-u_\e$,
\begin{equation}\label{fuglede 0}
\AC(u)-\psi(\e)\geq \frac1{C}\,\int_{\R^n}\e\,|\nabla h|^2+\frac{h^2}\e\,.
\end{equation}
\end{theorem}

Before entering into the proof of Theorem \ref{theorem fuglede estimate}, we show how it can be used to improve on the conclusions of Theorem \ref{theorem existence solutions}. In particular, it gives the uniqueness of minimizers in $\psi(\e)$ and, together with the resolution result in Theorem \ref{theorem asymptotics of minimizers}, allows to compute the precise asymptotic behavior of $\psi(\e)$ and $\l(\e)$ up to second and first order in $\e\to 0^+$ respectively. Notice in particular that \eqref{sharp expansion lambda eps} sharply improves \eqref{lambda estimate eps}.

\begin{corollary}\label{corollary optimal energy and lagrange multiplier}
  If $n\ge 2$ and $W\in C^{2,1}[0,1]$ satisfies \eqref{W basic} and \eqref{W normalization}, then there exists a universal constant $\e_0$ such that, if $\e<\e_0$, then $\psi(\e)$ admits a unique minimizer (modulo translations). In particular, for every $\e<\e_0$, $\l(\e)$ is unambiguously defined as the Lagrange multiplier of the unique minimizer $u_\e\in\RR_0$ of $\psi(\e)$ by the identity \eqref{lambda formula}, i.e.
\begin{equation}
\label{lambda formula again}
\l(\e)=\Big(1-\frac1n\Big)\,\psi(\e)+\,\frac1n\,\Big\{\frac1\e\int_{\R^n}W(u_\e)-\e\,\int_{\R^n}|\nabla u_\e|^2\Big\}\,.
\end{equation}
   Finally, $\e\in(0,\e_0)\mapsto\l(\e)$ is continuous and the following expansions hold as $\e\to 0^+$,
  \begin{eqnarray}
  \label{sharp expansion psi eps}
  \psi(\e)&=&2\,n\,\omega_n^{1/n}+2\,n\,(n-1)\,\omega_n^{2/n}\,\k_0\,\e +{\rm O}(\e^2)
  \\\label{sharp expansion lambda eps}
  \l(\e)&=&2\,(n-1)\,\omega_n^{1/n}+{\rm O}(\e)\,,
  \end{eqnarray}
  where $\k_0=\tau_0+\tau_1=\int_\R[\eta'\,V'(\eta)+ W(\eta)]\,s\,ds$.
\end{corollary}

\begin{proof}[Proof of Corollary \ref{corollary optimal energy and lagrange multiplier}] {\it Step one}: Let $\e \in (0,\e_0)$ and let $u_\e$ and $v_\e$ be two minimizers of $\psi(\e)$, so that, up to translations, $u_\e,v_\e\in\RR_0^*$ thanks to Theorem \ref{theorem existence solutions}. By Theorem \ref{theorem asymptotics of minimizers}, if we set $h_\e=v_\e-u_\e$, then
\[
h_\e(x)=f_\e\Big(\frac{|x|-R_0}{\e}\Big)\,,
\]
where $f_\e\in C^2(-R_0/\e,\infty)$, and
\begin{equation}
  \label{co fu 1}
  |f_\e(s)|\le C\,\e\,e^{-s/C}\,,\qquad\forall s\ge-R_0/\e\,.
\end{equation}
We thus see that $u=v_\e$ satisfies \eqref{fu hp 1} and \eqref{fu hp 3}. Moreover, by \eqref{co fu 1},
\[
\int_{\R^n}h_\e^2=n\,\om_n\,\int_{-R_0/\e}^\infty\,f_\e(s)^2\,(R_0+\e\,s)^{n-1}\,ds\le C\,\e^2\,,
\]
so that \eqref{fu hp 2} holds too. We can thus apply \eqref{fuglede 0} with $u=v_\e$, and exploit the minimality of $v_\e$ to deduce that
\[
0=\AC(v_\e)-\psi(\e)\ge\frac1{C}\,\int_{\R^n}\e\,|\nabla h_\e|^2+\frac{h_\e^2}\e\,,
\]
that is $h_\e=0$ on $\R^n$, as claimed.
		
\medskip
	
\noindent {\it Step two}: We prove \eqref{sharp expansion psi eps} and \eqref{sharp expansion lambda eps}. If $u_\e$ is the minimizer of $\psi(\e)$ in $\RR_0$, then by Theorem \ref{theorem asymptotics of minimizers} we have $u_\e(x)=z_\e(x)+f_\e((|x|-R_0)/\e)$ for every $x\in\R^n$, and with $f_\e$ satisfying \eqref{co fu 1}. Moreover, as proved in \eqref{ansatz energy}, we have
\[
\AC(z_\e)= 2\,n\,\omega_n^{1/n}+2\,n\,(n-1)\,\omega_n^{2/n}\,\k_0 +{\rm O}(\e^2)\,.
\]	
Since $\AC(u_\e)\le\AC(z_\e)$, we are left to prove that $\AC(u_\e)\ge\AC(z_\e)-C\,\e^2$. Setting $|x|=R_0+\e\,s$ we have
\[
u_\e(x)=\eta(s-\tau_\e)+f_\e(s)\,,\qquad \nabla u_\e(x)= \frac{\eta'(s-\tau_\e)+f_\e'(s)}\e\,\frac{x}{|x|}\,,
\]
while $z_\e$ satisfies the same identities with $f_\e=0$, so that
\begin{eqnarray}\label{coo}
  &&\AC(u_\e)-\AC(z_\e)=\int_{-R_0/\e}^\infty\,\Big(2\,\eta'(s-\tau_\e)\,f_\e'(s)+f_\e'(s)^2\Big)\,(R_0+\e\,s)^{n-1}\,ds
  \\\nonumber
  &&+\int_{-R_0/\e}^\infty \Big(W\big(\eta(s-\tau_\e)+f_\e(s)\big)-W(\eta(s-\tau_\e))\Big)\,(R_0+\e\,s)^{n-1}\,ds\,.
\end{eqnarray}
Integration by parts and $2\,\eta''=W'(\eta)$ give
\begin{eqnarray*}
  &&\int_{-R_0/\e}^\infty\,2\,\eta'(s-\tau_\e)\,f_\e'(s)\,(R_0+\e\,s)^{n-1}\,ds
  \\&=&-\int_{-R_0/\e}^\infty\,W'\big(\eta(s-\tau_\e)\big)\,f_\e(s)\,(R_0+\e\,s)^{n-1}\,ds
  \\&&-2\,(n-1)\,\e\,\int_{-R_0/\e}^\infty\,\eta'(s-\tau_\e)\,f_\e(s)\,(R_0+\e\,s)^{n-2}\,ds\,.
\end{eqnarray*}
Dropping the non-negative term with $f_\e'(s)^2$ in \eqref{coo}, and noticing that, by \eqref{W second order taylor} and \eqref{co fu 1}, we have
\[
\Big|W\big(\eta(s-\tau_\e)+f_\e(s)\big)-W(\eta(s-\tau_\e))-W'\big(\eta(s-\tau_\e)\big)\,f_\e(s)\Big|\le C\,f_\e(s)^2\,,
\]
for every $s>-R_0/\e$, we thus find
\begin{eqnarray*}
  \AC(u_\e)-\AC(z_\e)
  &\ge&-2\,(n-1)\,\e\,\int_{-R_0/\e}^\infty\,\eta'(s-\tau_\e)\,f_\e(s)\,(R_0+\e\,s)^{n-2}\,ds
  \\
  &&-C\,\int_{-R_0/\e}^\infty f_\e(s)^2\,(R_0+\e\,s)^{n-1}\,ds\ge -C\,\e^2\,,
\end{eqnarray*}
where in the last inequality we have used \eqref{co fu 1}, $|\tau_\e|\le C$ and the decay estimate for $\eta'$ in \eqref{eta decay first and second derivative}. Coming to \eqref{sharp expansion lambda eps}, rearranging terms in \eqref{lambda formula again} we have
\begin{eqnarray}\label{bella}
\l(\e)=\Big(1-\frac2n\Big)\,\psi(\e)+\,\frac2n\,\,\frac1\e\int_{\R^n}W(u_\e)\,.
\end{eqnarray}
By \eqref{co fu 1}
\[
\frac1\e\int_{\R^n}W(u_\e)=\frac1\e\int_{\R^n}W(z_\e)+{\rm O}(\e)=\frac{\psi(\e)}2+{\rm O}(\e)
\]
where in the second identity we have used \eqref{energy expansion 1}. Hence $\l(\e)=(1-(1/n))\,\psi(\e)+{\rm O}(\e)$ and \eqref{sharp expansion lambda eps} follows from \eqref{sharp expansion psi eps}.

\medskip

\noindent {\it Step three}: We prove the continuity of $\l$ on $(0,\e_0)$. Let $\e_j\to\e_*\in(0,\e_0)$ as $j\to\infty$ and set $h_j=u_{\e_j}-u_{\e_*}$. By the resolution formula \eqref{critical sq resolution of vj} we have
\begin{eqnarray*}
  |u_{\e_j}(x)-u_{\e_*}(x)|&\le&\Big|\eta\Big(\frac{|x|-R_0}{\e_j}-\tau_{\e_j}\Big)-\eta\Big(\frac{|x|-R_0}{\e_*}-\tau_{\e_*}\Big)\Big|
  \\
  &&+\Big|f_{\e_j}\Big(\frac{|x|-R_0}{\e_j}\Big)-f_{\e_*}\Big(\frac{|x|-R_0}{\e_*}\Big)\Big|
  \\
  &\le& C\,\e_*\,e^{-(|x|-R_0)/C\,\e_*}+\Big|\eta\Big(\frac{|x|-R_0}{\e_j}-\tau_0\Big)-\eta\Big(\frac{|x|-R_0}{\e_*}-\tau_0\Big)\Big|
\end{eqnarray*}
where we have used \eqref{lemma tau}, \eqref{critical sequence bounds on fj} and \eqref{eta decay}. Similarly, since $\e_j\to\e_*>0$, for $j$ large enough we see that
\begin{eqnarray*}
  &&\Big|\eta\Big(\frac{|x|-R_0}{\e_j}-\tau_0\Big)-\eta\Big(\frac{|x|-R_0}{\e_*}-\tau_0\Big)\Big|
  \\
  &\le&\int_0^1\,\Big|\eta'\Big(\frac{|x|-R_0}{\e_*+t(\e_j-\e_*)}-\tau_0\Big)\Big|\,\frac{||x|-R_0|}{(\e_*+t\,(\e_j-\e_*))^2}\,|\e_j-\e_*|
  \\
  &\le& C\,\frac{|\e_j-\e_*|}{\e_*^2}\,e^{-(|x|-R_0)/C\,\e_*}\,||x|-R_0|\le C\,\e_*\,e^{-(|x|-R_0)/C\,\e_*}\,.
\end{eqnarray*}
Setting $h_j=u_{\e_j}-u_{\e_*}$ we see that \eqref{fu hp 1}, \eqref{fu hp 2} and \eqref{fu hp 3} hold with $\e=\e_*$ and for $j$ large enough, thus deducing that
\begin{eqnarray*}
  \frac1{C}\,\int_{\R^n}\e_*\,|\nabla h_j|^2+\frac{h_j^2}{\e_*}&\le& \ac_{\e_*}(u_{\e_j})-\psi(\e_*)
  \\
  &\le& \max\Big\{\frac{\e_j}{\e_*},\frac{\e_*}{\e_j}\Big\}\,\psi(\e_j)-\psi(\e_*)\,.
\end{eqnarray*}
From the continuity of $\psi$ on $(0,\e_0)$ (Theorem \ref{theorem existence solutions}) we conclude that
\[
\lim_{j\to\infty}\int_{\R^n}|\nabla u_{\e_j}-\nabla u_{\e_*}|^2=0\,,\qquad \lim_{j\to\infty}\int_{\R^n}W(u_{\e_j})=\int_{\R^n}W(u_{\e_*})\,,
\]
and thus $\lambda$ is continuous on $(0,\e_0)$ thanks to \eqref{bella}.
\end{proof}

We now turn to the proof of Theorem \ref{theorem fuglede estimate}. This is based on a series of three lemmas, each containing a different stability estimate, coming increasingly closer to \eqref{fuglede 0}.
	
\begin{lemma}[First stability lemma]\label{lemma one dimensional} Let $n\ge 2$, let $W\in C^{2,1}[0,1]$ satisfy \eqref{W basic} and \eqref{W normalization}, and let
\[
Q(u)= \int_{\R} 2\,(u')^2 +W''(\eta)\,u^2\,,\qquad u\in H^1(\R)\,.
\]
Then $Q(u)\ge0$ on $H^1(\R)$, and $Q(u)=0$ if and only if $u=t\,\eta'$ for some $t\in\R$.
\end{lemma}

\begin{proof} Let us consider the variational problem
\[
\gamma=\inf\Big\{Q(u): \int_{\R} u^2=1\Big\}\,.
\]		
By \eqref{eta decay first and second derivative} we have $\eta'\in H^1(\R)$. Differentiating $2\,\eta''=W'(\eta)$ we find $2\,(\eta')''=W''(\eta)\,\eta'$, and then integration by parts gives $Q(\eta')= 0$. At the same time we clearly have $Q(u)\ge-\|W''\|_{C^0(0,1)}\,\int_\R u^2$ for every $u\in H^1(\R)$, so that
\[
-\|W''\|_{C^0(0,1)}\le \g\le 0\,.
\]
We now prove that $\g$ is attained. Let $\{w_j\}_j$ be a minimizing sequence for $\g$. By the concentration-compactness principle, $\{w_j^2\,dx\}_j$ is in the vanishing case if
\begin{equation}\label{by by}
\lim_{j\to \infty} \int_{I_R}w_j^2 =0\,,\qquad\forall R>0\,,
\end{equation}
where we have set $I_R=(-R,R)$. By \eqref{eta decay} and \eqref{W near the wells} there exists $S_0$ such that
\begin{equation}
  \label{def of S0}
  W''(\eta)\ge \frac1C\,,\qquad\mbox{on $\R\setminus I_{S_0}$}\,.
\end{equation}
Therefore by applying \eqref{by by} twice with $R=S_0$ we find
\begin{eqnarray*}
\limsup_{j\to\infty}\int_{\R}w_j^2&=&\limsup_{j\to\infty}\int_{\R\setminus I_{S_0}}w_j^2\le C\,\limsup_{j\to\infty}\int_{\R\setminus I_{S_0}}W''(\eta)\,w_j^2
\\
&=& C\,\limsup_{j\to\infty}\int_{\R}W''(\eta)\,w_j^2\le \lim_{j\to \infty}  Q(w_j)=\g\le 0
\end{eqnarray*}
a contradiction to $\int_\R w_j^2=1$. If, instead, $\{w_j^2\,dx\}_j$ is in the dichotomy case, then there is $\alpha \in (0,1)$ such that for every $\tau\in (0, \a/2)$ there exist $R>0$ and $R_j\to\infty$ as $j\to\infty$ such that
\begin{equation}
  \label{lemma 1 dicho}
  \Big|1-\a-\int_{I_R} w_j^2\Big|<\tau\,,\qquad \Big|\a-\int_{\R\setminus I_{R_j}} w_j^2\Big|<\tau\,,
\end{equation}
where, without loss of generality, we can assume $R\ge S_0$ for $S_0$ as in \eqref{def of S0}. In particular, if $\vphi$ is a cut-off function between $I_R$ and $I_{R_j}$, then we have
\begin{equation}\label{dichotomy inequality Q1}
Q(w_j)=Q(\vphi\,w_j)+Q\big((1-\vphi)\,w_j\big)+E_j\,,
\end{equation}		
where, taking into account that $\varphi'$ and $(1-\varphi)\,\varphi$ are supported in $I_{R_j}\setminus I_{R}$, we have
\begin{eqnarray}
\label{error dichotomy Q1}
E_j=2\,\int_{I_{R_j}\setminus I_R} W''(\eta)\,(1-\varphi)\,\varphi \,w_j^2+4\,\int_{I_{R_j}\setminus I_R}(\varphi w_j)'((1-\varphi) w_j)'\,.
\end{eqnarray}
The first integral in \eqref{error dichotomy Q1} is non-negative by \eqref{def of S0}, while the second integral contains a non-negative term of the form $\varphi\,(1-\varphi)\,(w_j')^2$: therefore, by \eqref{lemma 1 dicho},
\begin{eqnarray}\nonumber
E_j&\ge&4\,\int_{I_{R_j}\setminus I_R} w_j\,\,w_j'\,(1-\varphi)\,\varphi'-w_j\, w_j'\,\varphi\,\varphi'-w_j^2 \,(\varphi')^2\\
		&\geq& -C\,\int_{I_{R_j}\setminus I_R} w_j^2 -C\,\Big(\int_{I_{R_j}\setminus I_R} w_j^2\Big)^{1/2}\,\Big(\int_{\R} (w_j')^2\Big)^{1/2}\ge-C\,\sqrt\tau\,,\label{mnb}
\end{eqnarray}		
where we have also used $Q(w_j) \to \gamma$ as $j\to \infty$ to infer
\[
\int_{\R} (w_j')^2 \leq Q(w_j)+ \| W\|_{C^0[0,1]}\le C\,.
\]
We can take $\vphi$ supported in $I_{R+1}$. In this way, up to extracting a subsequence, we have that $\vphi\,w_j$ admits a weak limit $w$ in $H^1(\R)$. By lower semicontinuity, homogeneity of $Q$ and \eqref{lemma 1 dicho} we have that
\begin{equation}
  \label{what about vphi wj}
  \liminf_{j\to\infty}Q(\vphi\,w_j)\ge Q(w)\ge \g\,\int_\R w^2\ge\,(1-\a)\,\g-C\,\tau\,.
\end{equation}
Finally, since $(1-\vphi)$ is supported on $\R\setminus I_{S_0}$, by \eqref{def of S0} we have
\[
\int_\R Q\big((1-\vphi)\,v_j\big)\ge\frac1{C}\,\int_\R\,(1-\vphi)^2\,w_j^2\ge\frac{\a}C-C\,\tau\,,
\]
so that, combining \eqref{dichotomy inequality Q1}, \eqref{mnb}, and \eqref{what about vphi wj} we find
\[
\g\ge (1-\a)\,\g+\frac\a{C}-C\,\sqrt\tau\,.
\]
Letting $\tau\to 0^+$ we find a contradiction with $\g\le 0$ and $\a>0$. Having excluded vanishing and dichotomy, we have proved the existence of minimizers of $\g$.

\medskip

Let now $u$ be a minimizer of $\g$. Up to replace with $u$ with $|u|$ we can assume $u\ge0$. By a standard variational argument there exists $\l\in\R$ such that
		\begin{equation}\label{eq EL one dim}
		\int_{\R}2 u'\,v'+W''(\eta)\,u\,v=\lambda \int_{\R} u\,v\,,\qquad\forall v\in H^1(\R)\,.
		\end{equation}		
Testing with $v=\eta'$ and recalling that $2\,(\eta')''=W''(\eta)\eta'$, we deduce that
		$$\lambda \int_{\R} \eta' u=0,$$		
		and, since $u \geq 0$, $\int_\R u^2=1$, and $\eta'<0$, we find $\lambda =0$. Thus $u$ is a $C^2$-solution of the ODE $2\,u''=W''(\eta)\,u$ on $\R$. If $u(r_0)=0$ for some $r_0\in\R$, then $u'(r_0)\ne 0$ (otherwise we would have $u=0$ on $\R$, against $\int_\R u^2=1$), and $u'(r_0)\ne 0$ contradicts $u\ge0$ on $\R$. Hence, $u>0$ on $\R$.

\medskip

Having proved that every minimizer of $\g$ is either positive or negative on the whole $\R$, we deduce that $\g=0$. Indeed, if $u_1$ and $u_2$ are minimizers of $\g$ and, say, they are both positive on $\R$, then they solve \eqref{eq EL one dim} with $\l=0$, and thus
\[
2\,\g=Q(u_1)+Q(u_2)=Q(u_1-u_2)\ge \g\,\int_\R(u_1-u_2)^2=2\,\g\,\Big(1-\int_\R u_1\,u_2\Big)\,.
\]
Since $\int_\R u_1\,u_2\in(0,1)$, $\g<0$ would give a contradiction. Having established that $\g=0$, we now know that $\eta'$ is a minimizer of $\g$. If $u$ is also minimizer of $\g$, then, again by \eqref{eq EL one dim},
\[
Q(u+s\,\eta')=Q(u)+s^2\,Q(\eta')=0\qquad\forall s\in\R\,.
\]
In particular, if $s\in\R$ is such that $u+s\,\eta'$ is not identically zero on $\R$, then $(u+s\,\eta')/\|u+s\,\eta'\|_{L^2(\R)}^2$ is a minimizer of $\g$, and thus $u+s\,\eta'$ is either positive or negative on the whole $\R$. Let $s_0=\inf\{s:\mbox{$u+s\,\eta'<0$ on $\R$}\}$. If, say, $u$ is a negative minimizer (like $\eta'$ is), then $s_0\le 0$; while, clearly, $s_0>-\infty$, since, for $s$ negative enough, we must have $u+s\,\eta'>0$ at at least one point, and thus everywhere. Since $u+s_0\,\eta'\le0$ on $\R$ with $u+s_0\,\eta'=0$ at at least one point, we deduce that $u+s_0\,\eta'=0$ on $\R$.
\end{proof}

\begin{lemma}[Second stability lemma]\label{non linear stability} If $n\ge 2$ and $W\in C^{2,1}[0,1]$ satisfies \eqref{W basic} and \eqref{W normalization}, then there exists a universal constant $\e_0$ with the following property. If $u_\e\in\RR_0^*$ is a minimizer of $\psi(\e)$ for $\e<\e_0$ and $h\in H^1(\R^n)$ is a radial function such that
\begin{equation}\label{eq inf orthogonality}
\int_{\R^n} V'(u_\e)\, h=0\,,
\end{equation}
then
\begin{equation}\label{second varation 2}
	\int_{\R^n}2\,\e\,|\nabla h|^2+\Big(\frac{W''(u_\e)}\e-\l(\e) V''(u_\e)\Big)\, h^2\geq \frac1{C}\,\int_{\R^n}\e\,|\nabla h|^2+\frac{h^2}\e\,,
	\end{equation}
where $\l(\e)$ is the Lagrange multiplier of $u_\e$ as in \eqref{lambda formula again}.
\end{lemma}
	
\begin{proof} {\it Step one}: We show that is enough to prove the lemma with
\begin{equation}\label{second varation 2 redux}
\int_{\R^n}2\,\e\,|\nabla h|^2+\Big(\frac{W''(u_\e)}\e-\l(\e) V''(u_\e)\Big)\, h^2\geq \frac1C\,\int_{\R^n}\frac{h^2}\e\,,
\end{equation}
in place of \eqref{second varation 2}. Indeed, if $\e_0$ is small enough, then $|\l(\e)|\le c(n)$ thanks to \eqref{lambda estimate eps}, and thus we can find a universal constant $C_*$ such that
\[
\int_{\R^n}\Big|\frac1{\e}\,W''(u_\e)-\l(\e) V''(u_\e)\Big|\, h^2\leq C_* \int_{\R^n} \frac{h^2}\e\,,
\]
whenever $u_\e$ is a minimizer of $\psi(\e)$, $\e<\e_0$, and $h\in H^1(\R^n)$. Let us now fix a radial function $h\in H^1(\R^n)$ satisfying \eqref{eq inf orthogonality}. If $C_*\,\int_{\R^n} h^2/\e \leq \int_{\R^n}\e\,|\nabla h|^2$, then we trivially have
\begin{eqnarray*}
\int_{\R^n}2\,\e|\nabla h|^2+\Big(\frac{W''(u_\e)}{\e}-\l(\e) \,V''(\zeta_\e)\Big)\, h^2
\ge \int_{\R^n}2\,\e |\nabla h|^2- C_*\,\int_{\R^n} \frac{h^2}\e\geq \int_{\R^n}\e\,|\nabla h|^2\,;
\end{eqnarray*}		
if, instead, $C_*\,\int_{\R^n} h^2/\e \ge \int_{\R^n}\e\,|\nabla h|^2$, then we deduce from \eqref{second varation 2 redux}
\[
\int_{\R^n}2\,\e |\nabla h|^2+\Big(\frac{W''(u_\e)}\e-\l(\e) \,V''(u_\e)\Big)\, h^2\geq \frac1{C}\,\int_{\R^n}\,\frac{h^2}\e\geq \frac1{C\,C_*}\int_{\R^n}\e\,|\nabla h|^2\,.
\]		
In both cases, \eqref{second varation 2} is easily deduced.

\medskip

\noindent {\it Step two}: We prove \eqref{second varation 2 redux}. We argue by contradiction, and consider $\e_j\to 0^+$ as $j\to\infty$, $u_j\in\RR_0^*$ minimizers of $\psi(\e_j)$, and radial functions $h_j\in H^1(\R^n)$ such that
\begin{eqnarray}\label{js 1}
  &&\int_{\R^n}V'(u_j)\,h_j=0\,,
  \\\label{js 2}
  &&\int_{\R^n}2\,\e_j\,|\nabla h_j|^2+\Big(\frac{W''(u_j)}{\e_j}-\l_j\, V''(u_j)\Big) \,h_j^2<\frac1j\,\int_{\R^n}\frac{h_j^2}{\e_j}\,.
\end{eqnarray}
where $\l_j$ are the Lagrange multipliers corresponding to $u_j$. By homogeneity of \eqref{js 1} and \eqref{js 2} we can also assume that
\begin{equation}
  \label{js 3}
  \int_{\R^n}\frac{h_j^2}{\e_j}=1\,.
\end{equation}
Therefore, setting
\[
\eta_j(s)=u_j(R_0+\e_j\,s)\,,\qquad \beta_j(s)=h_j(R_0+\e_j\,s)\,,\qquad s\ge -\frac{R_0}{\e_j}\,,
\]
we can recast \eqref{js 2} and \eqref{js 3} as
\begin{eqnarray} \label{js 4}
&&\int_{-R_0/\e_j}^\infty  \Big(2\,(\beta_j')^2+(W''(\eta_j)-\e_j\,\l_j\,V''(\eta_j))\,\beta_j^2\Big)\,(R_0+\e_j\,s)^{n-1}\,ds\,\le \frac1j\,,
\\\label{js 5}
&&\int_{-R_0/\e_j}^\infty  \beta_j(s)^2\,(R_0+\e_j\,s)^{n-1}\,ds=1\,.
\end{eqnarray}
By $\e_j\to 0^+$ and by \eqref{lambda estimate eps} we know $\l_j\to c(n)$ as $j\to\infty$, which combined with $\|V''\|_{C^0[0,1]}\le C$ and $\e_j\to0^+$ shows that \eqref{js 4} and \eqref{js 5} imply
\begin{eqnarray}
\label{js 6}
\limsup_{j\to\infty}\int_{-R_0/\e_j}^\infty  \Big\{2\,(\beta_j')^2+W''(\eta_j)\,\beta_j^2\Big\}\,(R_0+\e_j\,s)^{n-1}\,ds\le 0\,.
\end{eqnarray}
Since $W''$ is bounded on $[0,1]$, by \eqref{js 5} and \eqref{js 6} we deduce that $\{\beta_j\}_j$ is bounded in $H^1(-s_0,s_0)$ for every $s_0>0$. In particular there exists $\beta\in H^1_{{\rm loc}}(\R)$ such that, up to extracting subsequences, $\beta$ is the weak limit of $\{\beta_j\}_j$ in $H^1(-s_0,s_0)$ for every $s_0>0$. By $\beta_j'\weak\beta'$ in $L^2(-s_0,s_0)$ for every $s_0>0$ we easily find
\begin{eqnarray}\label{js 7}
  \liminf_{j\to\infty}\int_{-R_0/\e_j}^\infty  2\,\beta_j'(s)^2(R_0+\e_j\,s)^{n-1}\,ds\ge R_0^{n-1}\,\int_{\R} 2\,(\beta')^2\,.
\end{eqnarray}
We now apply the concentration-compactness principle to the sequence of measures
\[
\mu_j=1_{(-R_0/\e_j,\infty)}(s)\,\beta_j(s)^2\,(R_0+\e_j\,s)^{n-1}\,ds\,,
\]
which satisfy $\mu_j(\R)=1$ thanks to \eqref{js 3}. We claim that, if the compactness case hold, and thus
\begin{equation}
  \label{js 8}
  \lim_{s_0\to+\infty}\sup_j\,\mu_j(\R\setminus[-s_0,s_0])=0\,,
\end{equation}
then we can reach a contradiction, and complete the proof of the lemma. To prove this claim, let us set
\[
\eta_0(s)=\eta(s-\tau_0)\,,
\]
for $\tau_0$ as in \eqref{definition of tau0}, and let us notice that, for every $s_0>0$ we have
\begin{eqnarray}\nonumber
  &&\limsup_{j\to\infty}\Big|\int_{-R_0/\e_j}^\infty  W''(\eta_j)\beta_j(s)^2\,(R_0+\e_j\,s)^{n-1}\,ds-R_0^{n-1}\,\int_{\R}W''(\eta_0)\,\beta^2\Big|
  \\\label{js 9}
  &\le&\limsup_{j\to\infty}\int_{-s_0}^{s_0} \Big| W''(\eta_j)\beta_j(s)^2\,(R_0+\e_j\,s)^{n-1}-R_0^{n-1}\,W''(\eta_0)\,\beta^2\Big|
  \\\nonumber
  &&+\|W''\|_{C^0[0,1]}\,\sup_{j\in\N}\mu_j(\R\setminus[-s_0,s_0])+R_0^{n-1}\,\|W''\|_{C^0[0,1]}\,\int_{\R\setminus[-s_0,s_0]}\beta^2\,.
\end{eqnarray}
Since $\beta_j\to\beta$ in $L^2_{{\rm loc}}(\R)$ and $\eta_j\to\eta_0$ locally uniformly on $\R$ thanks to Theorem \ref{theorem asymptotics of minimizers}, the first term on the right-hand side of \eqref{js 9} is equal to zero. Letting now $s_0\to\infty$, the second term goes to zero thanks to \eqref{js 8}, while the third terms goes to zero thanks to the fact that \eqref{js 8} implies in particular
\begin{equation}
  \label{js 10}
  R_0^{n-1}\,\int_\R\beta^2=1\,.
\end{equation}
We can combine this information with \eqref{js 7} and finally deduce from \eqref{js 6} that
\begin{equation}
  \label{js 11}
  \int_\R\,2\,(\beta')^2+W''(\eta_0)\,\beta^2\le 0\,.
\end{equation}
By Lemma \ref{lemma one dimensional} we deduce that, if we set $\beta_0(s)=\beta(s+\tau_0)$, then $\beta_0=t\,\eta'$ for some $t\ne 0$ ($t=0$ being ruled out by \eqref{js 10}). In particular, $\beta=t\,\eta_0'$, and therefore
\[
\int_\R\,V'(\eta_0)\,\beta=t\,V(\eta_0)|^{+\infty}_{-\infty}=t\,V(1)=t\ne 0\,.
\]
However, by \eqref{js 1}, we see that
\[
0=\int_{\R^n}V'(u_j)\,h_j=\int_{-R_0/\e_j}^\infty\,V'(\eta_j)\,\beta_j(s)\,(R_0+s\,\e_j)^{n-1}\,ds\,,\qquad\forall j\,,
\]
and we can thus obtain a contradiction by showing that
\begin{equation}
  \label{js 12}
\lim_{j\to\infty}\int_{-R_0/\e_j}^\infty\,V'(\eta_j)\,\beta_j(s)\,(R_0+s\,\e_j)^{n-1}\,ds=R_0^{n-1}\,\int_\R\,V'(\eta_0)\,\beta\,.
\end{equation}
This is proved by noticing that \eqref{V near the wells}, \eqref{eta decay}, \eqref{etaj 2} and \eqref{etaj 3} give
\[
0\le\max\{V'(\eta_j),V'(\eta_0)\}\le C\,e^{-|s|/C}\,,
\]
for every $s\in\R$ (or for every $s\ge-R_0/\e_j$, in the case of $\eta_j$). In particular,
\begin{eqnarray*}
&&\lim_{s_0\to\infty}\limsup_{j\to\infty}\,\Big[\int_{-R_0/\e_j}^{-s_0}+\int_{s_0}^\infty\Big] V'(\eta_j)\,|\beta_j|\,(R_0+s\,\e_j)^{n-1}\,ds
\\
&&\le\,C\,\lim_{s_0\to\infty}\,\limsup_{j\to\infty}\Big(\int_{\{|s|>s_0\}}e^{-|s|/C}\,(R_0+s\,\e_j)^{n-1}\,ds\Big)^{1/2}\,\mu_j(\R\setminus[-s_0,s_0])^{1/2}=0\,,
\end{eqnarray*}
so that a similar argument to the one used in \eqref{js 9} can be repeated to prove \eqref{js 12}.

\medskip

We are thus left to prove that the sequence of probability measures $\{\mu_j\}_j$ cannot be in the vanishing case nor in the dichotomy case of the concentration-compactness principle.

\medskip

\noindent {\it To exclude that $\{\mu_j\}_j$ is in the vanishing case}: Since $\eta_j\to\eta$ locally uniformly on $\R$, up to take $j$ large enough and for $S_0$ as in \eqref{def of S0} we have $W''(\eta_j(s))\ge 1/C$ for $|s|\ge S_0$, $s\ge-R_0/\e_j$. Since we are in the vanishing case, it holds
\begin{equation}
  \label{js vanishing 1}
  \lim_{j\to\infty}\int_{-S_0}^{S_0}\,\beta_j(s)^2\,(R_0+\e_j\,s)^{n-1}\,ds=0\,,
\end{equation}
so that, by using first the lower bound on $W''$, and then \eqref{js vanishing 1}, we get
\begin{eqnarray*}
&&\frac1C\,\limsup_{j\to\infty}\Big[\int_{-R_0/\e_j}^{-S_0}+\int_{S_0}^\infty\Big]\beta_j(s)^2\,\,(R_0+\e_j\,s)^{n-1}\,ds
\\&\le&\limsup_{j\to\infty}\Big[\int_{-R_0/\e_j}^{-S_0}+\int_{S_0}^\infty\Big]W''(\eta_j)\,\beta_j(s)^2\,\,(R_0+\e_j\,s)^{n-1}\,ds
\\
&=&\limsup_{j\to\infty}\int_{-R_0/\e_j}^\infty\,W''(\eta_j)\,\beta_j(s)^2\,\,(R_0+\e_j\,s)^{n-1}\,ds\le0
\end{eqnarray*}
where in the last inequality we have used \eqref{js 6}. Combining this information with \eqref{js vanishing 1} we obtain a contradiction to \eqref{js 5}, thus excluding the vanishing case.

\medskip

\noindent {\it To exclude that $\{\mu_j\}_j$ is in the dichotomy case}: With $S_0$ as above, if we are in the dichotomy case, then there exists $\alpha \in (0,1)$ such that for every $\tau\in(0,\a/2)$ there exist $R>S_0$ and $R_j\to\infty$ such that
\begin{equation}
  \label{js fine}
  |\mu_j(I_R) -(1-\alpha)| < \tau\,,\qquad |\mu_j(\R\setminus I_{R_j})-\alpha|<\tau\,,\qquad\forall j\,.
\end{equation}
Setting $A_j=\vphi\,\beta_j$, $B_j=(1-\vphi)\,\beta_j$, where $\vphi$ is a cut-off function between $B_R$ and $B_{R+1}$, and setting for the sake of brevity,
\[
Q_j(A,B)=\int_{-R_0/\e_j}^\infty  \Big\{2\,A'\,B'+W''(\eta_j)\,A\,B\Big\}\,(R_0+\e_j\,s)^{n-1}\,ds\,,\qquad Q_j(A)=Q_j(A,A)\,,
\]
we can rewrite \eqref{js 6} as
\begin{equation}
\label{js 6-di}
\limsup_{j\to\infty} Q_j(A_j)+Q_j(B_j)+2\,Q_j(A_j,B_j)\le 0\,,
\end{equation}
Now, since $\vphi'$ and $(1-\varphi)\,\varphi$ are supported in $I_{R+1}\setminus I_R$, we see that
\begin{eqnarray*}
  Q_j(A_j,B_j)&\ge&2\,
  \int_{I_{R+1}\setminus I_R}\,(1-2\,\vphi)\,\vphi'\,\beta_j\,\beta_j' \,(R_0+\e_j\,s)^{n-1}\,ds
  \\
  &&
  +\int_{I_{R+1}\setminus I_R}\,\Big\{W''(\eta_j)-(\vphi')^2\Big\}\,\beta_j^2\,(R_0+\e_j\,s)^{n-1}\,ds
\end{eqnarray*}
where, thanks to \eqref{js 6} and H\"older inequality,
\begin{eqnarray*}
    &&\int_{I_{R+1}\setminus I_R}\,(1-2\,\vphi)\,\vphi'\,\beta_j\,\beta_j'\,(R_0+\e_j\,s)^{n-1}\,ds\le C\,\mu_j(I_{R+1}\setminus I_R)^{1/2}\le C\,\sqrt\tau\,
    \\
    &&\int_{I_{R+1}\setminus I_R}\,\Big\{W''(\eta_j)-(\vphi')^2\Big\}\,\beta_j^2\,(R_0+\e_j\,s)^{n-1}\,ds
    \le C\,\mu_j(I_{R+1}\setminus I_R)\le C\,\tau\,.
\end{eqnarray*}
We thus conclude that $Q_j(A_j,B_j)\ge-C\,\sqrt\tau$ for every $j$, and thus, by \eqref{js 6-di}, that
\begin{equation}
\label{js 6-tri}
\limsup_{j\to\infty} Q_j(A_j)+Q_j(B_j)\le C\,\sqrt\tau\,.
\end{equation}
Now, since the supports of the $A_j$'s are uniformly bounded, we easily see that there exists $A\in H^1(\R)$ such that $A_j\weak A$ weakly in $H^1(\R)$; in particular,
\[
\liminf_{j\to\infty}Q_j(A_j)\ge\int_\R\,2\,(A')^2+W''(\eta_0)\,A^2\ge 0\,,
\]
where in the last inequality we have used Lemma \ref{lemma one dimensional}. By combining this last inequality with \eqref{js 6-tri},  $W''(\eta_j)\ge 1/C$ on $\R\setminus I_{S_0}$, and $R\ge S_0$, we conclude that
\[
C\,\sqrt\tau \ge\limsup_{j\to\infty} Q_j(B_j)\ge
\frac1{C}\,\limsup_{j\to\infty}\int_{-R_0/\e_j}^\infty\,(1-\vphi)^2\,\beta_j^2\,(R+s\,\e_j)^{n-1}\,ds
\]
and thus, by \eqref{js fine}, that $C\,\sqrt\tau \ge(\a/C)-C\,\tau$. Letting $\tau\to 0^+$ we obtain a contradiction with $\a>0$.
\end{proof}

\begin{lemma}[Third stability lemma]\label{lemma third stability} If $n\ge 2$ and $W\in C^{2,1}[0,1]$ satisfies \eqref{W basic} and \eqref{W normalization}, then there exist universal constants $\de_0$ and $\e_0$ such that, if $u_\e\in\RR_0^*$ is a minimizer of $\psi(\e)$ for $\e<\e_0$ and $u\in H^1(\R^n;[0,1])$ is a radial function with
\begin{eqnarray}
  \label{cls hp 1}
&&\int_{\R^n} V(u)=1\,,
\\
\label{cls hp 2}
&&\int_{\R^n}(u-u_\e)^2\le C\,\e\,,
\\
\label{cls hp 3}
&&\|u-u_\e\|_{L^\infty(\R^n)}\le\de_0\,,
\end{eqnarray}
then, setting $h=u-u_\e$,
\begin{equation}\label{cls conclusion}
	\int_{\R^n}2\,\e\,|\nabla h|^2+\Big(\frac{W''(u_\e)}\e-\l(\e) V''(u_\e)\Big)\, h^2\geq \frac1{C}\int_{\R^n}\e\,|\nabla h|^2+\frac{h^2}\e\,.
	\end{equation}
where $\l(\e)$ is the Lagrange multiplier of $u_\e$ as in \eqref{lambda formula again}.
\end{lemma}

\begin{proof} It will be convenient to set
\begin{eqnarray*}
P_\e(u,v)&=&\int_{\R^n}\e\,\nabla u\cdot\nabla v+ \frac{u\,v}\e\,,
\\
Q_\e(u,v)&=&\int_{\R^n}\e\,\nabla u\cdot\nabla v+\Big(\frac{W''(u_\e)}\e-\l(\e)\,V''(u_\e)\Big)\,u\,v\,,
\end{eqnarray*}
as well as $P_\e(u)=P_\e(u,u)$ and $Q_\e(u)=Q_\e(u,u)$. Let us start noticing that by Theorem \ref{theorem asymptotics of minimizers} we have
\[
\lim_{\s\to 0}\,\sup_{\e<\s}\,\sup_{v_\e}\,
\Big|\int_{\R^n}V'(v_\e)\,v_\e-R_0^{n-1}\,\int_\R V'(\eta)\,\eta\Big|=0\,,
\]
where $v_\e$ runs over all radial minimizers of $\psi(\e)$. Since $\int_\R V'(\eta)\,\eta$ is a positive constant depending on $n$ and $W$ only this shows in particular that
\begin{equation}
  \label{cls 1}
\frac1{C}\le\int_{\R^n}V'(u_\e)\,u_\e\le C\,,\qquad\forall \e<\e_0\,.
\end{equation}
By \eqref{cls 1}, given $h=u-u_\e$ as in the statement, we can always find $t\in\R$ such that
\begin{equation}
  \label{cls formula for t}
  \int_{\R^n} V'(u_\e)\,(h+t\,u_\e)=0\,,\qquad\mbox{i.e.}\quad t=-\frac{\int_{\R^n}V'(u_\e)\,h}{\int_{\R^n}V'(u_\e)\,u_\e}\,.
\end{equation}
By \eqref{V second order taylor}, \eqref{cls hp 3}, and since $0\le u_\e+h\le 1$, we have that, on $\R^n$,
\begin{equation}
  \label{cls taylor second for V}
  \Big|V(u_\e+h)-V(u_\e)-V'(u_\e)\,h-V''(u_\e)\,\frac{h^2}2\Big|\le C\,\de_0\,h^2\,,
\end{equation}
so that, by \eqref{cls hp 1},
\begin{eqnarray}\label{feb}
\Big|\int_{\R^n}V'(u_\e)\,h+V''(u_\e)\,\frac{h^2}2\Big|\le C\,\de_0 \,\int_{\R^n}h^2\,,
\end{eqnarray}
and thus, thanks to $\|V''\|_{C^0[0,1]}\le C$, \eqref{cls 1}, \eqref{cls hp 2}, and \eqref{cls formula for t},
\begin{equation}
  \label{cls bound for t}
  |t|\le C\,\int_{\R^n}h^2\le C\,\e\,\min\{P_\e(h),1\}\,,
\end{equation}
By \eqref{cls formula for t} we can apply Lemma \ref{non linear stability} to $u_\e+t\,h$ and find that
\[
Q_\e(h+t\,u_\e)\ge \frac{P_\e(h+t\,u_\e)}C\,,
\]
which can be more conveniently rewritten as
\begin{equation}
\label{eq perturbed second var}
Q_\e(h)\geq \frac{P_\e(h)}C+2\,t\, \Big\{\frac{P_\e(h,u_\e)}C-Q_\e(h,u_\e)\Big\}+t^2\,\Big\{\frac{P_\e(u_\e)}C-Q_\e(u_\e)\Big\}\,.
\end{equation}
By Theorem \ref{theorem asymptotics of minimizers}, we see that $P_\e(u_\e)+|Q_\e(u_\e)|\le C$ (uniformly on $\e<\e_0$), so that \eqref{eq perturbed second var} and \eqref{cls bound for t} give
\begin{equation}
\label{cls 2}
Q_\e(h)\geq \frac{P_\e(h)}C+2\,t\, \Big\{\frac{P_\e(h,u_\e)}C-Q_\e(h,u_\e)\Big\}\,.
\end{equation}
By H\"older inequality, $ab\le (a^2+b^2)/2$, $P_\e(u_\e)\le C$, and \eqref{cls bound for t} we see that
\begin{equation}
  \label{cls 3}
  |t|\,P_\e(h,u_\e)\le \frac{|t|}2\,\big(P_\e(h)+P_\e(u_\e)\big)\le C\,\e\,P_\e(h)\,,
\end{equation}
while by $|V'|+|W''|\le C$ and $|\l(\e)|\le C$ for $\e<\e_0$ we find, arguing as in \eqref{cls 3},
\begin{equation}
  \label{cls 4}
  |t|\,Q_\e(h,u_\e)\le |t|\,\Big\{\e\,\int_{\R^n}|\nabla h|\,|\nabla u_\e|+\frac{C}\e\,\int_{\R^n}|h|\,u_\e\Big\}\le C\,\e\,P_\e(h)\,.
\end{equation}
By combining \eqref{cls 2}, \eqref{cls 3}, and \eqref{cls 4} we conclude that $Q_\e(h)\ge P_\e(h)/C$, as desired.
\end{proof}

We are finally ready to prove Theorem \ref{theorem fuglede estimate}.

\begin{proof}[Proof of Theorem \ref{theorem fuglede estimate}] We are given $u_\e$ and $h$ as in Lemma \ref{lemma third stability}, and now want to prove that
\begin{equation}\label{fuglede}
\AC(u_\e+h)-\psi(\e)\geq \frac1{C}\,\int_{\R^n}\e\,|\nabla h|^2+\frac{h^2}\e\,,
\end{equation}
holds. By \eqref{W second order taylor} and \eqref{cls hp 3} we have that
\[
\Big|W(u_\e+h)-W(u_\e)-W'(u_\e)\,h-W''(u_\e)\,\frac{h^2}2\Big|\le C\,\de_0\,h^2\,,\qquad\mbox{on $\R^n$}\,,
\]
therefore
\begin{eqnarray}\nonumber
\AC(u_\e+h)-\AC(u_\e)&\ge&\int_{\R^n}2\,\e\,\nabla u_\e\cdot\nabla h+\frac{W'(u_\e)}\e\,h
\\\label{fu proof 1}
&&+\int_{\R^n}\e\,|\nabla h|^2+\frac{W''(u_\e)}{2\,\e}\, h^2-C\,\de_0\,\int_{\R^n}h^2\,.
\end{eqnarray}
By the Euler--Lagrange equation for $u_\e$, see \eqref{EL classic W V}, we have
\begin{equation}
  \label{fu proof 2}
  \int_{\R^n}2\,\e\,\nabla u_\e\cdot\nabla h+\frac{W'(u_\e)}\e\,h=\l(\e)\,\int_{\R^n} V'(u_\e)\,h\,.
\end{equation}
Moreover, by \eqref{feb},
\begin{equation}
  \label{fu proof 3}
  \Big|\int_{\R^n}V'(u_\e)h+\int_{\R^n}V''(u_\e) \frac{h^2}2\Big|\le C\,\de_0\, \,\int_{\R^n}h^2\,.
\end{equation}
On combining \eqref{fu proof 1}, \eqref{fu proof 2}, and \eqref{fu proof 3} with \eqref{cls conclusion} we find that
\begin{eqnarray*}
\AC(u_\e+h)-\psi(\e)&\ge&
\frac12\int_{\R^n}2\,\e\,|\nabla h|^2+\Big\{\frac1\e\,W''(u_\e)-\l(\e)\,V''(u_\e)\Big\}\,h^2-C\,\de_0\,\int_{\R^n}h^2
\\
&\ge&\int_{\R^n}\e\,|\nabla h|^2+\frac{h^2}\e-C\,\de_0\,\int_{\R^n}h^2\,,
\end{eqnarray*}
so that \eqref{fuglede} follows by taking $\de_0$ small enough.
\end{proof}

\section{Proof of the uniform stability theorem}\label{section proof uniform stab} In this section we prove Theorem \ref{theorem main}-(iii), i.e., we prove \eqref{global quantitative estimate}. We focus directly on the case $(\s,m)=(\e,1)$, from which the general case follows immediately by scaling.

\begin{theorem}
  \label{thm main 2} If $n\ge 2$ and $W\in C^{2,1}[0,1]$ satisfies \eqref{W basic} and \eqref{W normalization}, then there exist universal constants $\e_0>0$ and $C$ such that, if $\e<\e_0$ and $u\in H^1(\R^n;[0,1])$ with $\int_{\R^n}V(u)=1$, then
  \begin{equation}\label{global quantitative estimate eps uno}
  C\,\sqrt{\AC(u)-\psi(\e)}\ge \inf_{x_0\in\R^n}\,
  \int_{\R^n} \big|\Phi(u) -\Phi(T_{x_0}u_\e)\big|^{n/(n-1)}
  \end{equation}
  where $T_{x_0}u_\e(x)=u_\e(x-x_0)$, $x\in\R^n$, and $u_\e$ denotes the unique minimizer of $\psi(\e)$ in $\RR_0$.
\end{theorem}

In order to streamline the exposition of the proof of Theorem \ref{thm main 2}, we introduce the isoperimetric deficit and asymmetry of $u\in H^1(\R^n;[0,1])$ with $\int_{\R^n}V(u)=1$, by setting
\begin{eqnarray*}
  \de_\e(u)&=&\AC(u)-\psi(\e)\,,
  \\
  \a_\e(u)&=&\inf_{x_0\in\R^n}d_\Phi(u,T_{x_0}u_\e)\,.
\end{eqnarray*}
Here, as in Theorem \ref{theorem compactness for selection principle},
\[
d_\Phi(u,v)=\int_{\R^n}|\Phi(u)-\Phi(v)|^{n/(n-1)}\,,\qquad\forall u,v\in H^1(\R^n;[0,1])\,.
\]
With this notation, Theorem \ref{thm main 2} states the existence of universal constants $C$ and $\e_0$ such that, if $\e<\e_0$, then
\begin{equation}
  \label{the end 1}
  C\,\sqrt{\de_\e(u)}\ge \a_\e(u)\,,\qquad\forall u\in H^1(\R^n;[0,1])\,,\int_{\R^n}V(u)=1\,.
\end{equation}
In the following subsections we discuss some key steps of the proof of Theorem \ref{thm main 2}, which is then presented at the end of this section.

\subsection{Reduction to the small asymmetry case}\label{subsec small asymmetry} Thanks to the volume constraint $\int_{\R^n}V(u)= 1$ and to the triangular inequality in $L^{n/(n-1)}$, we always have $\a_\e(u)\le 2^{n/(n-1)}$. In particular, in proving \eqref{the end 1}, we can always assume that $\de_\e(u)\le \de_0$ for a universal constant $\de_0$. This is useful because, by the following lemma, by assuming $\de_\e(u)\le \de_0$ we can take $\a_\e(u)$ as small as needed in dependence of $n$ and $W$.

\begin{lemma}[$\e$-uniform qualitative stability]\label{lemma uniform qualitative stability}
  If $n\ge 2$ and $W\in C^{2,1}[0,1]$ satisfies \eqref{W basic} and \eqref{W normalization}, then there exists a universal constant $\e_0$ with the following property: for every $\a>0$ there exists $\de>0$ such that
  \[
  u\in H^1(\R^n;[0,1])\,,\qquad\int_{\R^n}V(u)=1\,,\qquad \e<\e_0\,,\qquad \de_\e(u)\le\de\,,
  \]
  imply
  \[
  \a_\e(u)\le\a\,.
  \]
\end{lemma}

\begin{proof}
  We pick $\e_0$ such that Theorem \ref{theorem existence solutions} and Corollary \ref{corollary optimal energy and lagrange multiplier} hold. If  the lemma is false for such $\e_0$, then there exists $\a_*>0$ and a sequence $\{u_j\}_j$ in $H^1(\R^n;[0,1])$ with $\int_{\R^n}V(u_j)=1$ such that
  \begin{equation}
    \label{a zero}
      \de_{\e_j}(u_j)\to 0^+\qquad\mbox{as $j\to\infty$}\,,
  \end{equation}
  for some $\e_j\to\e_*\in[0,\e_0]$ and with $\a_{\e_j}(u_j)\ge\a_*$. By \eqref{a zero}, there is $\ell_j\to 0^+$ as $j\to\infty$ such that
  \begin{equation}
    \label{a zero 2}
      \ac_{\e_j}(u_j)\le\psi(\e_j)+\ell_j\,,\qquad\forall j\,,
  \end{equation}
  We now distinguish two cases:

  \medskip

  \noindent {\it Case one, $\e_*>0$}: In this case, by continuity of $\psi$ (see Theorem \ref{theorem existence solutions}) and since
  \[
  \ac_{\e_*}(u_j)-\psi(\e_*)\le b_j\,\big(\ac_{\e_j}(u_j)-\psi(\e_j)\big)+b_j\,\psi(\e_j)-\psi(\e_*)\,,\qquad
  b_j=\max\Big\{\frac{\e_j}{\e_*},\frac{\e_*}{\e_j}\Big\}\,,
  \]
  we can assume that $\ac_{\e_*}(u_j)-\psi(\e_*)\le\ell_0$ for $\ell_0$ as in step two of the proof of Theorem \ref{theorem existence solutions}. We can thus apply that statement and conclude that, up to translations and up to subsequences, there is $u\in H^1(\R^n;[0,1])$ with $\int_{\R^n}V(u)=1$ such that $d_\Phi(u_j,u)\to 0$ as $j\to\infty$. In particular, $u$ is a minimizer of $\psi(\e_*)$, and therefore, up to a translation, we can assume that $u=u_{\e_*}\in\RR_0$. Now, by repeating this same argument with the minimizers $u_{\e_j}$ of $\psi(\e_j)$ in $\RR_0$ in place of $u_j$, we see that
  \[
  d_\Phi(u_{\e_j},u_{\e_*})\to 0\qquad\mbox{as $j\to\infty$}\,,
  \]
  so that, thanks to \eqref{careful with that}, we find the contradiction
  \[
  \a_*\le\a_{\e_j}(u_j)\le d_\Phi(u_j,u_{\e_j})\le d_\Phi(u_j,u_{\e_*})+C\,d_\Phi(u_{\e_j},u_{\e_*})^{(n-1)/n}\to 0^+\,,
  \]
  as $j\to\infty$.

  \medskip

  \noindent {\it Case two, $\e_*=0$}: In this case, thanks to \eqref{a zero 2},
  \[
  2\,|D[\Phi(u_j)]|(\R^n)\le\ac_{\e_j}(u_j)\le\psi(\e_j)+\ell_j\le 2\,n\,\om_n^{1/n}+C\,\e_j+\ell_j\,,
  \]
  so that $\{\Phi(u_j)\}_j$ is asymptotically optimal for the sharp $BV$-Sobolev inequality. By  the concentration-compactness principle
  (see, e.g., \cite[Theorem A.1]{fuscomaggipratelliBV}), up to subsequences and up to translations, $\Phi(u_j)\to a\,1_{B_r}$ in $L^{n/(n-1)}(\R^n)$ as $j\to\infty$, for some $a$ and $r$ such that $a^{n/(n-1)}\,\om_n\,r^n=1$. The fact that $\ac_{\e_j}(v_j)$ is bounded implies that $v_j\to \{0,1\}$ a.e. on $\R^n$, therefore, by $\Phi(0)=0$ and $\Phi(1)=1$, it must be $a=1$ and $R=R_0$ for $\om_n\,R_0^n=1$. By Theorem \ref{theorem asymptotics of minimizers}, if $u_{\e_j}$ is a the minimizer of $\psi(\e_j)$ in $\RR_0$, then
  \[
  d_\Phi(u_{\e_j},1_{B_{R_0}})\to 0\qquad\mbox{as $j\to\infty$}\,,
  \]
  which gives the contradiction
  \[
  \a_*\le\a_{\e_j}(u_j)\le d_\Phi(u_j,u_{\e_j})\le d_\Phi(u_j,1_{B_{R_0}})+C\,d_\Phi(u_{\e_j},1_{B_{R_0}})^{(n-1)/n}\to 0^+\,,
  \]
  as $j\to\infty$.
\end{proof}

\subsection{Proof of Theorem \ref{thm main 2} in the radial decreasing case}\label{section selection principle} We start by noticing that, thanks to the results proved in the previous sections, we can quickly prove Theorem \ref{thm main 2} for functions in $\RR_0$.
	
\begin{theorem}\label{thm main 2 radial}
		If $n\ge 2$ and $W\in C^{2,1}[0,1]$ satisfies \eqref{W basic} and \eqref{W normalization}, then there exist universal constants $C$ and $\e_0$ such that, for every $\e<\e_0$, denoting by $u_\e$ the unique minimizer of $\psi(\e)$ in $\RR_0$, one has
		\begin{equation}\label{global quantitative estimate radial}
		C\,\sqrt{\de_\e(u)}\ge d_\Phi(u,u_\e)\,,
		\end{equation}
		whenever $u\in H^1(\R^n;[0,1])\cap \RR_0$ with $\int_{\R^n} V(u)=1$.
\end{theorem}	

\begin{proof} Arguing by contradiction, we can find $\e_j\to 0^+$ and $\{v_j\}_j$ in $H^1(\R^n;[0,1])\cap\RR_0$ with
\[
\int_{\R^n}V(v_j)=1\,,\qquad a_j=\frac{\ac_{\e_j}(v_j)-\psi(\e_j)}{d_\Phi(v_j,u_j)^2}\to 0\qquad\mbox{as $j\to\infty$}\,,
\]
where $u_j=u_{\e_j}$ and, thanks to Lemma \ref{lemma uniform qualitative stability} and to $a_j\to 0^+$, we have
\begin{equation}
  \label{by step one}
  \lim_{j\to\infty}d_\Phi(v_j,u_j)=0\,.
\end{equation}
Correspondingly we consider the variational problems
\[
\g_j=\g(\e_j,a_j,v_j)=\inf\Big\{\ac_{\e_j}(w)+a_j\,d_\Phi(w,v_j):w\in H^1(\R^n;[0,1])\,,\int_{\R^n}V(w)=1\Big\}\,.
\]
With $a_0$, $\ell_0$ and $\e_0$ as in Theorem \ref{theorem compactness for selection principle}, we notice that, for $j$ large enough, we have $a_j\in(0,a_0)$, $\e_j<\e_0$, and
  \begin{equation}
    \label{spr hp with j}
    \ac_{\e_j}(v_j)\le\psi(\e_j)+a_j\,\ell_0\,,\qquad d_\Phi(v_j,u_j)\le \ell_0\,.
  \end{equation}
In particular we can apply Theorem \ref{theorem compactness for selection principle}, and deduce the existence of minimizers $w_j$  of $\g_j$. We claim that, as $j\to\infty$,
\begin{equation}
  \label{still contradiction}
  \lim_{j\to\infty}\frac{\ac_{\e_j}(w_j)-\psi(\e_j)}{d_\Phi(w_j,u_j)^2}=0\,.
\end{equation}
To show this, we first notice that, by comparing $w_j$ to $u_j$ we have
\[
\ac_{\e_j}(w_j)+a_j\,d_\Phi(w_j,v_j)\le \psi(\e_j)+a_j\,d_\Phi(u_j,v_j)\,,
\]
so that \eqref{by step one} gives $\de_{\e_j}(w_j)\to 0$, and then Lemma \ref{lemma uniform qualitative stability} implies
\begin{equation}
  \label{by step one two}
  \lim_{j\to\infty}d_\Phi(w_j,u_j)=0\,.
\end{equation}
Next, comparing $w_j$ to $v_j$ we find that
\[
\ac_{\e_j}(w_j)+a_j\,d_\Phi(w_j,v_j)\le\ac_{\e_j}(v_j)\,,
\]
so that $\psi(\e_j)\le \ac_{\e_j}(w_j)$ and the definition of $a_j$ give
\begin{equation}
  \label{bye}
  d_\Phi(w_j,v_j)\le\frac{\ac_{\e_j}(v_j)-\psi(\e_j)}{a_j}=d_\Phi(v_j,u_j)^2\,.
\end{equation}
By \eqref{careful with that}, \eqref{by step one}, \eqref{by step one two}, and \eqref{bye} we find
\begin{eqnarray*}
\big|d_\Phi(w_j,u_j)-d_\Phi(v_j,u_j)\big|&\le&
C\,\max\big\{d_\Phi(w_j,u_j),d_\Phi(v_j,u_j)\big\}^{1/n}\,d_\Phi(w_j,v_j)^{(n-1)/n}
\\
&=&{\rm o}\big(d_\Phi(v_j,u_j)^{2(n-1)/n}\big)\,,
\end{eqnarray*}
where $2(n-1)/n\ge1$ thanks to $n\ge2$. Thus, $d_\Phi(w_j,u_j)\ge d_\Phi(v_j,u_j)/C$ for $j$ large enough, and $\ac_{\e_j}(w_j)\le \ac_{\e_j}(v_j)$ gives
\[
\frac{\ac_{\e_j}(w_j)-\psi(\e_j)}{d_\Phi(w_j,u_j)^2}\le C\,\frac{\ac_{\e_j}(v_j)-\psi(\e_j)}{d_\Phi(v_j,u_j)^2}\to 0^+\,,
\]
as claimed in \eqref{still contradiction}.

\medskip

We now derive a contradiction to \eqref{still contradiction}. By Theorem \ref{theorem compactness for selection principle}, we know that $w_j\in \RR_0^*\cap C^{2,1/(n-1)}_{{\rm loc}}(\R^n)$, $0<w_j<1$ on $\R^n$, and
\begin{equation}
	\label{EL classic W V Z jjj}
	-2\,\e_j^2\,\Delta w_j=\e_j\,w_j\,(1-w_j)\,\Err_j-W'(w_j)\qquad\mbox{on $\R^n$}\,,
\end{equation}
where $\Err_j$ is a continuous radial function on $\R^n$ with
\begin{equation}
  \label{err eps jjj}
  \sup_{\R^n}|\Err_j|\le C\,.
\end{equation}
We can thus apply Theorem \ref{theorem asymptotics of minimizers} to $w_j$. In particular, since both $u_j$ and $w_j$ obey the resolution formula \eqref{critical sq resolution of vj}, we have that $h_j=w_j-u_j$ satisfies
\begin{equation}
  \label{what about hj}
  |h_j(R_0+\e_j\,s)|\le C\,\e_j\,e^{-|s|/C}\qquad\forall s\ge-\frac{R_0}{\e_j}\,.
\end{equation}
In particular,
\[
\|h_j\|_{L^\infty(\R^n)}\le C\,\e_j\,,\qquad \int_{\R^n}h_j^2\le C\,\e_j\,.
\]
and we can thus apply Theorem \ref{theorem fuglede estimate} to deduce
\begin{eqnarray}\nonumber
\ac_{\e_j}(w_j)-\psi(\e_j)&\geq& \frac1{C}\,\int_{\R^n}\e_j\,|\nabla h_j|^2+\frac{h_j^2}{\e_j}
\\\label{ch1}
&\ge&\frac1{C}\,\int_{\R^n}|\nabla(h_j^2)|\ge \frac1{C}\,\Big(\int_{\R^n}|h_j|^{2n/(n-1)}\Big)^{(n-1)/n}\,,
\end{eqnarray}
where we have also used the $BV$-Sobolev inequality. By \eqref{what about hj}, and by applying \eqref{critical sq vj estimates on R infinity} to $u_j$ in combination with \eqref{W near the wells}, we find that, if $A_j=B_{R_0+ c_j}\setminus B_{R_0-b_j}$, then, for every $x\in\R^n\setminus A_j$ we have
\[
|\Phi(u_j(x))-\Phi(w_j(x))|\le |h_j(x)|\,\int_0^1\,\sqrt{W(u_j(x)+t\,h_j(x))}\,dt\le C\,|h_j(x)|\,e^{-||x|-R_0|/C\,\e_j}\,,
\]
and, therefore,
\begin{eqnarray}\nonumber
\int_{\R^n\setminus A_j}|\Phi(u_j)-\Phi(w_j)|^{n/(n-1)}&\le&
C\,\int_{\R^n\setminus A_j}|h_j|^{n/(n-1)}\,e^{-||x|-R_0|/C\,\e_j}
\\\label{ch2}
&\le&\,C\,\sqrt{\e_j}\,\Big(\int_{\R^n}|h_j|^{2\,n/(n-1)}\Big)^{1/2}\,.
\end{eqnarray}
If, instead, $x\in A_j$, then by $|\Phi(u_j)-\Phi(w_j)|\le C\,|h_j|$ and $\L^n(A_j)\le C\,\e_j$ we find
\begin{equation}
  \label{ch3}
  \int_{A_j}|\Phi(u_j)-\Phi(w_j)|^{n/(n-1)}\le C\,\sqrt{\e_j}\,\Big(\int_{\R^n}|h_j|^{2\,n/(n-1)}\Big)^{1/2}\,.
\end{equation}
By combining \eqref{ch1}, \eqref{ch2} and \eqref{ch3}, and thanks to $\e_j\le 1$, $n/(n-1)\ge1$, and $\de_{\e_j}(w_j)\le1$, we conclude that
\[
d_\Phi(u_j,w_j)\le C\,\sqrt{\e_j}\,\de_{\e_j}(w_j)^{n/2\,(n-1)}\le C\,\sqrt{\de_{\e_j}(w_j)}\,,
\]
in contradiction to \eqref{still contradiction}.
\end{proof}

\begin{remark}
  {\rm The argument we have just presented provides further indication that \eqref{global quantitative estimate radial} should not provide a sharp rate on radial decreasing functions. The sharp stability estimate on small radial perturbations of $u_\e$ is clearly given in Theorem \ref{theorem fuglede estimate}, but it is not clear what form the sharp stability estimate should take on $\RR_0$ (or, more generally, on arbitrary radial functions).}
\end{remark}

\subsection{Reduction to radial decreasing functions}\label{subsec reduction to radial} We now discuss the reduction of \eqref{the end 1} to the case of radial decreasing functions. We do this by adapting to our setting the ``quantitative symmetrization'' strategy developed in \cite{fuscomaggipratelli,fuscomaggipratelliBV} in the study of Euclidean isoperimetry.

\medskip

Given $n\ge 2$ and $k\in\{1,...,n\}$ we say that $u:\R^n\to\R$ is {\bf $k$-symmetric} if there exists $k$ mutually orthogonal hyperplanes such that $u$ is symmetric by reflection through each of these hyperplanes. The class of $n$-symmetric functions is particularly convenient when it comes to quantify sharp inequalities involving radial decreasing rearrangements. Consider for example the P\'olya-Szeg\"o inequality
\begin{equation}
  \label{ps inq final}
  \int_{\R^n}|\nabla u|^2\ge\int_{\R^n}|\nabla u^*|^2\,,
\end{equation}
where $u^*$ is the radial decreasing rearrangement of $u$. A classical result of Brothers and Ziemer \cite{brothersziemer} shows that equality can hold in \eqref{ps inq final} without $u$ being a translation of $u^*$; in general, the additional condition that $(u^*)'<0$ a.e. must be assumed to deduce symmetry from equality in \eqref{ps inq final} (compare with step six in the proof of Theorem \ref{theorem existence solutions}). However, if $u$ is $n$-symmetric, then equality in \eqref{ps inq final} automatically implies that $u$ is radial decreasing. A quantitative version of this statement is proved in \cite[Theorem 2.2]{fuscomaggipratelliBV} in the $BV$-case of \eqref{ps inq final}, and in \cite[Theorem 3]{cianchifuscomaggipratelliSOBOLEV} in the Sobolev case. The following theorem is an adaptation of those results to our setting.

\begin{theorem}[Reduction from $n$-symmetric to radial decreasing functions]\label{theorem n symmetric to radial}
  If $n\ge 2$ and $W\in C^{2,1}[0,1]$ satisfies \eqref{W basic} and \eqref{W normalization}, then there exists a universal constant $C$ with the following property. If $u\in H^1(\R^n;[0,1])$ is a $n$-symmetric function with $\int_{\R^n}V(u)=1$ and $u^*$ is its radial decreasing rearrangement, then
  \begin{equation}\label{Allen Cahn symmetrization}
	d_\Phi(u,u^*)\le C\,\Big(\int_{\R^n}W(u) \Big)^{1/2}\,\Big(\int_{\R^n} |\nabla u|^2-\int_{\R^n} |\nabla u^*|^2\Big)^{1/2}\,.
  \end{equation}		
  Moreover, for every $\e>0$ we have
  \begin{equation}\label{eq reduction to symmetric case}
  \alpha_\e(u)\leq C\,\Big(\alpha_\e(u^*)+ \big(\AC(u)\,\delta_\e(u)\big)^{1/2}\,\Big)\,.
  \end{equation}
\end{theorem}

\begin{proof} We first claim that
  \begin{eqnarray}
  \label{CFMP modfied}
  d_\Phi(u,u^*) &\leq&\frac{n}{n-1}\,\int_0^1 \L^n(E_t) \Phi(t)^{1/(n-1)}\,\sqrt{W(t)}\,dt\,,
  \\
  \label{lower bound defficit}
  \int_{\R^n} |\nabla u|^2-	\int_{\R^n} |\nabla u^*|^2&\geq&\frac1{C(n)}\, \int_0^1 \Big( \frac{\L^n(E_t)}{\mu(t)}\Big)^2 \frac{\mu(t)^{2\,(n-1)/n}}{-\mu'(t)}\,dt\,,
  \end{eqnarray}
  where $E_t=\{u>t\}\Delta\{u^*>t\}$, $\mu(t)=\L^n(\{u>t\})$, and $\mu'(t)$ denotes the absolutely continuous part of the distributional derivative of the decreasing function $\mu$. To prove \eqref{CFMP modfied} we recall that, by \cite[Lemma 5]{cianchifuscomaggipratelliSOBOLEV}, we have
  \[
  d_\Phi(u,u^*)\le \frac{n}{n-1}\,\int_0^1\,\L^n(F_s)\,s^{1/(n-1)}\,ds\,,
  \]
  provided $F_s=\{\Phi(u)>s\}\Delta \{\Phi(u^*)>s\}$. Since $\Phi$ is strictly increasing, we have $F_{\Phi(t)}=E_t$, so that the change of variables $s=\Phi(t)$ gives \eqref{CFMP modfied}. To prove \eqref{lower bound defficit} we just notice that this is \cite[Equation (3.18)]{cianchifuscomaggipratelliSOBOLEV}. Now, by H\"older inequality and \eqref{CFMP modfied}, we find that
  \begin{eqnarray}\nonumber
  \int_0^1 \L^n(E_s)\, \Phi^{1/(n-1)}\,\sqrt{W}&=&
  \int_0^1 \frac{\L^n(E_s)}{\mu}\,\frac{\mu^{(n-1)/n}}{(-\mu')^{1/2}}\, \frac{(-\mu')^{1/2}}{\mu^{-1/n}}\,\Phi^{1/(n-1)}\,\sqrt{W}
  \\\nonumber
  &\le&
  \Big(\int_0^1 \Big(\frac{\L^n(E_s)}{\mu}\Big)^2 \frac{\mu^{2\,(n-1)/n}}{-\mu'}\Big)^{1/2}\,\Big( \int_0^1 \frac{-\mu'}{\mu^{-2/n}}\,\Phi^{2/(n-1)}\,W\,\Big)^{1/2}\,.
  \end{eqnarray}
  By $1=\int_{\R^n} V(u)\ge V(t)\,\mu(t)$ for every $t\in(0,1)$, we have
  \[
  \int_0^1 \frac{-\mu'}{\mu^{-2/n}}\,\Phi^{2/(n-1)}\,W\le\int_0^1 -\mu'\,\big(V\,\mu\big)^{2/n}\,W\le\int_0^1\,-\mu'\,W\le\int_{\R^n}W(u)\,,
  \]
  where in the last inequality we have used $-\mu'\,d\L^1\le -D\mu$, integration by parts and Fubini's theorem to deduce
  \[
  -\int_0^1\,W\,d[D\mu]=\int_0^1\,W'(t)\,\mu(t)\,dt=\int_{\R^n}\,dx\int_0^{u(x)}W'(t)\,dt=\int_{\R^n}W(u)\,.
  \]
  By combining \eqref{CFMP modfied}, \eqref{lower bound defficit} and these estimates we find \eqref{Allen Cahn symmetrization}. To prove \eqref{eq reduction to symmetric case}, we notice that, by $\int_{\R^n}W(u)=\int_{\R^n}W(u^*)$ and $\int_{\R^n}V(u^*)=1$, \eqref{Allen Cahn symmetrization} gives
  \begin{equation}\label{defficit symmetric and spherical}
  d_\Phi(u,u^*)\leq C\,\AC(u)^{1/2}\,\Big(\AC(u)-\AC(u^*)\Big)^{1/2}\le C\,\AC(u)^{1/2}\,\de_\e(u)^{1/2}
  \end{equation}
  and then \eqref{eq reduction to symmetric case} follows by the triangular inequality in $L^{n/(n-1)}(\R^n)$.
\end{proof}

Next we discuss the reduction from generic functions to $n$-symmetric ones.

\begin{theorem}
  [Reduction to $n$-symmetric functions]\label{theorem symmetrization reduction} If $n\ge 2$ and $W\in C^{2,1}[0,1]$ satisfies \eqref{W basic} and \eqref{W normalization}, then there exist universal constants $\e_0$ and $\de_0$ with the following property. If $u\in H^1(\R^n;[0,1])$, $\int_{\R^n}V(u)=1$ and $\de_\e(u)\le \de_0$ for some $\e<\e_0$, then there exists $v\in H^1(\R^n;[0,1])$ with $\int_{\R^n}V(v)=1$ such that $v$ is $n$-symmetric and
  \begin{equation}
    \label{reduction inequalities}
      \a_\e(u)\le C\,\a_\e(v)\,,\qquad \de_\e(v)\le C\,\de_\e(u)\,.
  \end{equation}
\end{theorem}
	
\begin{proof} Without loss of generality we can assume that $\de_\e(u)\le\de_0$ for a universal constant $\de_0$. By Lemma \ref{lemma uniform qualitative stability} we can choose $\de_0$ so that $\a_\e(u)\le\a_0$ for $\a_0$ a universal constant of our choice. We divide the proof into a few steps.

\medskip

\noindent {\it Step one}: We prove that, if $u$ is $k$-symmetric, $\{H_i\}_{i=1}^k$ are the mutually orthogonal hyperplanes of symmetry of $u$, and $J=\bigcap_{i=1}^k H_i$, then
\begin{equation}
  \label{lemma changing centers}
  \a_\e(u;J)=\inf_{x\in J}\,d_\Phi(u,T_x u_\e)\le C(n)\,\a_\e(u)\,.
\end{equation}
In other words, in computing the asymmetry of $u$ in the proof of an estimate like \eqref{the end 1}, we can compare $u$ with a translation of $u_\e$ with maximum on $J$.

\medskip

Indeed, let $x_0\in\R^n$ be such that $\alpha_\e(u)=d_\Phi(u,T_{x_0}u_\e)$. Without loss of generality, we can assume $x_0\not\in J$. In particular, if $y_0$ denotes the reflection of $x_0$ with respect to $J$, then $y_0\ne x_0$ and
\begin{equation}\label{eq reflection symmetry}
d_\Phi(u,T_{y_0}u_\e)=d_\Phi(u,T_{x_0}u_\e)=\a_\e(u)\,,
\end{equation}
that is, also $y_0$ is an optimal center for computing $\a_\e(u)$. Let $z_0=(x_0+y_0)/2$, so that $z_0\in J$, let $\nu=(x_0-y_0)/|x_0-y_0|$ (which is well defined by $x_0\ne y_0$), and let $H$ be the open half-space orthogonal to $\nu$, containing $x_0$, and such that $z_0\in\pa H$.
By $T_{z_0+t\,\nu}u_\e(x)=u_\e(x-z_0-t\,\nu)$, we have that
\[
\frac{d}{dt}\,T_{z_0+t\,\nu}u_\e(x)=-\nu\cdot\frac{x-z_0-t\,\nu}{|x-z_0-t\,\nu|}\,u_\e'(|x-x_0-t\,\nu|)>0\,,\qquad\forall x\in H\,,t<0\,,
\]
since $u_\e'<0$, and since the fact that $\nu$ points inside $H$ gives
\begin{eqnarray*}
  &&(z-z_0)\cdot\nu>0\,,\qquad\forall z\in H\,,
  \\
  &&z=x-t\,\nu\in H\,,\qquad\forall x\in H\,,t<0\,.
\end{eqnarray*}
We thus find that, if $t<0$,
\begin{eqnarray*}
  &&\frac{d}{dt}\,\int_{H} |\Phi(T_{x_0}u)-\Phi(T_{z_0+t\,\nu}u_\e)|^{n/(n-1)}
  \\
  &=&\frac{n}{n-1}\,\int_{H}|\Phi(u)-\Phi(T_{z_0+t\,\nu}u_\e)|^{1/(n-1)}\,\sqrt{W(T_{z_0+t\,\nu}u_\e)}\,\frac{d}{dt}\,T_{z_0+t\,\nu}u_\e>0\,,
\end{eqnarray*}
so that
\begin{eqnarray}\nonumber
 \int_{H} |\Phi(T_{x_0}u)-\Phi(T_{y_0}u_\e)|^{n/(n-1)}&=&\int_{H^-} |\Phi(T_{x_0}u)-\Phi(T_{z_0+t\,\nu}u_\e)|^{n/(n-1)}\Big|_{t=-|x_0-y_0|/2}
 \\\nonumber
 &\le&\int_{H} |\Phi(T_{x_0}u)-\Phi(T_{z_0+t\,\nu}u_\e)|^{n/(n-1)}\Big|_{t=0}
 \\\label{eq decreasing}
 &\le&\int_{H} |\Phi(T_{x_0}u)-\Phi(T_{z_0}u_\e)|^{n/(n-1)}
\end{eqnarray}
Now, since both $u$ and $T_{z_0}u_\e$ are symmetric by reflection with respect to $\pa H$, we have that
\begin{equation}\label{eq sym hminus}
\int_{\R^n} |\Phi(u)-\Phi(T_{z_0}u_\e)|^{n/(n-1)}=2\,\int_H |\Phi(u)-\Phi(T_{z_0}u_\e)|^{n/(n-1)}\,,
\end{equation}		
therefore, by \eqref{eq reflection symmetry}, \eqref{eq decreasing} and \eqref{eq sym hminus} we conclude that
\begin{eqnarray*}
  \a_\e(u;J)&\le& d_\Phi(u,T_{z_0}u_\e)=2\,\int_H |\Phi(u)-\Phi(T_{z_0}u_\e)|^{n/(n-1)}
  \\
  &\le& C(n)\,\Big(\int_H |\Phi(u)-\Phi(T_{x_0}u_\e)|^{n/(n-1)}
  +\int_H |\Phi(T_{x_0}u_\e)-\Phi(T_{z_0}u_\e)|^{n/(n-1)}\Big)
  \\
  &\le& C(n)\,\Big(\a_\e(u)+ \int_H |\Phi(T_{y_0}u_\e)-\Phi(T_{x_0}u_\e)|^{n/(n-1)}\Big)
  \\
  &\le& C(n)\,\Big(\a_\e(u)+ d_\Phi(T_{y_0}u_\e,T_{x_0}u_\e)\Big)
  \\
  &\le& C(n)\,\Big(\a_\e(u)+ d_\Phi(T_{y_0}u_\e,u)+d_\Phi(u,T_{x_0}u_\e)\Big)=C(n)\,\a_\e(u)\,,
\end{eqnarray*}
that is \eqref{lemma changing centers}.

\medskip

\noindent {\it Step two}: Let $H_1$ and $H_2$ be two orthogonal hyperplanes through the origin, let $H_i^\pm$ be the half-spaces defined by $H_i$, and let $x_i^\pm\in\pa H_i$. For $i=1,2$, consider the functions
\[
U[u_\e,H_i,x_i^+,x_i^-]=1_{H_i^+}\,T_{x_i^+}u_\e+1_{H_i^-}\,T_{x_i^-}u_\e\,,
\]
obtained by ``gluing'' the restriction of $u_\e$ to $H_i^+$ translated by $x_1^+$ to the restriction of $u_\e$ to $H_i^-$ translated by $x_1^+$ (notice that translating by $x_i^\pm$ brings $H_i^+$ and $H_i^-$ into themselves). Setting for brevity
\[
U_{\e,i}=U[u_\e,H_i,x_i^+,x_i^-]
\]
we claim that, for every $a\in(0,1)$ there is $\kappa=\kappa(a,n,W)>0$ such that if
\begin{equation}
  \label{vicini}
  \max\Big\{|x_1^+-x_1^-|,|x_2^+-x_2^-|,|x_1^+-x_2^+|\Big\}\le \kappa\,,
\end{equation}
then, for every $\e<\e_0$,
\begin{equation}
  \label{fundamental}
  \max\Big\{d_\Phi(T_{x_1^+}u_\e,T_{x_1^-}u_\e),d_\Phi(T_{x_2^+}u_\e,T_{x_2^-}u_\e)\Big\}\le\frac{8}{1-a}\, d_\Phi(U_{\e,1},U_{\e,2})\,.
\end{equation}
Indeed, since $H_1$ and $H_2$ are hyperplanes through the origin and $u_\e\in\RR_0$, we have
\[
\int_{H_1^\pm}V(T_{x_1^\pm}u_\e)=\frac12\,,\qquad \int_{H_2^\pm}V(T_{x_2^\pm}u_\e)=\frac12\,.
\]
It is in general not true that, say, $H_1^+\cap H_2^+$ has measure $1/4$ for either $V(T_{x_1^\pm}u_\e)\,dx$ or $V(T_{x_1^\pm}u_\e)\,dx$. However, provided we choose $\kappa$ sufficiently small, thanks to Theorem \ref{theorem asymptotics of minimizers}, we can definitely ensure that, for every $\e<\e_0$ and $\b,\g\in\{+,-\}$, we have
\[
\int_{H_1^\b\cap H_2^\g}|\Phi(T_{x_1^\b}u_\e)-\Phi(T_{x_2^\g}u_\e)|^{n/(n-1)}\ge \frac{1-a}4\,d_\Phi(T_{x_1^\b}u_\e,T_{x_2^\g}u_\e)\,.
\]
Correspondingly,
\begin{eqnarray*}
  d_\Phi(U_{\e,1},U_{\e,2})&\ge&
  \int_{H_1^\b\cap H_2^\g}|\Phi(U_{\e,1})-\Phi(U_{\e,2})|^{n/(n-1)}
  \\
  &=&\int_{H_1^\b\cap H_2^\g}|\Phi(T_{x_1^\b}u_\e)-\Phi(T_{x_2^\g}u_\e)|^{n/(n-1)}\ge \frac{1-a}4\,d_\Phi(T_{x_1^\b}u_\e,T_{x_2^\g}u_\e)\,,
\end{eqnarray*}
and thus
\begin{eqnarray*}
d_\Phi(T_{x_1^+}u_\e,T_{x_1^-}u_\e)^{(n-1)/n}&\le& d_\Phi(T_{x_1^+}u_\e,T_{x_2^+}u_\e)^{(n-1)/n}+d_\Phi(T_{x_2^+}u_\e,T_{x_1^-}u_\e)^{(n-1)/n}
\\
&\le& \Big(\frac{8}{1-a}\,d_\Phi(U_{\e,1},U_{\e,2})\Big)^{(n-1)/n}\,,
\end{eqnarray*}
as claimed.

\medskip

\noindent {\it Step three}: Given $u\in H^1(\R^n;[0,1])$ with $\int_{\R^n}V(u)=1$, we now consider an hyperplane $H$ such that, if $H^+$ and $H^-$ denote the two open half-spaces defined by $H$, then
\[
\int_{H^+}V(u)=\int_{H^-}V(u)=\frac12\,.
\]
Denoting by $\rho_H$ the reflection with respect to $H$, we let
\begin{equation}
  \label{u plusminus}
  u^+=1_{H^+}\,u+1_{H^-}\,(u\circ \rho_H)\,,\qquad u^-=1_{H^-}\,u+1_{H^+}\,(u\circ \rho_H)\,,
\end{equation}
and notice that $u^\pm\in H^1(\R^n;[0,1])$, with
\begin{equation}
  \label{what about u nu pm}
  2\,\AC(u)=\AC(u^+)+\AC(u^-)\,,\qquad \int_{\R^n}V(u^+)=\int_{\R^n}V(u^-)=1\,.
\end{equation}
We claim that
\begin{equation}
  \label{general reduction H}
  \max\{\de_\e(u^+),\de_\e(u^-)\}\le 2\,\de_\e(u)\,,\qquad \a_\e(u)\le C(n)\,\Big\{\a_\e(u^+)+\a_\e(u^-)+d_\Phi(T_{x^+}u_\e,T_{x^-}u_\e)\Big\}\,,
\end{equation}
provided $T_{x^+}u_\e=$ and $T_{x^-}u_\e=T_{x^-}u_\e$ are such that $x^+,x^-\in H$ with
\[
\a_\e(u^+;H)=d_\Phi(u^+,T_{x^+}u_\e)\,,\qquad \a_\e(u^-;H)=d_\Phi(u^-,T_{x^-}u_\e)\,.
\]
The first inequality in \eqref{general reduction H} is obvious from \eqref{what about u nu pm}. To prove the second one we notice that
\begin{eqnarray*}
  \a_\e(u)&\le&d_\Phi(u,T_{x^+}u_\e)=\int_{H^+}|\Phi(u)-\Phi(T_{x^+}u_\e)|^{n/(n-1)}+\int_{H^-}|\Phi(u)-\Phi(T_{x^+}u_\e)|^{n/(n-1)}
  \\
  &=&\int_{H^+}|\Phi(u^+)-\Phi(T_{x^+}u_\e)|^{n/(n-1)}+\int_{H^-}|\Phi(u^-)-\Phi(T_{x^+}u_\e)|^{n/(n-1)}
  \\
  &\le&C(n)\,\Big\{d_\Phi(u^+,T_{x^+}u_\e)+d_\Phi(u^-,T_{x^-}u_\e)+d_\Phi(T_{x^-}u_\e,T_{x^+}u_\e)\Big\}\,,
\end{eqnarray*}
that is the second inequality in \eqref{general reduction H}.

\medskip

With these preliminary considerations in place, we now prove that if $u\in H^1(\R^n;[0,1])$ with $\int_{\R^n}V(u)=1$, if $H_1$ and $H_2$ are orthogonal hyperplanes such that the corresponding half-spaces $H_i^\pm$ satisfy
\[
\int_{H_i^\pm}V(u)=\frac12\,,
\]
if $u_i^\pm$ as in \eqref{u plusminus} starting from $H_i$, then there is at least one $v\in\{u_1^+,u_1^-,u_2^+,u_2^-\}$ such that \eqref{reduction inequalities} holds. Given that $\de_\e(v)\le 2\,\de_\e(u)$ for every $v\in\{u_1^+,u_1^-,u_2^+,u_2^-\}$, we need to show that
\begin{equation}
  \label{davvero}
  \mbox{$\exists v\in\{u_1^+,u_1^-,u_2^+,u_2^-\}$ such that $\a_\e(u)\le C\,\a_\e(v)$}\,.
\end{equation}
Denoting by $x_i^\pm$ the points in $H_i$ such that
\[
\a_\e(u_i^\pm;H_i)=d_\Phi(u_i^\pm,T_{x_i^\pm}u_\e)\,,
\]
we notice that \eqref{davvero} follows if we can show that, provided $\a_0$ is small enough, then
\begin{eqnarray}\label{the end prime}
  \mbox{either}&&\qquad d_\Phi(T_{x_1^+}u_\e,T_{x_1^-} u_\e)\le M\,\Big\{\a_\e(u_1^+;H_1)+\a_\e(u_1^-;H_1)\Big\}
  \\\label{the end second}
  \mbox{or}&&\qquad d_\Phi(T_{x_2^+}u_\e,T_{x_2^-} u_\e)\le M\,\Big\{\a_\e(u_2^+;H_2)+\a_\e(u_2^-;H_2)\Big\}\,,
\end{eqnarray}
for a constant $M$ (as it turns out, any $M>16$ works). Indeed, if, for example, \eqref{the end prime} holds, then \eqref{lemma changing centers} and \eqref{general reduction H} with $H=H_1$ give
\[
\a_\e(u)\le C\,\Big\{\a_\e(u_1^+)+\a_\e(u_1^-)+\a_\e(u_1^+;H_1)+\a_\e(u_1^-;H_1)\Big\}
\le C\,\Big\{\a_\e(u_1^+)+\a_\e(u_1^-)\Big\}\,,
\]
and then either $C\,\a_\e(u_1^+) \ge\a_\e(u)$ or $C\,\a_\e(u_2^+) \ge\a_\e(u)$; in particular, \eqref{davvero} holds. We now want to prove that either \eqref{the end prime} or \eqref{the end second} holds. We argue by contradiction. Recalling that $\a_\e(u_i^\pm;H_i)=d_\Phi(u_i^\pm,T_{x_i^\pm}u_\e)$, let us thus assume that both
 \begin{eqnarray}\label{the end prime not}
  &&d_\Phi(T_{x_1^+}u_\e,T_{x_1^-} u_\e)>M\,\Big\{d_\Phi(u_1^+,T_{x_1^+}u_\e)+d_\Phi(u_1^-,T_{x_1^-}u_\e)\Big\}\,,
  \\\label{the end second not}
  &&d_\Phi(T_{x_2^+}u_\e,T_{x_2^-} u_\e)>M\,\Big\{d_\Phi(u_2^+,T_{x_2^+}u_\e)+d_\Phi(u_2^-,T_{x_2^-}u_\e)\Big\}\,,
\end{eqnarray}
hold for $M$ to be determined. In particular, if $U_{\e,i}$, $i=1,2$, are defined as in step two, and $\a_0$ is small enough that \eqref{vicini} holds, then, by \eqref{fundamental}, we have
\begin{eqnarray*}
&&\max\Big\{d_\Phi(T_{x_1^+}u_\e,T_{x_1^-}u_\e),d_\Phi(T_{x_2^+}u_\e,T_{x_2^-}u_\e)\Big\}^{(n-1)/n}
\\
&\le&\Big(\frac{8}{1-a}\, d_\Phi(U_{\e,1},U_{\e,2})\Big)^{(n-1)/n}\le\Big(\frac{8}{1-a}\Big)^{(n-1)/n}\,\sum_{i=1}^2 d_\Phi(U_{\e,i},u)^{(n-1)/n}
\\
&=&\Big(\frac{8}{1-a}\Big)^{(n-1)/n}\,\sum_{i=1}^2\,\Big(\sum_{\b=+,-}\int_{H_i^\b}|\Phi(T_{x_i^\b}u_\e)-\Phi(u_i^\b)|^{n/(n-1)}\Big)^{(n-1)/n}
\\
&\le&\Big(\frac{8}{M\,(1-a)}\Big)^{(n-1)/n}\,\sum_{i=1}^2\,\Big(d_\Phi(T_{x_i^+}u_\e,T_{x_i^-} u_\e)\Big)^{(n-1)/n}
\\
&\le&\Big(\frac{16}{M\,(1-a)}\Big)^{(n-1)/n}\,\max\Big\{d_\Phi(T_{x_1^+}u_\e,T_{x_1^-}u_\e),d_\Phi(T_{x_2^+}u_\e,T_{x_2^-}u_\e)\Big\}^{(n-1)/n}\,.
\end{eqnarray*}
We fix $M>16$ and apply the above with $a\in(0,1)$ such that $M\,(1-a)>16$. We find that either $x_1^+=x_1^-$ (a contradiction to \eqref{the end prime not}), or $x_2^+=x_2^-$ (a contradiction to \eqref{the end second not}).

\medskip

\noindent {\it Step four}: We now pick a family of $n$ mutually orthogonal hyperplanes $\{H_i\}_{i=1}^n$ such that, denoting by $H_i^\pm$ the corresponding half-spaces, we have
\[
\int_{H_i^\pm}V(u)=\frac12\qquad\forall i=1,...,n\,.
\]
Considering the hyperplanes in pairs and arguing inductively on step three, up to a relabeling we reduce to a situation where there exists a function $v$, symmetric by reflection with respect to each $H_i$, $i=1,...,n-1$, and such that
\[
\a_\e(u)\le C\,\a_\e(v)\,,\qquad \de_\e(v)\le 2^n\,\de_\e(v)\,,\qquad \int_{H_n^\pm}V(v)=\frac12\,.
\]
We can thus consider the functions $v^\pm$ obtained by reflecting $v$ with respect to $H_n$ as in step three. By \eqref{general reduction H} we have
\[
\max\{\de_\e(v^+),\de_\e(v^-)\}\le 2\,\de_\e(v)\,,\qquad \a_\e(u)\le C(n)\,\Big\{\a_\e(v^+)+\a_\e(v^-)+d_\Phi(T_{x^+}u_\e,T_{x^-}u_\e)\Big\}\,,
\]
where $x^+$ and $x^-$ are optimal centers for $\a_\e(v^+;\bigcap_{i=1}^nH_i)$ and $\a_\e(v^-;\bigcap_{i=1}^nH_i)$. However, $\bigcap_{i=1}^nH_i$ {\it is a point}, therefore $x^+=x^-$ and we have actually proved
\[
\a_\e(u)\le C(n)\,\Big\{\a_\e(v^+)+\a_\e(v^-)\Big\}\,.
\]
Either $v^+$ or $v^-$ is an $n$-symmetric function with the required properties.
\end{proof}

\subsection{Proof of Theorem \ref{thm main 2}}\label{section reduction to radial} We finally prove Theorem \ref{thm main 2}. By Theorem \ref{theorem symmetrization reduction} we can directly assume that $u$ is $n$-symmetric. Hence, by Theorem \ref{theorem n symmetric to radial}, we can directly assume that $u\in\RR_0$. For $u\in\RR_0$, the conclusion follows from Theorem \ref{thm main 2 radial}. Theorem \ref{thm main 2} is proved.
	
\section{Proof of the Alexandrov-type theorem}\label{section Alexandrov} In this section we complete the proof of Theorem \ref{theorem main}, including in particular proof of the Alexandrov-type result of part (iv) of the statement. We begin by proving some of the properties of $\Psi(\s,m)$ stated in Theorem \ref{theorem main}-(i) and not yet discussed. We then review, in section \ref{sec radial symm and uniqueness}, some classical uniqueness and symmetry results for semilinear PDEs in relation to our setting. Finally, in section \ref{section final proof} we review how the various results of the paper combines into Theorem \ref{theorem main}.

\subsection{Some properties of $\Psi(\s,m)$}\label{section some properties of Psi} We prove here the properties of $\Psi(\s,m)$ stated in Theorem \ref{theorem main}--(ii). As explained in the introduction, these properties will be crucial in proving Theorem \ref{theorem main}-(iv).

\begin{theorem}\label{theorem basics of Psi}
  If $n\ge 2$ and $W\in C^{2,1}[0,1]$ satisfies \eqref{W basic} and \eqref{W normalization}, then there exists a universal constant $\e_0$ such that, setting
  \[
  \X(\e_0)=\big\{(\s,m):0<\s<\e_0\,m^{1/n}\big\}\,,
  \]
  the following holds:

  \medskip

  \noindent {\it (i)}: for every $\s>0$, $\Psi(\s,\cdot)$ is concave on $(0,\infty)$; it is strictly concave on $(0,\infty)$ in $n\ge3$ and on $((\s/\e_0)^n,\infty)$ if $n=2$;

  \medskip

  \noindent {\it (ii)}: $\Lambda(\s,m)$ is continuous on $\X(\e_0)$ and
  \begin{equation}\label{dai}
  \Big|m^{1/n}\,\Lambda(\s,m)-2\,(n-1)\,\omega_n^{1/n}\Big|\le C\,\frac{\s}{m^{1/n}}\,,\qquad\forall(\s,m)\in \X(\e_0)\,.
  \end{equation}

  \medskip

  \noindent {\it (iii)}: $\Psi(\s,\cdot)$ is differentiable with
  \begin{equation}
    \label{Lambda is pa Psi pa m}
      \frac{\pa\Psi}{\pa m}(\s,m)=\Lambda(\s,m)\qquad\forall (\s,m)\in\X(\e_0)\,.
  \end{equation}
  In particular, for every $\s>0$
  \begin{eqnarray*}
    \mbox{$\Psi(\s,\cdot)$ is strictly increasing on $((\s/\e_0)^n,\infty)$}\,,
    \\
    \mbox{$\Lambda(\s,\cdot)$ is strictly decreasing $((\s/\e_0)^n,\infty)$}\,.
  \end{eqnarray*}

  \medskip

  \noindent {\it (iv)}: for every $m>0$, $\Psi(\cdot,m)$ is increasing on $(0,\e_0\,m^{1/n})$.
\end{theorem}

\begin{proof} We recall for convenience the scaling formulas
\begin{eqnarray}\label{scaling potential x}
\int_{\R^n}f(\rho_tu)&=&\frac1{t}\,\int_{\R^n} f(u)\,,
\\\nonumber
\int_{\R^n}|\nabla(\rho_tu)|^2&=&t^{(2/n)-1}\,\int_{\R^n}|\nabla u|^2\,,
\\
\label{scaling AC x}
\ac_\e(\rho_tu)&=&\e\,t^{(2/n)-1}\,\int_{\R^n}|\nabla u|^2+\frac1{\e\,t}\int_{\R^n}W(u)=\frac{\ac_{\e\,t^{1/n}}(u)}{t^{(n-1)/n}}\,.
\\\nonumber
\Psi(\s,m)&=&m^{(n-1)/n}\,\psi\Big(\frac{\s}{m^{1/n}}\Big)\,,
\end{eqnarray}
where $\rho_tu(x)=u(t^{1/n}\,x)$ for $x\in\R^n$ and $t>0$, and the divide the argument in a few steps.

\medskip

\noindent {\it Step one}: We prove the concavity of $\Psi(\s,\cdot)$. Given $m_2>m_1>0$, $ t \in (0,1)$, $\s>0$, and a minimizing sequence  $\{w_j\}_j$ for $\Psi(\s, t\,m_1+(1- t)\,m_2)$, we set
\[
\alpha_1=\frac{ t\,m_1+(1- t)\,m_2}{m_1}\,,\qquad \alpha_2= \frac{ t\,m_1+(1- t)\,m_2}{m_2}\,,
\]
so that $ t/\a_1+(1- t)/\a_2=1$. Since $\rho_{\a_1}w_j$ and $\rho_{\a_2}w_j$ are competitors for $\Psi(\s,m_1)$ and $\Psi(\s,m_2)$ respectively, by the concavity of $t\mapsto t^{(n-2)/n}$ (strict if $n\ge 3$), we see that
\begin{eqnarray} \label{strict 1}
&& t\, \Psi(\s, m_1)+(1- t)\,\Psi(\s,m_2)
 \leq t \,\ac_\s(\rho_{\a_1}w_j)+ (1- t) \,\ac_\s(\rho_{\a_2}w_j)
\\\nonumber
&&= \frac{ t}{\a_1}\,\Big(\int_{\R^n}\s\,\alpha_1^{2/n}\,|\nabla w_j|^2 + \frac{W(w_j)}{\s}\Big)
+\frac{1- t}{\alpha_2}\,\Big(\int_{\R^n}\s\,\alpha_2^{2/n}\,|\nabla w_j|^2 + \frac{W(w_j)}{\s}\Big)
\\ \label{strict 2}
&&= \ac_\s(w_j)+\Big( t\, \Big(\frac1{\alpha_1}\Big)^{(n-2)/n}+(1- t)\,\Big(\frac1{\alpha_2}\Big)^{(n-2)/n}-1\Big)\,\s\int_{\R^n}|\nabla w_j|^2
\\ \label{strict limits}
&& \leq \ac_\s(w_j)\,.
\end{eqnarray}
Letting $j\to\infty$ we deduce the concavity of $\Psi(\s,\cdot)$ on $(0,\infty)$ (strict, if $n\ge 3$). If $n=2$ and $m_1\ge (\s/\e_0)^n$, then by Theorem \ref{theorem existence solutions} we can replace the minimizing sequence $\{w_j\}_j$ in the above argument with a minimizer $w$ of $\Psi(\s, t\,m_1+(1- t)\,m_2)$. Since $w$ solves the Euler-Lagrange equation \eqref{semilinear PDE}, there cannot be a $t\ne 1$ such that $\rho_tw$ solves \eqref{semilinear PDE} with the same $\s$ and some $ t\in\R$. Thus, $\rho_{\a_i}w$ cannot be a minimizer of $\Psi(\s,m_i)$, and therefore we have a strict inequality in \eqref{strict 1}, and no need to take a limit in \eqref{strict limits} (since $\ac_\s(w)=\Psi(\s, t\,m_1+(1- t)\,m_2)$).

\medskip

\noindent {\it Step two}: By Theorem \ref{theorem existence solutions} and Corollary \ref{corollary optimal energy and lagrange multiplier} for every $m>0$ and $\s<\e_0\,m^{1/n}$ there exists a unique $u_{\s,m}\in\RR_0$ such that $u_{\s,m}$ is a minimizer of $\Psi(\s,m)$ and every other minimizer of $\Psi(\s,m)$ is a translation of $u_{\s,m}$. Moreover, for some $\Lambda(\s,m)>0$ such that
\[
-2\,\s^2\,\Delta u_{\s,m}=\s\,\Lambda(\s,m)\,V'(u_{\s,m})-W'(u_{\s,m})\,,\qquad\mbox{on $\R^n$}\,.
\]
If $u_\e$ denotes as usual the unique minimizer of $\psi(\e)$ in $\RR_0$, then by \eqref{scaling potential x} and \eqref{scaling AC x} we find
\[
u_{\s,m}=\rho_{1/m}\,u_\e\,,\qquad \e=\frac{\s}{m^{1/n}}\,,
\]
and thus
\begin{equation}
  \label{scaling lambda}
  \Lambda(\s,m)=\frac{\l(\e)}{m^{1/n}}\,,\qquad \e=\frac{\s}{m^{1/n}}\,.
\end{equation}
By combining \eqref{scaling lambda} with Corollary \ref{corollary optimal energy and lagrange multiplier} and with \eqref{sharp expansion lambda eps} we thus find that $\Lambda$ is continuous on $\X(\e_0)$, with
\begin{equation}
  \Big|\Lambda(\s,m)-\frac{2\,(n-1)\,\omega_n^{1/n}}{m^{1/n}}\Big|\le C\,\frac{\s}{m^{2/n}}\,.
\end{equation}

\medskip

\noindent {\it Step three}: We prove statement (iii). For $(\s,m)\in \X(\e_0)$, is we set
\[
a(t)=\ac_\s((1+t)\,u_{\s,m})\,,\qquad m(t)=\int_{\R^n}V((1+t)\,u_{\s,m})
\]
then
\[
m'(0)=\int_{\R^n}\Phi(u_{\s,m})^{1/(n-1)}\,\sqrt{W(u_{\s,m})}\,u_{\s,m}>0
\]
and thus there exist $t_*>0$ and an open interval $I$ of $m$ such that $m$ is strictly increasing from $(-t_*,t_*)$ to $I$ with $m(0)=m$. From $\Psi(\s,m(t))\le a(t)$ for every $|t|<t_*$ and from that fact that $a$ is differentiable on $(-t_*,t_*)$ we deduce that, if $m$ is such that $\Psi(\s,\cdot)$ is differentiable at $m$, then
\[
\frac{\pa\Psi}{\pa m}(\s,m)=\frac{a'(0)}{m'(0)}
=\frac{\int_{\R^n}2\,\nabla u_{\s,m}\cdot\nabla u_{\s,m}+W'(u_{\s,m})\,u_{\s,m}}{\int_{\R^n}V'(u_{\s,m})\,u_{\s,m}}=\Lambda(\s,m)\,.
\]
Now, by statement (i), $\Psi(\s,\cdot)$ is differentiable a.e. on $((\s/\e_0)^n,\infty)$, as well as absolutely continuous, while $\Lambda(\s,\cdot)$ is continuous on $((\s/\e_0)^n,\infty)$: by the fundamental theorem of calculus we thus conclude that $(\pa\Psi/\pa m)(\s,\cdot)$ exists for every $m>(\s/\e_0)^n$ and agrees with $\Lambda(\s,m)$.	

\medskip

\noindent {\it Step four}: We prove statement (iv). Recalling that
\begin{equation}
  \label{relation}
  \Psi(\s,m)=m^{(n-1)/n}\,\psi\Big(\frac{\s}{m^{1/n}}\Big)\,,\qquad\forall\s,m>0\,,
\end{equation}
we see that, since $\Psi(\s,\cdot)$ is differentiable on $((\s/\e_0)^n,\infty)$, then $\psi$ is differentiable on $(0,\e_0)$. Since $\psi$ is differentiable on $(0,\e_0$, by \eqref{relation} we see that $\Psi(\cdot,m)$ is differentiable on $(0,\e_0\,m^{1/n})$ for every $m>0$, with
\[
\frac{\pa\Psi}{\pa\s}=m^{(n-2)/n}\,\psi'\Big(\frac{\s}{m^{1/n}}\Big)\,.
\]
Statement (iv) will thus follow by proving that $\psi'>0$ on $(0,\e_0)$. To derive a useful formula for $\psi$ we differentiate \eqref{relation} in $m$ and use \eqref{Lambda is pa Psi pa m} and $\l(\s/m^{1/n})=m^{1/n}\Lambda(\s,m)$ to find that
\begin{eqnarray*}
  \frac{n-1}{n}\,\frac1{m^{1/n}}\,\psi\Big(\frac{\s}{m^{1/n}}\Big)-\frac1n\,\frac{\s}{m^{2/n}}\,\psi'\Big(\frac{\s}{m^{1/n}}\Big)=\frac{\l(\s/m^{1/n})}{m^{1/n}}\,.
\end{eqnarray*}
In particular, by \eqref{lambda formula again},
\begin{eqnarray*}
\e\,\psi'(\e)=(n-1)\,\psi(\e)-n\,\l(\e)
=\e\,\int_{\R^n}|\nabla u_\e|^2-\frac1\e\int_{\R^n}W(u_\e)\,.
\end{eqnarray*}
By \eqref{critical sq resolution of vj}, if we set $\eta_\e(s)=\eta(s-\tau_\e)$ and change variables according to $|x|=R_0+\e\,s$ we find
\begin{equation}
  \label{opti1}
  \e\,\psi'(\e)=\int_{-R_0/\e}^\infty\,\Big\{\Big(\eta_\e'+f_\e'\Big)^2-W\big(\eta_\e+f_\e\big)\Big\}\,(R_0+\e\,s)^{n-1}\,ds\,.
\end{equation}
Multiplying by $u_\e'$ and then integrating on $(r,\infty)$ the Euler-Lagrange equation
\[
-2\,\e^2\,\Big\{u_\e''+(n-1)\,\frac{u_\e'}r\Big\}=\e\,\l(\e)\,V'(u_\e)-W'(u_\e)\,,
\]
we obtain as usual
\[
\e^2\,(u_\e')^2-2\,(n-1)\,\e^2\,\int_r^\infty\,\frac{(u_\e')^2}{\rho}\,d\rho=W(u_\e)-\e\,\l(\e)\,V(u_\e)\,,
\]
for every $r>0$; by the change of variables $r=R_0+\e\,s$ we thus find
\[
(\eta_\e'+f_\e')^2-2\,(n-1)\,\e\,\int_s^\infty\,\frac{(\eta_\e'+f_\e')^2}{R_0+\e\,t}\,dt=W(\eta_\e+f_\e)-\l(\e)\,\e\,V(\eta_\e+f_\e)
\]
for every $s\in(-R_0/\e,\infty)$. We combine this identity into \eqref{opti1} to find
\begin{equation}
  \label{opti2star}
  \e\,\psi'(\e)=\int_{-R_0/\e}^\infty\,\Big\{2\,(n-1)\,\e\,\int_s^\infty\,\frac{(\eta_\e'+f_\e')^2}{R_0+\e\,t}\,dt
  -\l(\e)\,\e\,V(\eta_\e+f_\e)\Big\}\,(R_0+\e\,s)^{n-1}\,ds\,
\end{equation}
We now notice that, by \eqref{eta decay}, \eqref{eta decay first and second derivative}, and \eqref{critical sequence bounds on fj} (that is, by the exponential decay of $\eta$, $\eta'$, $\eta''$ and by $|f_\e(s)|\le C\,\e\,e^{-|s|/C\,\e}$ for $s>-R_0/\e$), we have
\begin{eqnarray*}
  \int_s^\infty\,\frac{(\eta_\e'+f_\e')^2-(\eta_\e')^2}{R_0+\e\,t}\,dt&\ge&
  2\,\int_s^\infty\,\frac{\eta_\e'\,f_\e'}{R_0+\e\,t}\,dt
  \\
  &=&-2\,\frac{\eta_\e'(s)\,f_\e(s)}{R_0+\e\,s}-2\,\int_s^\infty\,f_\e(s)\,\Big(\frac{\eta_\e'}{R_0+\e\,t}\Big)'\,dt
  \\
  &\ge&- C\,\e\,e^{-|s|/C}
\end{eqnarray*}
so that \eqref{opti2star} gives
\begin{equation}
  \label{opti2}
  \e\,\psi'(\e)\ge\int_{-R_0/\e}^\infty\,\Big\{2\,(n-1)\,\e\,\int_s^\infty\,\frac{(\eta_\e')^2\,dt}{R_0+\e\,t}
  -\l(\e)\,\e\,V(\eta_\e+f_\e)\Big\}\,(R_0+\e\,s)^{n-1}\,ds-C\,\e^2\,.
\end{equation}
By \eqref{sharp expansion lambda eps}, \eqref{critical sequence bounds on fj}, $R_0=\om_n^{-1/n}$ and \eqref{opti2}, we have
\begin{equation}
  \label{opti3}
  \psi'(\e)\ge 2\,(n-1)\,\om_n^{1/n}\,\int_{-R_0/\e}^\infty\,\Big\{\int_s^\infty\,(\eta_\e')^2\,dt
  -V(\eta_\e)\Big\}\,(R_0+\e\,s)^{n-1}\,ds-C\,\e\,.
\end{equation}
Since $\int_s^\infty(\eta_\e')^2=\Phi(\eta_\e(s))$ thanks to $\eta_\e'=-\sqrt{W(\eta_\e)}=-\Phi'(\eta_\e)$, by \eqref{opti3} we have
\begin{eqnarray*}
  \psi'(\e)&\ge&
  2\,(n-1)\,\om_n^{1/n}\,\int_\R\big(\Phi(\eta_\e)-V(\eta_\e)\big)\,(R_0+\e\,s)^{n-1}\,ds-C\,\e
  \\
  &\ge&
  2\,(n-1)\,\om_n^{1/n}\,R_0^{n-1}\,\int_\R\big(\Phi(\eta)-V(\eta)\big)\,ds-C\,\e\,.
\end{eqnarray*}
Since $\Phi$ takes values in $(0,1)$, $V=\Phi^{n/(n-1)}<\Phi$ on $(0,1)$, and
\[
\int_\R\big(\Phi(\eta)-V(\eta)\big)\,ds
\]
is a universal constant. In particular, $\psi'(\e)\ge 1/C$ for every $\e<\e_0$.
\end{proof}

\subsection{General criteria for radial symmetry and uniqueness}\label{sec radial symm and uniqueness} In this brief section we exploit two classical results from \cite{gidas1981symmetry} and \cite{peletier1983uniqueness} to deduce a symmetry and uniqueness result for the kind of semilinear PDE arising as the Euler-Lagrange equation of $\Psi(\s,m)$.
	
\begin{theorem}\label{theorem ggn plus sp} Let $n\ge 2$, let $W\in C^{2,1}[0,1]$ satisfy \eqref{W basic} and \eqref{W normalization}, and consider $\ell\in\R$ and $\s>0$.

\medskip

\noindent {\bf (i):} if $u\in C^2(\R^n;[0,1])$ is a non-zero solution to
\begin{equation}
\label{serrin 1}
-2\,\s^2\,\Delta u=\s\,\ell\,V'(u)-W'(u)\qquad\mbox{on $\R^n$}\,,
\end{equation}
with $u(x)\to 0$ as $|x|\to\infty$,  then $0<u<1$ on $\R^n$ and $u\in\RR_0^*$.

\medskip

\noindent {\bf (ii):} there exists a universal constant $\nu_0$ such that, if $0<\s\,\ell<\nu_0$, then, modulo translation, \eqref{serrin 1} has a unique solution among functions $u\in\RR_0^*$, with $u(x)\to 0$ as $|x|\to\infty$ and $0<u<1$ on $\R^n$.
\end{theorem}
	
\begin{remark}
  {\rm Notice that the smallness of $\s\,\ell$ is required only for proving statement (ii).}
\end{remark}

	
	\begin{proof}
{\it Step one}: We prove statement (i). We intend to apply the following particular case of \cite[Theorem 2]{gidas1981symmetry}: {\it if $n\ge2$, $u\in C^2(\R^n;[0,1])$, $u>0$ on $\R^n$, $u(x)\to0$ as $|x|\to\infty$, $-\Delta u+m\,u=g(u)$ on $\R^n$ with $m>0$ and $g\in C^1[0,1]$ with $g(t)={\rm O}(t^{1+\a})$ as $t\to 0^+$ for some $\a>0$, then, up to translations, $u\in\RR_0^*$}.

\medskip

To this end we reformulate \eqref{serrin 1} as
\begin{equation}\label{strong form GNN}
-\Delta u+m\, u = g(u)\qquad\mbox{on $\R^n$}\,,
\end{equation}	
where $m=W''(0)/2\,\s^2>0$ and
\[
g(t)=\frac{\ell\,V'(t)}{2\,\s}+\frac{W''(0)\,t-W'(t)}{2\,\s^2}\,,\qquad t\in[0,1]\,.
\]
As noticed in section \ref{subsection W}, $V\in C^{2,\g}[0,1]$ for some $\g\in(0,1]$, while $W\in C^{2,1}[0,1]$: in particular $g\in C^1[0,1]$. By  $W\in C^{2,1}[0,1]$ with $W'(0)=0$ we have $|W'(t)-W''(0)\,t||\le C\,t^2$, while \eqref{V near the wells} states that $|V'(t)|\le C\,t^{1+\a}$ for $t\in[0,1]$ for some $\a>0$, so that
\begin{equation}\label{GGN g}
|g(t)|\le C(n,W,\ell,\s)\,t^{1+\a}\,,\qquad\forall t\in[0,1]\,.
\end{equation}
To check that $u>0$ on $\R^n$, we notice that, by \eqref{GGN g}, for every $m'\in(0,m)$, we can find $t_0>0$ such that \eqref{strong form GNN} implies that $-\Delta u+ m'\,u\ge 0$ on the open set $\{u<t_0\}$. Since $u\ge 0$ and $u$ is non-zero, we conclude by the strong maximum principle that $u>0$ on $\{u<t_0\}$, and thus, on $\R^n$. We are thus in the position to apply the stated particular case of \cite[Theorem 2]{gidas1981symmetry} and conclude that $u\in\RR_0^*$.

\medskip

We prove that $u<1$ on $\R^n$. Let us set
\begin{equation}
  \label{serrin f}
  f(t)=\frac{\ell\,V'(t)}{2\,\s}-\frac{W'(t)}{2\,\s^2}\,,\qquad t\in[0,1]\,,
\end{equation}
and notice that \eqref{serrin 1} is equivalent to $-\Delta u=f(u)$ on $\R^n$. Since $f$ is a Lipschitz function on $[0,1]$ with $f(1)=0$, we can find $c>0$ such that $f(t)+c\,t$ is increasing on $[0,1]$, and rewrite $-\Delta u=f(u)$ as
\[
-\Delta\,(1-u)+c\,(1-u)=(f(t)+c\,t)\Big|^{t=1}_{t=u}\geq 0\,.
\]
We thus conclude that $v=1-u$ is non-negative on $\R^n$ and such that $-\Delta\,v +c\,v\ge 0$. Since $v$ is non-zero (thanks to $u(x)\to 0$ as $|x|\to\infty$), by the strong maximum principle we conclude that $v>0$ on $\R^n$, i.e. $u<1$ on $\R^n$.
		
\medskip
		
\noindent {\it Step two}: We prove statement (ii). We intend to use \cite[Theorem 2]{peletier1983uniqueness}: {\it if
\begin{enumerate}
 \item[(a)] $f$ locally Lipschitz on $(0,\infty)$;
 \item[(b)] $f(t)/t\to -m$ as $t\to 0^+$ where $m>0$;
 \item[(c)] setting $F(t)=\int_0^t f(s)\,ds$, there exists $\de>0$ such that $F(\de)>0$;
 \item[(d)] setting $\beta=\inf\{t>0:F(t)>0\}$ (so that by (b) and (c), $\beta\in(0,\de)$), the function $t\mapsto f(t)/(t-\beta)$ is decreasing on $(\beta,\infty)\cap\{f>0\}$\,;
\end{enumerate}
then there is at most one $u\in C^2(\R^n)\cap\RR_0$, with $u>0$ on $\R^n$ and $u(x)\to 0$ as $|x|\to\infty$, solving $-\Delta u=f(u)$ on $\R^n$}.

\medskip

Since, by statement (i), solutions to \eqref{serrin 1} satisfy $0<u<1$ on $\R^n$, in checking that $f$ as in \eqref{serrin f} satisfies the above assumptions it is only the behavior of $f$ on $(0,1)$ (and not on $(0,\infty)$) that matters. Evidently (a) holds, since $f\in C^{1,\a}[0,1]$ for some $\a\in(0,1)$. Assumption (b) holds with $m=W''(0)/2\,\s^2$. Property (c) holds (with $\de\in(0,1)$) since
\[
F(t)=\int_0^t\,f(s)\,ds=\frac{\ell\,V(t)}{2\,\s}-\frac{W(t)}{2\,\s^2}\,,\qquad t\in[0,1]\,,
\]
and $F(1)=(\ell\,V(1)/2\,\s)=\ell/2\,\s>0$ by $\ell>0$ and $W(1)=0$. We finally prove (d). Notice that, clearly, $\beta\in(0,1)$ and, by continuity of $F$, $F(\beta)=0$, so that, taking \eqref{W basic u} and \eqref{W near the wells} into account, and using $\s\,\ell<\nu_0$ and $V(1)=1$,
\begin{equation}
\label{ps 1}
\frac{\min\{\beta^2,(1-\b)^2\}}C\le W(\b)=\s\,\ell\,V(\b)\le\nu_0\,.
\end{equation}
If $\nu_0<1$, then by \eqref{W near the wells} and \eqref{V near the wells} we find
\begin{equation}
\label{ps 2}
2\,\s^2\,F(t)=\s\,\ell\,V(t)-W(t)\le V(t)-W(t)\le C\,t^{2n/(n-1)}-\frac{t^2}C<0\,,\qquad\forall t\in(0,\de_0)\,.
\end{equation}
By \eqref{ps 2} it must be $\beta\ge\de_0$. Hence, by \eqref{ps 1}, if $\nu_0$ is sufficiently small, then $(1-\beta)^2\le C\,\nu_0$. Up to further decrease the value of $\nu_0$, we can finally entail that $(\beta,1)\subset(1-\de_0,1)$, with $\de_0$ as in section \ref{subsection W}.

\medskip

We are now going to check property (d) by showing that
\begin{equation}
  \label{ps 3}
  f'(t)\,(t-\beta)\le f(t)\qquad\forall t\in(\beta,1)\,,
\end{equation}
(recall that $0<u<1$ on $\R^n$, so we can use a version of \cite[Theorem 2]{peletier1983uniqueness} localized to $(0,1)$). Using the explicit formula for $f$, \eqref{ps 3} is equivalent to
\begin{equation}
  \label{ps 4}
  \s\,\ell\, V''(t)\,(t-\beta)\le \s\,\ell\,V'(t)-W'(t)+W''(t)\,(t-\b)\,,\qquad\forall t\in(\beta,1)\,.
\end{equation}
By \eqref{W near the wells}, we have $W''(t)(t-\b)>0$ on $(\beta,1)\subset(1-\de_0,1)$, and since $V'\ge 0$ on $[0,1]$, \eqref{ps 4} is implied by checking that, for every $t\in(\beta,1)$,
\begin{eqnarray*}
-W'(t)&\ge&\s\,\ell\, V''(t)
\\
&=& \s\,\ell\,\Big\{\frac{n}{(n-1)^2}\,\frac{W(t)}{\Big(\int_0^t\sqrt{W}\Big)^{(n-2)/(n-1)}}
+\frac{n}{n-1}\,\Big(\int_0^t\sqrt{W}\Big)^{1/(n-1)}\,\frac{W'(t)}{2\,\sqrt{W(t)}}\Big\}\,.
\end{eqnarray*}
In turn, since $W'<0$ on $(1-\de_0,1)$ and $\s\,\ell<\nu_0< 1$, it is actually enough to check that
		\begin{eqnarray}\label{fine serrin}
		-W'(t)\ge \frac{n}{(n-1)^2}\,\frac{W(t)}{\Big(\int_0^t\sqrt{W}\Big)^{(n-2)/(n-1)}}\,,\qquad\forall t\in(1-\de_0,1)\,.
		\end{eqnarray}
But, up to further decreasing the value of $\de_0$, this is obvious: indeed \eqref{W near the wells} gives $-W'(t)\ge (1-t)/C$ and $W(t)\le C\,(1-t)^2$ for every $t\in(1-\de_0,1)$.
	\end{proof}

\subsection{Proof of Theorem \ref{theorem main}}\label{section final proof} Theorem \ref{theorem existence solutions}, Corollary \ref{corollary optimal energy and lagrange multiplier}, Theorem \ref{theorem basics of Psi} and a scaling argument show the validity of statements (i) and (ii), while statement (iii) follows similarly by scaling and by Theorem \ref{thm main 2}. To prove the Alexandrov-type theorem, that is, statement (iv)\footnote{Notice that we are using $\ell$ in \eqref{critical point proof} rather than $\lambda$ (as done in \eqref{critical point}) to denote the Lagrange multiplier of $u$. This is meant to avoid confusion with the function $\lambda(\e)=(\pa\Psi/\pa m)(\e,1)$ appearing in the argument.} we consider $u\in C^2(\R^n;[0,1])$, with $u(x)\to 0$ as $|x|\to\infty$, and solving
  \begin{equation}\label{critical point proof}
    -2\,\s^2\,\Delta u=\s\,\ell\,V'(u)-W'(u)\qquad\mbox{on $\R^n$}\,,
  \end{equation}
for some $\s$ and $\ell$ with $0<\s\,\ell<\nu_0$. By Theorem \ref{theorem ggn plus sp}--(i), $u\in\RR_0^*$, and by Theorem \ref{theorem ggn plus sp}, provided $\nu_0$ is small enough, we know that there is at most one radial solution to \eqref{critical point proof}. Since we know that $u_{\s,m}$ is a radial solution of \eqref{critical point proof} with $\ell=\Lambda(\s,m)$, we are left to prove that for every $\ell\in(0,\nu_0/\s)$ there exists a unique $m\in((\s/\e_0)^n,\infty)$ such that $\Lambda(\s,m)=\ell$.

\medskip

To this end, we first notice that, by \eqref{sharp expansion lambda eps} and by scaling, for every $\s>0$ we have
\[
\Lambda(\s,m)=\frac1{m^{1/n}}\,\l\big(\s/m^{1/n})\to 0^+\qquad\mbox{as $m\to+\infty$}\,.
\]
In particular, since, by Theorem \ref{theorem basics of Psi}, $\Lambda(\s,\cdot)$ is continuous and strictly decreasing on $((\s/\e_0)^n,\infty)$, we have that
\[
\Big\{\Lambda(\s,m):m>\Big(\frac{\s}{\e_0}\Big)^n\Big\}=\Big(0,\Lambda\big(\s,(\s/\e_0)^n\big)\Big)\,.
\]
Now, setting $m=(\s/\e_0)^n$ in \eqref{dai}, that is, in
\[
\Big|m^{1/n}\,\Lambda(\s,m)-2\,(n-1)\,\omega_n^{1/n}\Big|\le C\,\frac{\s}{m^{1/n}}\,,
\]
we find that
\[
\Big|\s\,\Lambda\big(\s,(\s/\e_0)^n\big)-2\,(n-1)\,\omega_n^{1/n}\,\e_0\Big|\le C\,\e_0^2\,,
\]
which implies
\[
\Lambda\big(\s,(\s/\e_0)^n\big)\ge \frac{(n-1)\,\omega_n^{1/n}\,\e_0}\s\,,\qquad\forall\s>0\,,
\]
provided $\e_0$ is small enough. Up to further decrease the value of $\nu_0$ so to have $\nu_0\le (n-1)\,\omega_n^{1/n}\,\e_0$, we have proved that
\[
(0,\nu_0/\s)\subset \Big\{\Lambda(\s,m):m>\Big(\frac{\s}{\e_0}\Big)^n\Big\}\,,
\]
and that for each $\ell\in(0,\nu_0/\s)$ there is a unique $m>(\s/\e_0)^n$ such that $\ell=\Lambda(\s,m)$, as claimed. This completes the proof of Theorem \ref{theorem main}.

\appendix

\section{Frequently used auxiliary facts}\label{appendix auxiliary}

\subsection{Scaling identities} If $u\in H^1(\R^n;[0,\infty))$, $t>0$, we set
\[
\rho_t\,u(x)=u(t^{1/n}\,x)\,,\qquad x\in\R^n\,,
\]
and notice that
\begin{eqnarray}\label{scaling potential}
\int_{\R^n}f(\rho_tu)&=&\frac1{t}\,\int_{\R^n} f(u)\,,\qquad\int_{\R^n}|\nabla(\rho_tu)|^2=t^{(2/n)-1}\,\int_{\R^n}|\nabla u|^2\,,
\\
\label{scaling AC}
\ac_\e(\rho_tu)&=&\e\,t^{(2/n)-1}\,\int_{\R^n}|\nabla u|^2+\frac1{\e\,t}\int_{\R^n}W(u)=\frac{\ac_{\e\,t^{1/n}}(u)}{t^{(n-1)/n}}\,.
\end{eqnarray}
whenever $f:\R\to\R$ is continuous.

\subsection{Concentration-compactness principle}\label{subsection ccc} Denoting by $B_r(x)$ the ball of center $x$ and radius $r$ in $\R^n$, and setting $B_r=B_r(0)$ when $x=0$, we provide a reference statement for Lions' concentration-compactness criterion, which is repeatedly used in our arguments:  {\it if $\{\mu_j\}_j$ is a sequence of probability measures in $\R^n$, then, up to extract subsequences and compose each $\mu_j$ with a translation, one the following mutually excluding possibilities holds:

\medskip
	
\noindent {\bf Compactness case:} for every $\tau>0$ there exists $R>0$ such that
\[
\inf_j\,\mu_j(B_R)\ge 1-\tau\,;
\]
	
\medskip
	
\noindent {\bf Vanishing case:} for every $R>0$,
\[
\lim_{j\to\infty}\,\sup_{x\in\R^n}\mu_j(B_R(x))=0\,;
\]
	
\medskip
	
\noindent {\bf Dichotomy case:} there exists $\a\in(0,1)$ such that for every $\tau>0$ one can find $S>0$ with $S_j\to\infty$ such that
\[
\sup_j\,|\a-\mu_j(B_S)|<\tau\,,\qquad \sup_j\,\big|(1-\a)-\mu_j(\R^n\setminus B_{S_j})\big|<\tau\,.
\]
}

\noindent Notice that the formulation of the dichotomy case used here is a bit more descriptive than the original one presented in \cite[Lemma I]{lions}. Its validity is inferred by a quick inspection of the proof presented in the cited reference.

\subsection{Estimates for $W$, $\Phi$ and $V$}\label{subsection W} Throughout the paper we work with a double well potential $W\in C^{2,1}[0,1]$ satisfying \eqref{W basic} and \eqref{W normalization}, that is,
\begin{eqnarray}
  \label{W basic u}
	&&\mbox{$W(0)=W(1)=0$}\,,\quad\mbox{$W>0$ on $(0,1)$}\,,\quad\mbox{$W''(0),W''(1)>0$}\,,
\\
	\label{W normalization u}
	&&\int_0^1\sqrt{W}=1\,.
\end{eqnarray}	
Frequently used properties of $W$ are the validity, for a universal constant $C$, of the expansion
\begin{equation}\label{W second order taylor}
\Big|W(b)-W(a)-W'(a)(b-a)-W''(a)\frac{(b-a)^2}2\Big|\le C\,|b-a|^3\,,\qquad\forall a,b\in[0,1]\,,
\end{equation}
and the existence of a universal constant $\de_0<1/2$ such that
\begin{equation}
  \label{W near the wells}
  \begin{split}
  &\frac{1}C\le \frac{W}{t^2}\,,\frac{W'}t\,,W''\leq C\quad\hspace{1.5cm}\mbox{on $(0,\de_0]$}\,,
  \\
  &\frac{1}C\le\frac{W}{(1-t)^2}\,, \frac{-W'}{1-t}\,,W''\leq C\quad\hspace{0.4cm}\mbox{on $[1-\de_0,1)$}\,.
  \end{split}
  \end{equation}
  We can use \eqref{W near the wells} to quantify the behaviors near the wells of $\Phi$ and, crucially, of $V$. We first notice that, by \eqref{W basic u}, $\Phi\in C^3_{{\rm loc}}(0,1)$, with
  \[
  \Phi'=\sqrt{W}\,,\quad\Phi''=\frac{W'}{2\,\sqrt{W}}\,,\quad\Phi'''=\frac{W''}{2\,\sqrt{W}}-\frac{(W')^2}{4\,W^{3/2}}\,,\qquad\mbox{on $(0,1)$}\,.
  \]
  By \eqref{W near the wells} and \eqref{W normalization u} we thus see that $\Phi$ satisfies
  \begin{equation}
  \label{Phi near the wells}
  \begin{split}
  &\frac{1}C\le \frac{\Phi}{t^2}\,,\frac{\Phi'}t\,,\Phi''\leq C\,,\qquad\mbox{on $(0,\de_0]$}\,,
  \\
  &\frac{1}C\le \frac{1-\Phi}{(1-t)^2}\,,\frac{\Phi'}{1-t}\,,-\Phi''\le C\,,\qquad\mbox{on $[1-\de_0,1)$}\,,
  \end{split}
  \end{equation}
  from which we easily deduce
  \begin{equation}
    \label{Phi quadratic from below}
    |\Phi(b)-\Phi(a)|\ge\frac{(b-a)^2}{C}\,,\qquad\forall a,b\in[0,1]\,.
  \end{equation}
  Moreover, by exploiting \eqref{Phi near the wells} and setting for brevity $a=W''(0)$, we see that as $t\to 0^+$
  \begin{eqnarray*}
  \Phi'''&=&\frac{2\,W''\,W-(W')^2}{4\,W^{3/2}}
  =\frac{2\,(a+{\rm O}(t))\,(a\,(t^2/2)+{\rm O}(t^3))-(a\,t+{\rm O}(t^2))^2}{4\,(a\,(t^2/2)+{\rm O}(t^3))^{3/2}}
  \\
  &=&\frac{{\rm O}(t^3)}{4\,a^{3/2}\,t^3+{\rm o}(t^3)}\,,
  \end{eqnarray*}
  and a similar computation holds for $t\to 1^-$, so that
  \begin{equation}
    \label{Phi third}
    |\Phi'''|\le C\,,\qquad\mbox{on $(0,\de_0)\cup(1-\de_0,1)$}\,.
  \end{equation}
  By \eqref{Phi near the wells} and \eqref{Phi third} we see that $\Phi\in C^{2,1}[0,1]$ with a universal estimate on its $C^{2,1}[0,1]$-norm: in particular,
  \begin{equation}\label{Phi second order taylor}
\Big|\Phi(b)-\Phi(a)-\Phi'(a)(b-a)-\Phi''(a)\frac{(b-a)^2}2\Big|\le C\,|b-a|^3\,,\qquad\forall a,b\in(0,1)\,.
\end{equation}
Since $V=\Phi^{1+\a}$ for $\a=1/(n-1)\in(0,1]$ (recall that $n\ge 2$) and $\Phi(t)=0$ if and only if $t=0$, we easily see that $V\in C^3_{{\rm loc}}(0,1)$, with
\begin{eqnarray*}
&&V'=(1+\a)\,\Phi^\a\,\Phi'\,,\qquad V''=(1+\a)\Big\{\a\,\frac{(\Phi')^2}{\Phi^{1-\a}}+\Phi^\a\,\Phi''\Big\}\,,
\\
&&|V'''|\le C(\a)\,\Big\{\frac{(\Phi')^3}{\Phi^{2-\a}}+\frac{\Phi'\,|\Phi''|}{\Phi^{1-\a}}+\Phi^\a\,|\Phi'''|\Big\}\,.
\end{eqnarray*}
By \eqref{Phi second order taylor}, and keeping track of the sign of $\Phi''$ and of the fact that negative powers of $\Phi(t)$ are large only near $t=0$, but are bounded near $t=1$, we find that
  \begin{equation}
  \label{V near the wells}
  \begin{split}
    &\frac1C\le \frac{V}{t^{2+2\a}}\,,\frac{V'}{t^{1+2\a}}\,,\frac{V''}{t^{2\a}}\leq C\,,\hspace{0.7cm}\quad |V'''|\le\frac{C}{t^{1-2\a}}\qquad\hspace{0.2cm}\mbox{on $(0,\de_0]$}\,,
    \\
    &\frac1C\le  \frac{1-V}{(1-t)^2}\,,\frac{V'}{1-t}\le C\,,\quad |V''|\,,|V'''|\le C\,,\qquad\mbox{on $[1-\de_0,1)$}\,.
  \end{split}
  \end{equation}
  In particular, $V\in C^{2,\g(n)}[0,1]$, $\g(n)=\min\{1,2/(n-1)\}\in(0,1]$, with second order Taylor expansions of the form
\begin{equation}\label{V second order taylor}
\Big|V(b)-V(a)-V'(a)(b-a)-V''(a)\frac{(b-a)^2}2\Big|\le C\,|b-a|^{2+\g(n)}\,,\qquad\forall a,b\in(0,1)\,.
\end{equation}
We finally notice that we can find a universal constant $C$ such that
\begin{equation}
  \label{W controlla V fin quasi ad uno}
  \frac{t^2}C\le W(t)\,,\qquad V(t)\le C\,t^2\,,\qquad V(t)\le C\,W(t)\,,\qquad\forall t\in(0,1-\de_0)\,,
\end{equation}
(as it is easily deduced from the bounds on $W$ and $V$ in \eqref{W near the wells} and \eqref{V near the wells} and from the fact that $W>0$ on $(0,1)$); and that we can also find $C$ so that
\begin{equation}
  \label{V limitata dal basso via da zero}
  V(t)\ge\frac1C\,,\qquad\forall t\in(\de_0,1)\,.
\end{equation}

\subsection{Estimates for the optimal transition profile $\eta$}\label{subsection eta} A crucial object in the analysis of the Allen-Cahn energy is of course the optimal transition profile $\eta$, defined by the first order ODE
\begin{equation}\label{eta ODE}
    \left\{
    \begin{split}
      &\eta'=-\sqrt{W(\eta)}\qquad\mbox{on $\R$}\,,
      \\
      &\eta(0)=\frac12\,,
    \end{split}
    \right .
  \end{equation}
  which can be seen to satisfy (see, e.g. \cite{leoni2016second}) $\eta\in C^{2,1}(\R)$, $\eta'<0$ on $\R$ (and $-C\le\eta'\le -1/C$ for $|s|\le 1$), $\eta(-\infty)=1$, and $\eta(+\infty)=0$, with the exponential decay properties
  \begin{equation}
    \label{eta decay}
    1-\eta(s)\le C\,e^{s/C}\qquad\forall s<0\,,\qquad \eta(s)\le C\,e^{-s/C}\qquad\forall s>0\,,
  \end{equation}
  for a universal constant $C$. Similarly, by combining \eqref{eta decay} with \eqref{eta ODE}, with the second order ODE satisfied by $\eta$, namely,
  \begin{equation}
    \label{eta ODE second}
    2\,\eta''=W'(\eta)\qquad\mbox{on $\R$}\,,
  \end{equation}
  and with \eqref{W near the wells} we see that also the first and second derivatives of $\eta$ decay exponentially
  \begin{equation}
    \label{eta decay first and second derivative}
    |\eta'(s)|\,,|\eta''(s)|\le C\,e^{-|s|/C}\qquad\forall s\in\R\,.
  \end{equation}
  Combining again \eqref{eta decay} and \eqref{W near the wells} we also see that
  \[
  s\in\R\mapsto\,1_{(-\infty,0)}(s) - V(\eta(s-\tau))
  \]
  belongs to $L^1(\R)$ for every $\tau\in\R$, with
  \[
  \tau\in\R\mapsto   \int_{-\infty}^\infty \Big(1_{(-\infty,0)}(s) - V(\eta(s-\tau))\Big)\, ds
  \]
  increasing in $\tau$ and converging to $\mp\infty$ as $\tau\to\pm\infty$. In particular, there is a unique universal constant $\tau_0$ such that
  \begin{equation}
  \label{definition of tau0}
  \int_{-\infty}^\infty \Big(1_{(-\infty,0)}(s) - V(\eta(s-\tau_0))\Big)\, ds=0\,.
  \end{equation}
  The constant $\tau_0$ appears in the computation of the first order expansion of $\psi(\e)$ as $\e\to 0^+$ and can be characterized, equivalently, to be
  \begin{equation}
    \label{other def of tau0}
    \tau_0=\int_\R \eta'\,V'(\eta)\,s\,ds\,.
  \end{equation}
  Indeed, \eqref{definition of tau0} gives
  \begin{eqnarray*}
  0&=&\int_{-\infty}^\infty \Big(1_{(-\infty,0)}(s) - V(\eta(s-\tau_0))\Big)\, ds
  \\
  &=&\int_{-\infty}^0 \Big(1- V(\eta(s-\tau_0))\Big)\, ds-\int_0^\infty V(\eta(s-\tau_0))\, ds
  \\
  &=&
  -\int_{-\infty}^0 \,ds\,\int_{-\infty}^{s-\tau_0}\eta'(t)\,V'(\eta(t))\,dt+\int_0^\infty ds\int_{s-\tau_0}^\infty \eta'(t)\,V'(\eta(t))\,dt\,.
  \end{eqnarray*}
  Both integrands are non-negative, therefore by Fubini's theorem
  \begin{eqnarray*}
  0&=&
  -\int_{-\infty}^{-\tau_0} \,dt\,\int_{t+\tau_0}^0\,\eta'(t)\,V'(\eta(t))\,ds-\int_{-\tau_0}^\infty dt\int_0^{t+\tau_0} \eta'(t)\,V'(\eta(t))\,ds
  \\
  &=&
  \int_{-\infty}^{-\tau_0} \,(t+\tau_0)\,\eta'(t)\,V'(\eta(t))\,dt+\int_{-\tau_0}^\infty (t+\tau_0) \eta'(t)\,V'(\eta(t))\,dt
  \end{eqnarray*}
  that is
  \[
  \int_\R\,\eta'\,V'(\eta)\,t\,dt=-\tau_0\,\int_\R\eta'V'(\eta)=V(1)\,\tau_0=\tau_0\,,
  \]
  as claimed.

	\bibliographystyle{alpha}
	\bibliography{references_mod}
	
\end{document}